\documentclass[11pt]{article}
 
\usepackage[margin=1in]{geometry} 
\usepackage{amsmath,amsthm,amssymb,bm}
\usepackage{mathtools}
\usepackage{tikz}
\usepackage{tikz-cd}
\usepackage{tikz-3dplot}
\usetikzlibrary{arrows}
\usetikzlibrary{decorations.markings}
\usepackage{ upgreek }
\usepackage{color}
\usepackage{fancyhdr}
\usepackage[title]{appendix}
\usepackage{afterpage}
\usepackage{tcolorbox}
\usepackage{hyperref}
\usepackage{enumerate}
\usepackage[shortlabels]{enumitem}
\usepackage{setspace}
\usepackage{caption}
\usepackage{minitoc}
\usepackage{ifthen}
\usepackage{subcaption}
\usepackage{wrapfig}
\usepackage{pgfplots}
\usepackage[latin1]{inputenc}
\usepackage{listings}
\usepackage{float}
\usepackage[normalem]{ulem}
\usepackage{cancel}
\usepackage{bm}

\newcommand{\N}{\mathbb{N}}
\newcommand{\cW}{\mathcal{W}}
\newcommand{\cK}{\mathcal{K}}
\newcommand{\cB}{\mathcal{B}}
\newcommand{\cS}{\mathcal{S}}

\newcommand{\C}{\mathbb{C}}
\newcommand{\R}{\mathbb{R}}
\newcommand{\E}{\mathbb{E}}
\newcommand\Tstrut{\rule{0pt}{2.6ex}}  
\newcommand\Bstrut{\rule[-0.9ex]{0pt}{0pt}}

\newcommand{\ru}{\mathrm{u}}
\newcommand{\rs}{\mathrm{s}}
\newcommand{\rme}{\mathrm{e}}
\newcommand{\rmi}{\mathrm{i}}
\newcommand{\GL}{\mathrm{GL}}
\newcommand{\nmu}{\hat{\mu}}
\newcommand{\nnu}{\hat{\nu}}
\newcommand{\changes}[1]{{\color{black}   {#1}}}

\newtheorem{thm}{Theorem}[section]
\newtheorem{theorem}[thm]{Theorem}
\newtheorem{corollary}[thm]{Corollary}
\newtheorem{lemma}[thm]{Lemma}  
\newtheorem{proposition}[thm]{Proposition}

\newtheorem{definition}[thm]{Definition}

\theoremstyle{remark}
\newtheorem{remark}[thm]{Remark}

\newtheorem*{lemma*}{Lemma}

\numberwithin{equation}{section}

\begin{document}

\title{Resonant vector bundles, conjugate points, and the stability of pulse solutions to the {S}wift-{H}ohenberg equation using validated numerics: Part I}

\author{Margaret Beck\footnote{Department of Mathematics and Statistics, Boston University; Boston, MA, USA; mabeck@bu.edu}, Jonathan Jaquette\footnote{Department of Mathematical Sciences, New Jersey Institute of Technology; Newark, NJ, USA; jcj@njit.edu; ORCID: 0000-0001-8380-3148}, Hannah Pieper\footnote{hpieper@bu.edu} 
}

\date{\today}
\maketitle

\begin{abstract} 
In this paper, we develop new theory connected with resonant vector bundles that will allow for the use of validated numerics to rigorously determine the stability of pulse solutions in the context of the Swift-Hohenberg equation. For many PDEs, the stability of stationary solutions is determined by the absence of point spectra in the open right half of the complex plane. Recently, theoretical developments have allowed one to use objects called conjugate points to detect such unstable eigenvalues for certain linearized operators. Moreover, in certain cases these conjugate points can themselves be detected using validated numerics. The aim of this work is to extend this framework to contexts where the vector bundles, which control the existence of conjugate points, have certain resonances. Such resonances can prevent the use of standard (though involved) techniques in computer assisted proofs, and in this paper we provide a method to overcome this obstacle. Due to its length, the analysis has been divided into two parts: Part I in the present work, and Part II in \cite{BeckJaquettePieperStorm25}. 
\end{abstract}


\tableofcontents

\setlength{\headheight}{13.59999pt}.
\addtolength{\topmargin}{-1.59999pt}.


\section{Introduction}\label{sec:Intro}
 
In this paper, we are interested in developing new theory to allow for the broadening of the class of equations for which the stability of stationary solutions can be determined using a framework involving conjugate points and validated numerics. To achieve this, we focus on the Swift-Hohenberg equation, which is given by
\begin{equation}\label{E:SH} 
u_t = -(1+\partial_x^2)^2u + f(u), \qquad f(u) =\nnu u^2 - u^3 - \nmu u,
\end{equation}
where $\nmu$ and $\nnu$ are positive real parameters. There are several reasons why we focus on this equation. First, it is a fundamental model of pattern formation and thus relevant for a wide variety of applications. In addition, the conjugate point framework for detecting stability, which we will introduce in more detail below, has recently been extended to this equation \cite{BeckJaquettePieper24}. 

A key innovation we develop is a high-order parameterization of the unstable bundle over a stable manifold arising from the spatial dynamics. This treatment is somewhat intricate: a local stable manifold is only positively invariant, and consequently its unstable bundle is not unique. In the absence of eigenvalue resonances, a parameterization method to compute these bundles was developed in \cite{vandenBergJames16}. 

However, the Swift-Hohenberg equation exhibits a resonance in its vector bundles, which prevents the use of this technique. In fact, the resonance we encounter in the vector bundles is not a peculiarity of the Swift-Hohenberg equation, but rather a universal phenomenon that occurs in Hamiltonian systems. Since the first-order nonlinear ODE that governs the existence of stationary solutions to \eqref{E:SH} is Hamiltonian, it is thus not surprising that we encounter resonances. Because of this universality, our development to treat such resonances is a major contribution of the paper.

To begin, suppose that $\varphi$ is a steady-state pulse solution to \eqref{E:SH}, meaning that $u(x, t) = \varphi(x)$ satisfies \eqref{E:SH} with $\lim_{|x| \to \infty} \varphi(x) = 0$. Note this means that, if we write $(u_1, u_2, u_2, u_4) = (u, u_x, u_{xx}, u_{xxx})$, then $u = \varphi$ is a solution of
\begin{align}
u_1' &= u_2 \nonumber \\
u_2' &= u_3 \nonumber \\
u_3' &= u_4 \label{E:SH-vf} \\
u_4' &= -2u_3 -u_1 + f(u_1) \nonumber
\end{align}
that is asymptotic to the origin as $|x| \to \infty$. To investigate stability, we set $u(x, t) = \varphi(x) + v(x, t)$ and wish to determine the dynamics of the perturbation $v(x,t)$, which satisfies
\[
v_t = \mathcal{L} v + \mathcal{N}(v),
\]
where
\begin{equation}\label{E:defL}
\mathcal L  = -\partial_x^4 - 2  \partial_x^2 - 1+ f'(\varphi(x))
\end{equation}
and $\mathcal{N}(v)$ is nonlinear in $v$. It is known that, if $f'(0) < 0$, then the essential spectrum of $\mathcal{L}$ is in the open left half of the complex plane, and hence stability of $\varphi$ is determined by the point spectrum of $\mathcal{L}$ \cite{KapitulaPromislow13}. In general, rigorously determining the point spectrum of a differential operator is difficult. However, the eigenvalue problem associated with $\mathcal{L}$, $\lambda v = \mathcal{L} v$, has a useful structure when written as a first-order system. In particular, if we define  
\begin{align}\label{E:defq}
	Q &= \begin{pmatrix} v \\ v_{xx} \\ v_{xxx} + 2 v_{x} \\ v_x \end{pmatrix}=
	S \begin{pmatrix} v \\ v_x \\ v_{xx} \\ v_{xxx} \end{pmatrix}
	,
	&
	S &= \begin{pmatrix} 1 & 0 & 0 & 0 \\ 0 & 0 & 1 & 0 \\ 0 & 2 & 0 & 1 \\ 0 & 1 & 0 & 0 \end{pmatrix}
	\end{align}
then we obtain
\begin{equation}\label{E:eval}
Q' = B(x; \lambda) Q, \qquad B(x; \lambda) = J C(x; \lambda),
\end{equation}
 where we define the symplectic  matrix
 \begin{equation}\label{E:defJ}
 J = \begin{pmatrix} 0 & I_2 \\ -I_2 & 0 \end{pmatrix} 
 \end{equation}
 and
 \begin{equation*}
 C(x; \lambda) = \begin{pmatrix} \lambda + 1 - f'(\varphi(x)) & 0 & 0 & 0 \\0 & -1 & 0 & 0 \\ 0 & 0 & 0 & 1 \\ 0 & 0 & 1 & -2 \end{pmatrix}, \qquad B(x; \lambda) = \begin{pmatrix}0 & 0 & 0 & 1 \\   0 & 0 & 1 & -2 \\ - \lambda - 1 + f'(\varphi(x)) & 0 & 0 & 0 \\0 & 1 & 0 & 0 \end{pmatrix}.
\end{equation*}
We note in particular that $C(x; \lambda)$ is symmetric; this is due to the fact that $\mathcal{L}$ is self-adjoint, and hence $\lambda \in \R$. \changes{We also note that coordinates similar to those in \eqref{E:defq} have been used to illuminate the symplectic structure in other fourth-order systems; see, for example, \cite{ChardardDiasBridges09, ChardardBridgesDias09b}.}

In general, eigenvalues of a differential operator $\mathcal{L}$ are elements $\lambda \in \C$ for which solutions of $\lambda v = \mathcal{L} v$ exist in an appropriate function space. There is some freedom in choosing which space to work in, but essentially any reasonable choice would require that $v(x; \lambda) \to 0$ as $|x| \to \infty$\footnote{Some function spaces might only require the weaker condition that $v(x; \lambda)$ remain bounded as $|x| \to \infty$. However, for the system considered here, the asymptotic matrix $\lim_{|x| \to \infty} B(x; \lambda)$ is hyperbolic when $\lambda \geq 0$, and so in order to remain bounded at infinity, solutions must in fact decay to zero there.}. For this reason, the existence of eigenvalues can be understood by studying the so-called stable and unstable subspaces of \eqref{E:eval}. If we let $\Phi(x,y; \lambda)$ be the evolution associated with \eqref{E:eval}, then these are defined by
\begin{align}
\E^{\mathrm{u}}_-(x; \lambda) &= \{Q_0 \in \R^{2n}: |\Phi(x, y; \lambda)Q_0| \to 0 \mbox{ as } y \to -\infty\} \nonumber \\
\E^{\mathrm{s}}_+(x; \lambda) &= \{Q_0 \in \R^{2n}: |\Phi(y, x; \lambda)Q_0| \to 0 \mbox{ as } y \to +\infty\}. \label{E:stable-unstable}
\end{align} 
In other words, $\E^{\mathrm{u}}_-(x; \lambda)$ is the set of all initial conditions at time $x$ that lead to solutions that decay to zero in backwards time, and hence are asymptotic to the unstable eigenspace of 
\begin{equation}\label{E:asym}
B_\infty(\lambda) = \lim_{|x| \to \infty} B(x; \lambda).
\end{equation}
Similarly, $\E^{\mathrm{s}}_+(x; \lambda)$ is the set of all initial conditions at time $x$ that lead to solutions that decay to zero in forwards time, and hence are asymptotic to the stable eigenspace of $B_\infty(\lambda)$ at $+\infty$. We note that one can similarly define $\E^{\mathrm{s}}_-(x; \lambda)$ and $\E^{\mathrm{u}}_+(x; \lambda)$, which consist of solutions that grow exponentially in backwards and forwards time, respectively, although these are not unique. 

The requirement that an eigenfunction decay to zero as $|x| \to \pm \infty$ implies that a solution $Q(x; \lambda)$ to \eqref{E:eval} corresponds to an eigenfunction of $\mathcal{L}$ with eigenvalue $\lambda$ if and only if $Q(x; \lambda) \in \E^{\mathrm{u}}_-(x; \lambda) \cap \E^{\mathrm{s}}_+(x; \lambda)$ for any, and hence all, values of $x$. This is, for example, the idea behind the Evans function, which has been used to study stability in a variety of contexts \cite{Sandstede02}. What is perhaps surprising is that, for certain  fixed choices of two-dimensional subspaces $\ell_* \subseteq \mathbb{R}^4$, called reference planes, it is also the case that the number of intersections of $\E^{\mathrm{u}}_-(x; 0)$ with $\ell_*$, counted with sign and multiplicity as $x$ ranges in $\R$, is equal to the number of positive eigenvalues of $\mathcal{L}$. 

When $\lambda = 0$, \eqref{E:eval} is the variational equation associated with \eqref{E:SH-vf} and the solution $\varphi$, albeit written in terms of the coordinates given in \eqref{E:defq} instead of the more commonly used coordinates that lead to \eqref{E:SH-vf}. Nevertheless, because of this connection, when $\lambda = 0$ the unstable subspace $\E^\ru_-(x; 0)$ can also be viewed as the unstable vector bundle over the local unstable manifold of the origin of \eqref{E:SH-vf}. Similarly, $\E^\rs_+(x; 0)$ can be viewed as the stable bundle over the local stable manifold of the origin of \eqref{E:SH-vf}. This vector-bundle perspective will be utilized below.

The values of $x \in \R$ for which $\E^{\mathrm{u}}_-(x; 0) \cap \ell_* \neq \{0\}$ are called conjugate points. Recently, there has been a growing body of literature connecting the existence of conjugate points with the existence of unstable eigenvalues in systems with structure similar to \eqref{E:eval}. The ideas go back to classical Sturm-Liouville Theory, which can be found in many ODE textbooks; see also \cite{Arnold67,Arnold85, Bott56, Duistermaat76, Edwards64, Maslov65, Morse32, Samle65}. Their first use in a more modern dynamical systems context can be found in \cite{Jones88}. More recent works includes \cite{BairdCornwellCox21, Beck20, BeckCoxJones18, BeckJaquette22, BeckJaquettePieper24, ChardardBridges15,ChardardBridgesDias09b,ChardardDiasBridges09,ChenHu07, Cornwell19, CornwellJones20,CornwellJonesKiers21, CoxCurranLatushkinMarangell23, CoxJonesLatushkin16, CurranMarangell25, Howard21,HowardLatushkinSukhtayev17,HowardLatushkinSukhtayev18, JonesLatushkinMarangell13,slyman2025tipping}. 

This theory is based on the fact that equations of the form \eqref{E:eval} with $C(x; \lambda)$ symmetric lead to stable and unstable subspaces that are Lagrangian relative to the symplectic form $\omega$ defined by the skew-symmetric matrix $J$ given in \eqref{E:defJ}; this is connected with the fact that \eqref{E:SH-vf} is Hamiltonian. More precisely, let
\[
\omega: \R^4 \times \R^4 \to \R, \qquad \omega(P, Q) = \langle P, JQ \rangle_{\R^4}. 
\]
A Lagrangian plane is a two-dimensional subspace $\ell \subseteq \R^4$  on which $ \omega$ vanishes; i.e. such that, for any two vectors in the subspace, $P, Q \in \ell$, one has $\omega (P, Q) = 0$. In \S1.1 of \cite{BeckJaquettePieper24} it was shown that the stable and unstable subspaces associated with \eqref{E:eval} are Lagrangian for all $x \in \R$ and all $\lambda \geq 0$.   Moreover, in that work the reference plane was chosen to be the so-called sandwich plane
\begin{equation}\label{E:def-sand}
\ell_\mathrm{sand} = \left\{ \begin{pmatrix} 0 & 0 \\ 1 & 0 \\ 0 & 1 \\ 0 & 0 \end{pmatrix} U: U \in \R^2 \right \}.
\end{equation}
The modifier ``sandwich" refers to the fact that the nonzero entries are sandwiched between rows of zeros. 

The main result of \cite{BeckJaquettePieper24}, Theorem 1.5, states that, assuming a particular technical hypothesis is satisfied, then the number of unstable eigenvalues of $\mathcal{L}$ in $H^4(\R)$ is equal to the number of conjugate points of $\E^{\mathrm{u}}_-(x; 0)$ relative to the reference plane $\ell_\mathrm{sand}$. With this result in hand, we are now in a position to develop a theory that will allow us to use validated numerics \cite{tucker2011validated} to rigorously count the conjugate points (and to verify the technical hypothesis), and ultimately produce computer-assisted proofs of the stability of various pulse solutions to the Swift-Hohenberg equation. We note that the (nonvalidated) numerics that correspond to the computations we discuss here were already conducted in \cite{BeckJaquettePieper24}, and so we have strong evidence that the theoretical results in that work can be used to determine the stability of pulses.  However, to finalize this program we must synthesize a global error bound on all of these numerical computations. 

Our strategy to rigorously count the number of conjugate points for a particular pulse solution $\varphi$, and thus determine its stability, is as follows. Since we can only carry out computations on a finite domain, we first prove that there exist explicit constants $L^\pm_{\mathrm{conj}}$ such that all conjugate points lie in the compact interval $[-L^-_{\mathrm{conj}}, L^+_{\mathrm{conj}}]$. With this result in hand, the rest of the analysis is focused on counting the conjugate points in that interval. Doing so requires computing the path of the unstable subspace $\E^\ru_-(x; 0)$ and, since its basis vectors are solutions of the variational equation \eqref{E:eval}, doing so requires not only solving that variational equation, but also solving for the underlying pulse solution in a precise way. In order to develop higher-order approximations, we will utilize the parameterization method, which has been developed in \cite{CabreFontichdelaLlave03a,CabreFontichdelaLlave03b,CabreFontichdelaLlave05}. 

Briefly, the parameterization method provides a means for approximating the stable and unstable manifolds of fixed points of ODEs, and for approximating the vector bundles associated with these manifolds, in a neighborhood of the fixed point. As mentioned above, $\mathbb E^\ru_-(-L^-_{\mathrm{conj}}; 0)$ exactly corresponds to the unstable vector bundle over the unstable manifold at $\varphi(- L^-_{\mathrm{conj}})$, and because $L^-_{\mathrm{conj}}$ is large, it lies in a neighborhood of the fixed point at the origin of \eqref{E:eval}. Once we have characterized this object (and some additional related objects) via the parameterization method, we can use rigorous numerics to compute the path of $\E^\ru_-(x; 0)$ for all $x \in [-L^-_{\mathrm{conj}}, L^+_{\mathrm{conj}}]$ and look for conjugate points. 

This method is illustrated in Figure \ref{fig:computation-depiction}. The path of $\E^\ru_-(x; 0)$ on $(-\infty, -L^-_{\mathrm{conj}}]$ and $[L^+_{\mathrm{conj}}, \infty)$ corresponds to the dotted orange curve in Figure \ref{fig:computation-depiction}, which we obtain via the parameterization method. Between $x = -L^-_{\mathrm{conj}}$ (the blue dot) and $x = L^+_{\mathrm{conj}}$ (the green dot), we will use validated numerics to compute both the underlying pulse solution and the path of $\E^\ru_-(x; 0)$; this corresponds to the solid purple curve in Figure \ref{fig:computation-depiction}. The shaded plane represents the reference plane $\ell_{\mathrm{sand}}$, and the yellow star represents a conjugate point. 
\begin{figure}[H]
\centering \includegraphics[width=.8\textwidth]{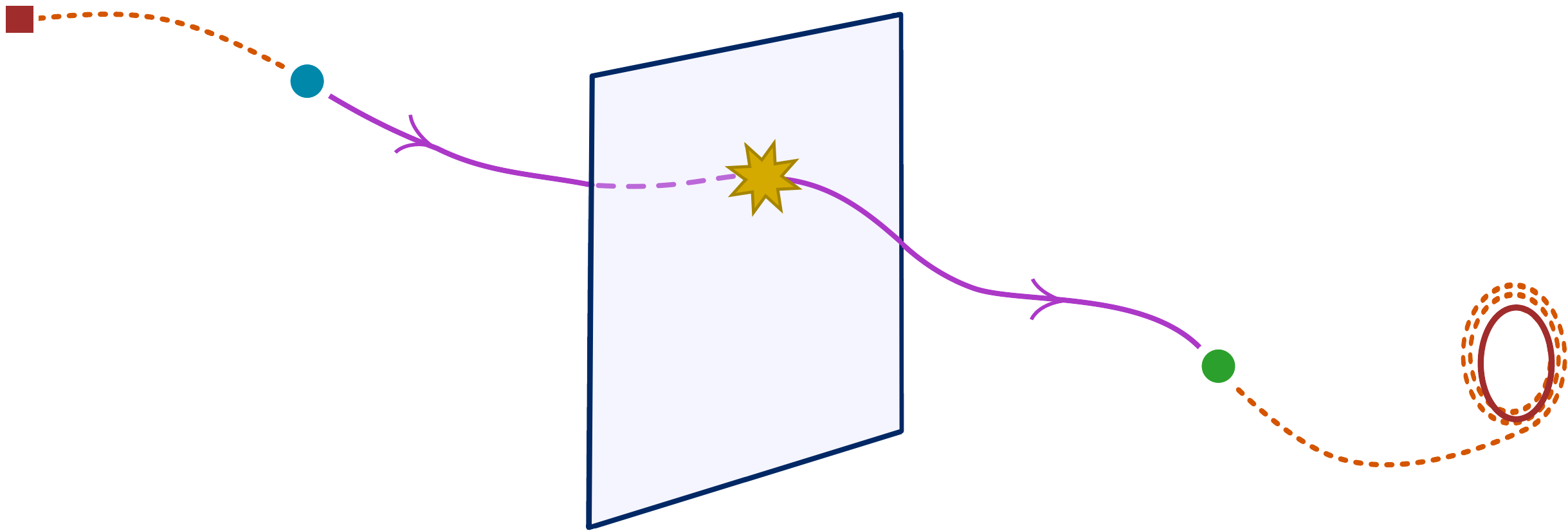}
\caption{Schematic depiction of the path of $\E^{\mathrm{u}}_-(x; 0)$. The plane $\ell_\mathrm{sand}$ is depicted as the shaded rectangle. The dotted orange line, light blue dot, green dot, purple trajectory, and yellow star all form various components of our computer assisted proof.  } \label{fig:computation-depiction}
\end{figure}

Overall this process is quite lengthy, and so we have divided things up into two papers: the present work (Part I) containing the analysis independent of the particular pulse $\varphi$ ; and a forthcoming one \cite{BeckJaquettePieperStorm25} (Part II) which synthesizes the entire analysis into a complete stability calculation. In the flowchart in Figure \ref{F:flowchart} we indicate which step is contained in which paper, and also how the different steps relate to each other. We first describe the steps that appear in this paper. 

First, the parameterization method will be used to explicitly characterize the local stable and unstable manifolds of the origin for \eqref{E:SH-vf}. These parameterizations will then enable us to also parameterize the corresponding local stable and unstable vector bundles over each manifold. Together, these steps will provide a detailed description of the asymptotic  dynamics of the underlying pulse solution $\varphi$ and of the associated stable and unstable subspaces $\E^\ru_-(x; 0)$ and $\E^{\rs, \ru}_+(x; 0)$ near $\pm \infty$. We will also need to prove a theorem that provides explicitly verifiable conditions for constants $L_{\mathrm{conj}}^\pm$ such that no conjugate points can occur on the intervals $(-\infty, -L_{\mathrm{conj}}^-]$ and $[L_{\mathrm{conj}}^+, \infty)$.
While $L^-_{\mathrm{conj}}$ can be determined from an essentially local calculation, the bound on $L^+_{\mathrm{conj}}$ is highly dependent on the specific pulse $\varphi$.

The remaining steps in the flowchart all correspond to work that will appear in Part II. In particular, Part II will begin by using validated numerics to compute the underlying pulse solution $\varphi$ on the compact interval $[-L_{\mathrm{bvp}}^-, L_{\mathrm{bvp}}^+]$; all other steps in Part II rely on the solution of this boundary value problem, and this has informed the way we have divided the analysis into Parts I and II. Once we have the pulse, we can use the variational equation, \eqref{E:eval} with $\lambda = 0$, to take the bundle $\E^\ru_-(-L_{\mathrm{bvp}}^-; 0)$, which we will have from the parameterizations near $-\infty$, and numerically propagate it forward to $x = L_{\mathrm{bvp}}^+$. 
Then, based on the specific calculation of  $\E^\ru_-(L_{\mathrm{bvp}}^+; 0)$, we will be able to calculate a bound on $L^+_{\mathrm{conj}}$. 

\begin{figure}[H]
\centering \includegraphics[width=1\textwidth]{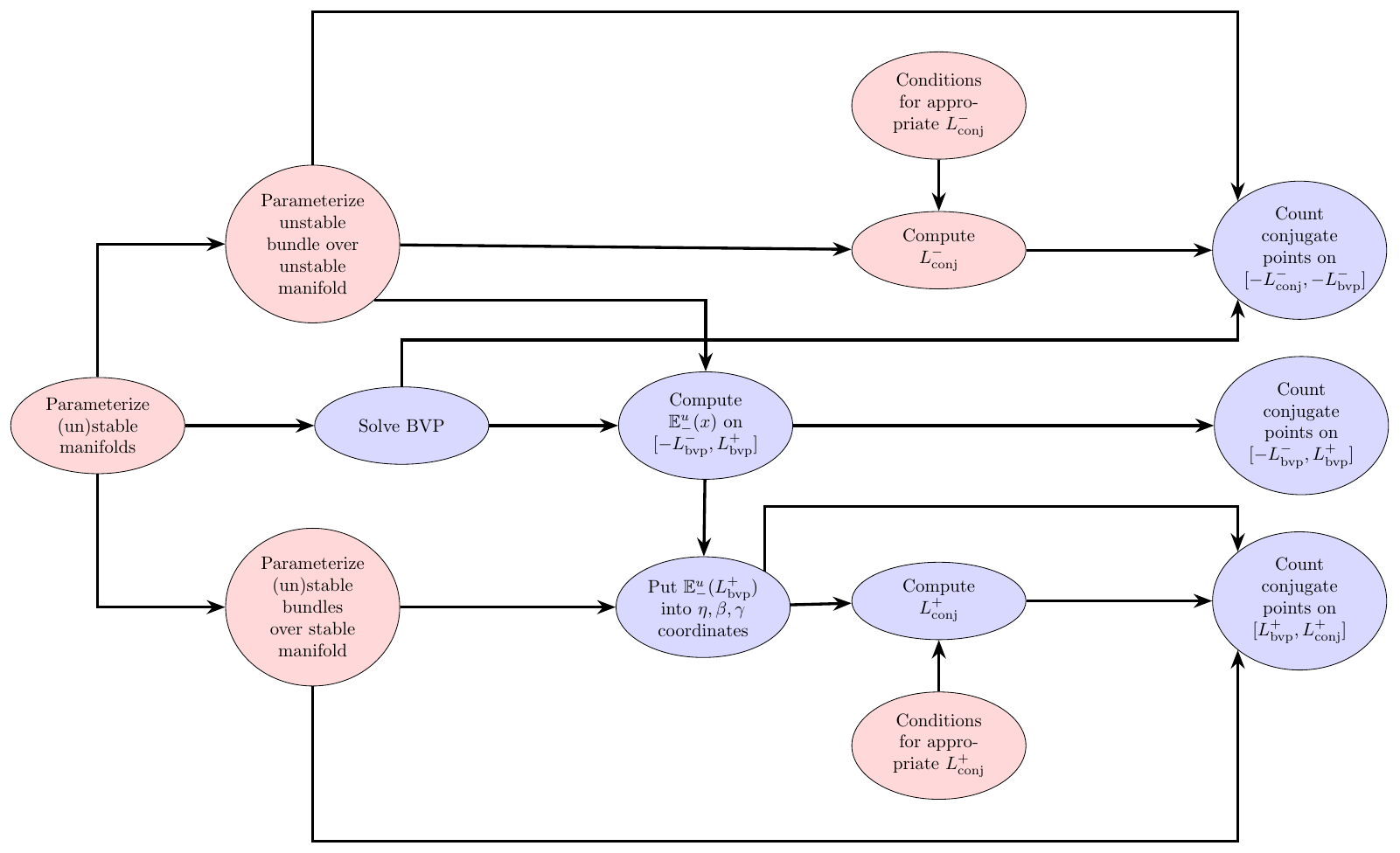}
\caption{Flowchart indicating which parts of the computer assisted proof are contained in the present paper, Part I (red), and which are contained in Part II (blue),  \cite{BeckJaquettePieperStorm25}.} \label{F:flowchart}
\end{figure}

At this point we will have an explicit, numerical representation of $\E^\ru_-(x; 0)$ on three domains: on $(-\infty, -L_{\mathrm{bvp}}^-]$, where the representation is given by the unstable bundle over the unstable manifold; on $[-L_{\mathrm{bvp}}^-,L_{\mathrm{bvp}}^+]$, where the representation is obtained by solving the variational equation; and on $[L_{\mathrm{bvp}}^+, \infty)$, where the representation is given in terms of the stable and unstable bundles over the stable manifold. With this explicit description of $\E^\ru_-(x; 0)$ in hand, we may then count the conjugate points (and verify the requisite technical hypotheses) on the three domains: $[-L_{\mathrm{conj}}^-, -L_{\mathrm{bvp}}^-]$, $[-L_{\mathrm{bvp}}^-, L_{\mathrm{bvp}}^+]$, and $[L_{\mathrm{bvp}}^+, L_{\mathrm{conj}}^+]$. 

This computational approach may be seen as an extension of \cite{BeckJaquette22}, with the difference there being that we essentially took $ L_{\mathrm{conj}}^\pm = L_{\mathrm{bvp}}^\pm$. There we used a 0th order approximation of  $\E^\ru_-(x; 0)$ for $x$ near $ \pm \infty$ (e.g. using just the eigenvectors of the asymptotic matrix), which thereby  required a long-time integration to obtain $\E^\ru_-(x; 0)$ on $[-L_{\mathrm{conj}}^-, L_{\mathrm{conj}}^+]$.  While the problem is dynamically (and thereby numerically) unstable, through the use of a high-order validated integrator \cite{kapela2021capd}, we were able to close the argument. 
In the present work, we represent $\E^\ru_-(x; 0)$ near $ \pm \infty$ over a much larger domain by using the high-order parameterized (un)stable bundles. This avoids the difficulties of long-time integration. 
While it would be fortuitous if we could also take $ L_{\mathrm{conj}}^\pm = L_{\mathrm{bvp}}^\pm$, this would be overly optimistic. 
 
We comment now on which parts of the analysis are similar to previous works, and which require the development of new theory. The determination of verifiable conditions for constants $L_{\mathrm{conj}}^\pm$ such that no conjugate points can occur on the intervals $(-\infty, -L_{\mathrm{conj}}^-]$ and $[L_{\mathrm{conj}}^+, \infty)$ is similar to what was done previously in \cite{BeckJaquette22}, but here we must deal with the additional complication that the eigenvalues of \eqref{E:asym} are complex; this is contained in \S\ref{Ch: Lpm} of this paper.  

Although the parameterization method for manifolds   \cite{CabreFontichdelaLlave03a,CabreFontichdelaLlave03b,CabreFontichdelaLlave05, haro2016parameterization}  and bundles    \cite{vandenBergJames16,van2023validated,delshams2008geometric,haro2006parameterization} has been used in many previous works, the types of resonance we encounter in the vector bundles has not been handled elsewhere, and it is here where the development of new theory is required. When employing the parameterization method, (un)stable bundles are typically assumed  to be both analytic and invariant, and this assumption yields a highly effective computational methodology. However, the bundle resonances force us to give up either analyticity or invariance, and we relinquish the latter. This introduces a nonlinear partial coupling on the bundle dynamics, similar to the torsion of an invariant torus  \cite{gonzalez2014singularity}. Nevertheless, after determining the coupling terms, the parameterization method may be employed to compute the Taylor coefficients of the bundles to high order.

Our method is described in \S\ref{S:background} and the parameterization method is applied to the manifolds and bundles for our problem in \S \ref{S:mflds} - \S\ref{Ch:res bundles}. Specifically, our new theory plays a key role in \S\ref{Ch:res bundles}. We note that our overall strategy can be adapted to non-validated numerics by skipping the validation steps. This is precisely what was done in \cite{BeckJaquettePieper24} for several pulse solutions as a proof of concept.   

Until recently, validated numerics have only been used in a handful of  instances to determine the stability of nonlinear waves. In particular, we are aware of two instances of validated numerics employing the Evans function \cite{ArioliKoch15, BarkerZumbrun16}, and one instance of validated numerics utilizing the conjugate point framework \cite{BeckJaquette22}. This latter work was conducted in the context of reaction-diffusion systems with gradient nonlinearity. Counting unstable eigenvalues with the Evans function necessitates computing a winding number as $\lambda$ varies along a contour in the complex plane, whereas counting conjugate points only requires computation for one value of $\lambda$, namely $\lambda = 0$. Thus, it is plausible that utilizing conjugate points to determine stability with validated numerics could be a more efficient strategy compared to using the Evans function. We also mention the recent success determining stability of nonlinear waves in  \cite{cadiot2025stability,cadiot2024rigorous,cadiot2025constructive,van2025existence}, where the waves are approximated by Fourier series on a large domain, and via  computer-assisted proofs the (in)stability of the waves may be obtain on the unbounded domain.  
Additionally we mention the paper \cite{BarkerBronskiHurYang25}, which focuses on the Kortegew-deVries-Burgers equation and in which a sufficient condition for stability is developed and then numerically verified. 

The Swift-Hohenberg equation has long served as a paradigmatic model for pattern formation, and it has a rich literature of study via computer-assisted proofs. On finite one-dimensional domains, early work developed validated parameter-continuation of equilibria \cite{DayLessardMischaikow07} and, using techniques based on the Conley index, demonstrated a semiconjugacy of the dynamics with explicit finite dimensional systems \cite{day2005rigorous}. More recent works have developed a  refined analysis of the infinite dimensional dynamics, constructing validated approximations of the stable manifolds of saddle equilibria  \cite{van2023validated} and validated numerical integrators \cite{duchesne2025rigorous}. In certain parameter regimes the spatial dynamics of the system have been shown to possess chaotic braided solutions \cite{vandenBergLessard08} and homoclinic snaking   \cite{vandenBergDuchesneLessard25}. Further work has studied localized solutions of the planar Swift-Hohenberg equation, such as the existence of radial solutions  \cite{vandenBergHenotLessard23}, hexagonal patterns \cite{cadiot2025stationary}, and their stability \cite{cadiot2025stability}.  

An outline of the remainder of the paper is as follows. In \S \ref{S:background} we provide the necessary background on the parameterization method and on the notion of resonance. We also illustrate our method for overcoming resonances in the simpler setting of the bistable equation. The reader who is primarily interested in our technique for overcoming vector bundle resonances, rather than the application to the Swift-Hohenberg equation, may wish to focus on this section. In \S \ref{Ch: Lpm} we determine numerically verifiable conditions for $L^\pm_{\mathrm{conj}}$ that ensure that no conjugate points exist for $x \leq -L^-_{\mathrm{conj}}$ or for $x \geq L^+_{\mathrm{conj}}$. We then apply the parameterization method to compute the stable and unstable manifolds, and their associated vector bundles, in the context of the Swift-Hohenberg equation, in \S \ref{S:mflds}, \S\ref{Ch:nonres bundles}, and \S\ref{Ch:res bundles}. Our new theory related to vector bundle resonances is made use of specifically in \S\ref{Ch:res bundles}. Finally, in \S \ref{S:CAP}, we describe the numerical implementation and computer-assisted proofs of these results.

\subsection*{Acknowledgments} We thank Jay Mireles James for helpful discussions regarding the parameterization method. \changes{We also thank Michael Storm for reading a preliminary draft and pointing out two errors.} This work supported in part by the National Science Foundation via award DMS-2205434.
  

\section{Background: analytic parameterizations of manifolds \& bundles}\label{S:background}

In order to understand the behavior of solutions to \eqref{E:eval}, which we rewrite here for convenience, 
\begin{equation}\label{E:eval2}
Q' = B(x; \lambda)Q, \qquad Q \in \R^4, \quad \lambda \in \R,
\end{equation}
and in order to in particular understand the existence (or not) of solutions that satisfy $Q(x; \lambda) \to 0$ as $x \to \pm\infty$, we must understand two classes of objects. First, since the matrix $B(x; \lambda)$ depends on the underlying pulse solution $\varphi(x)$, we must understand this solution. This function is a stationary solution of \eqref{E:SH} and hence satisfies
\[
0 = -(1+\partial_x^2)^2\varphi + f(\varphi).
\]
Equivalently, it satisfies the nonlinear ODE \eqref{E:SH-vf}, which we reproduce here
\begin{align}
u_1' &= u_2 \nonumber \\
u_2' &= u_3 \nonumber \\
u_3' &= u_4 \label{E:exist} \\
u_4' &= -2u_3 -u_1 + f(u_1). \nonumber
\end{align}
Since $\varphi$ is a pulse with $\lim_{|x| \to \infty} \varphi(x) = 0$, the vector-valued function $(\varphi, \varphi_x, \varphi_{xx}, \varphi_{xxx})$ must lie in the intersection of the stable and unstable manifolds of the origin that are associated with \eqref{E:exist}. Therefore, in order to have a good approximation of $\varphi$, we must have a good approximation of these stable and unstable manifolds.

Second, as mentioned above, the existence of solutions of \eqref{E:eval2} that decay to zero at infinity can be understood in terms of the stable and unstable subspaces defined in \eqref{E:stable-unstable}. When $\lambda = 0$, $\E^\mathrm{u}_-(x; 0)$ can be viewed as the tangent space to the unstable manifold of the origin for \eqref{E:exist}, which is the same as the unstable bundle of that manifold. Similarly, $\E^\mathrm{s}_+(x; 0)$ can be viewed as the tangent space to the stable manifold of the origin for \eqref{E:exist}, which is the same as the stable bundle of that manifold. Computing these objects is relatively straightforward. What is more involved is computing the unstable bundle of the stable manifold (or vice versa), and we will need to compute such objects. The difficulty is related to the fact that, for the stable manifold, only the stable bundle is uniquely defined, while the unstable bundle is not (and similarly for the unstable manifold). In terms of the stable and unstable subspaces, this is equivalent to the fact that $\E^\mathrm{s}_+(x; \lambda)$ is uniquely defined, while $\E^\mathrm{u}_+(x; \lambda)$ is not; similarly, $\E^\mathrm{u}_-(x; \lambda)$ is uniquely defined, while $\E^\mathrm{s}_-(x; \lambda)$ is not\footnote{These subspaces are also connected with the exponential dichotomies associated with \eqref{E:eval2}. The subspace $\E^\mathrm{u}_-(x; \lambda)$ is the range of the unstable part of the dichotomy on the negative half-line, and $\E^\mathrm{s}_+(x; \lambda)$ is the range of the stable part of the dichotomy on the positive half-line. Similarly, $\E^\mathrm{s}_-(x; \lambda)$ is the range of the stable part of the dichotomy on the negative half-line, and $\E^\mathrm{u}_+(x; \lambda)$ is the range of the unstable part of the dichotomy on the positive half-line, and in these latter two cases the ranges are not uniquely defined.}.  

\begin{remark}\label{rem:NonuniqueUnstableBundle} 
For an explicit two-dimensional example of this non-uniqueness, consider $Q' = BQ$ with $ B = \left( \begin{smallmatrix} -1  &0 \\ 0 & 1\end{smallmatrix} \right)$. This has solutions  $Q_1(x) =\left( \begin{smallmatrix} e^{-x}\\0 \end{smallmatrix} \right) $ and  $ Q_2(x) =\left( \begin{smallmatrix} 0\\e^{x} \end{smallmatrix} \right) $. While $ \E^\mathrm{s}_+(x) = \mathrm{span} \{ Q_1(x) \} $  is uniquely defined, we may define $ \E^\mathrm{u}_+(x) = \mathrm{span} \{ c Q_1(x) + Q_2(x)\} $ for any $c \in \R$. 
\end{remark}

In order to obtain high-order Taylor approximations of both the stable and unstable manifolds, and of the relevant bundles, we will use the parameterization method. Such approximations will converge in the analytic category, provided certain nonresonance conditions are met regarding the rational dependence of the (spatial) eigenvalues, which control the growth and decay rates of solutions in the manifolds and bundles\footnote{These nonresonance conditions are similar to those that appear in the Hartman-Grobman Theorem.}. We do not encounter resonances for the manifolds, but in the Hamiltonian case we find ourselves in, such a resonance is guaranteed for the bundles. In this paper we develop a novel method to overcome this obstacle. 

In this section we review the relevant literature on the parameterization method, and on invariant vector bundles. Then we present a method for overcoming resonance in the vector bundles, and we also apply this method to an explicit two-dimensional example. 

\begin{remark}
It may seem surprising that the eigenvalue equation \eqref{E:eval2} is formulated in terms of the coordinates \eqref{E:defq} that reveal the symplectic structure of the equation, while the parameterizations of the invariant manifolds and their bundles are carried out for \eqref{E:exist}, which uses different coordinates. It turns out to be convenient to compute the parameterizations in terms of the more standard coordinates. Once this is done for the Swift-Hohenberg equation, we will convert things over to the symplectic coordinates in order to compute $\mathbb{E}^\ru_-(x; 0)$ for $x \in [-L^-_{\mathrm{conj}}, L^+_{\mathrm{conj}}]$.
\end{remark}

\paragraph{Multi-index Notation} Throughout the remainder of this section, and in future sections, we make use of the following notation. An $m$-dimensional multi-index has the form $\alpha = (\alpha_1, \dots, \alpha_m ) \in \N^m$, where $\N = \{0, 1, 2, \dots\}$,   with order given by $|\alpha| = \alpha_1 + \dots + \alpha_m$. Given $\sigma \in \C^m$ and a multi-dimensional index $\alpha$, we can define the operations
\[
\sigma^\alpha = \sigma_1^{\alpha_1}\cdots \sigma_m^{\alpha_m}, \qquad \alpha \cdot \sigma = \alpha_1 \sigma_1 + \dots + \alpha_m \sigma_m.
\]
We also define the sequence space  
\[
\mathcal S_m = \big\{ b = \{b_\alpha\}\ : \ b_\alpha \in \R^n, \  \alpha = (\alpha_1, \dots, \alpha_m) \in \N^m \big\},
\] 
which is the set in which the coefficients for our series expansions, below, will lie. Note that $\R^n$ will sometimes be replaced by $\C^n$, $\R^{n \times n}$, or $\C^{n \times n}$, depending on what type of object we are formulating an expansion of. Additionally, define the Cauchy product $*:\mathcal S_m  \times \mathcal S_m  \to \mathcal S_m $ as
\begin{equation}\label{E:cprod}
(b * \tilde b)_\alpha =  \sum_{\beta \leq \alpha} b_{\alpha - \beta}\cdot \tilde b_\beta,
\end{equation}
where $\alpha - \beta = (\alpha_1 - \beta_1, \dots, \alpha_m - \beta_m)$, and 
$ \beta \leq \alpha $ iff $ \beta_1 \leq \alpha_1, \beta_2 \leq \alpha_2 , \dots, \beta_m \leq \alpha_m$.   

We note the Cauchy product in \eqref{E:cprod} defines a  product on sequences corresponding to the multiplication of scalar valued Taylor series in $m$-variables. 
If we have matrix or vector valued Taylor series, we can analogously define a Cauchy product corresponding to matrix-matrix or matrix-vector multiplication.
 
Finally, we introduce the product $\hat *: \mathcal S_m  \times \mathcal S_m  \to \mathcal S_m $ by 
\begin{equation}\label{eq: star hat def}
(b \hat * \tilde b)_\alpha  = \sum_{|\beta| \leq |\alpha|} \hat \delta^\alpha_\beta b_{\alpha - \beta} \tilde b_\beta, \qquad \qquad \hat\delta_\beta^\alpha = \begin{cases} 0 & \text{ if } \beta = 0 \\
0 & \text{ if } \beta = \alpha \\
1 & \text{ otherwise}
\end{cases}.
\end{equation}
The $\hat *$ is related to the usual Cauchy product $*$ via $ (b*\tilde b)_\alpha = b_0\tilde b_\alpha + b_\alpha \tilde b_0 + (b \hat * \tilde b)_\alpha.$ Below it will be important to place an appropriate norm on $\mathcal S_m$; see \S \ref{S:analytic-framework} for more details.


\subsection{Parameterization of an invariant manifold} \label{sec:Parameterization_Review}
 
We begin with a general discussion of the parameterization method applied to invariant manifolds of equilibria. Consider an ordinary differential equation of the form 
\begin{equation}\label{E:param}
U' = G(U), \quad G: \R^{n} \to \R^n,
\end{equation}
where $G$ is an analytic vector field that generates the nonlinear flow $\Phi(x; U)$.  We note that \eqref{E:param} may be viewed as a generalization of \eqref{E:exist}. 

Suppose that $U_* \in \R^n$ is a fixed point, and further suppose that $DG(U_*)$ is diagonalizable and hyperbolic with $m$ stable eigenvalues and $n-m$ unstable eigenvalues. The invariant manifold theorems imply that, since $U_*$ is a hyperbolic fixed point, there exist local stable and unstable manifolds in a neighborhood of $U_*$, which we denote by $\mathcal{M}^{\mathrm{s}}_{\mathrm{loc}}(U_*)$ and $\mathcal{M}^{\mathrm{u}}_{\mathrm{loc}}(U_*)$. Moreover, their dimensions are $m$ and $n-m$, respectively, and the growth/decay rates in them are determined by the corresponding sets of eigenvalues. 

There are several different approaches for proving the invariant manifold theorems and each leads to numerical methods for computing the manifolds. The parameterization method discussed here computes a numerical approximation of each invariant manifold and its internal dynamics, bounds the error in the computation, and proves each manifold exists. It was developed in \cite{CabreFontichdelaLlave03a,CabreFontichdelaLlave03b,CabreFontichdelaLlave05} and consists of finding a parameterization of the manifold that satisfies a functional equation, has a range that is invariant, and semi-conjugates the nonlinear dynamics of the ODE to the dynamics of a simpler map. Using the functional equation, we derive an approximate numerical expression for the invariant manifold. Without loss of generality, we discuss the formulation for the stable manifold. 

To begin, denote the eigenvalues of $DG(U_*)$ as $\mu_j$ and corresponding eigenvectors $\hat V_j$, for $j \in \{1, \dots, n\}$, and suppose the eigenvalues are real and ordered so they are stable for $j = 1, \dots, m$, and unstable otherwise.  Define the matrices $\Omega^\rs = \text{diag}(\mu_1 , \dots ,\mu_m)$ and $\mathcal V^\rs = \left[\hat V_1 | \cdots | \hat V_m\right]$. Consider the stable linear dynamics for equation \eqref{E:param}, $U' = \Omega^\rs U$, and note its flow is given by $e^{\Omega^\rs x}$. For all $\sigma \in [-\delta, \delta]^m =: B^m_\delta(0)$, with $\delta > 0$, the goal is to find a map $P: \R^m \to \R^n$ satisfying 
\begin{equation}\label{E:invar}
\Phi(x, P(\sigma)) = P \left(e^{\Omega^\rs x} \sigma \right), \quad P(0) = U_*, \quad DP(0) = \mathcal V^\rs.
\end{equation}
See Figure \ref{F:cong}. Requiring that fixed points are preserved under $P$ and enforcing the conjugacy relation in \eqref{E:invar} results in $P(B^m_\delta(0)) \subset \mathcal{M}^{\mathrm{s}}_{\mathrm{loc}}(U_*)$. The following lemma formulates an equivalent way to characterize the parameterization $P$ that is easier to work with computationally.

\begin{lemma}\cite{CabreFontichdelaLlave03a,CabreFontichdelaLlave03b,CabreFontichdelaLlave05}\label{lemma: param lemma manifolds} Let $P: B^m_\delta(0) \subset \R^m \to \R^n$ be a smooth function. Then for all $\sigma \in B^m_\delta(0)$, $P$ satisfies
\begin{equation}\label{E:param-1}
\Phi(x, P(\sigma)) = P\left(e^{\Omega^\rs x} \sigma\right)
\end{equation}
if and only if $P$ satisfies 
\begin{equation}\label{E:param-2}
G(P(\sigma)) = DP(\sigma)\Omega^\rs \sigma.
\end{equation}
\end{lemma}
The essential idea of the proof of this result is to differentiate \eqref{E:param-1} and evaluate it at $x=0$, which produces \eqref{E:param-2}.
 
 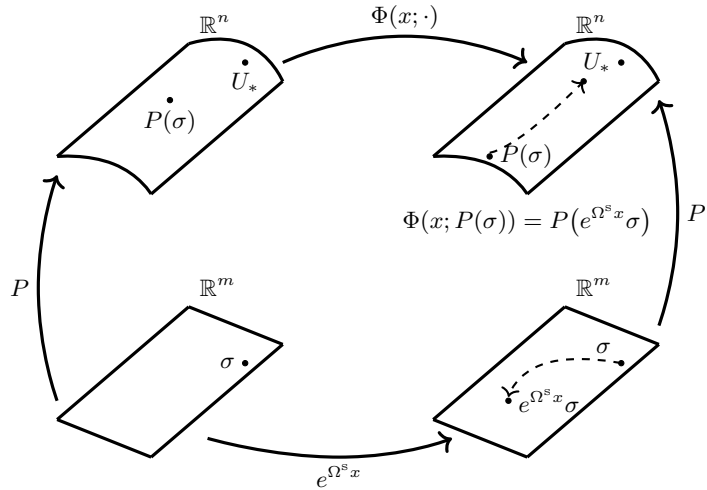
\begin{figure}
 	\centering
 	\begin{tikzpicture}[scale=.5]
 		\draw[black, very thick, -] (-9,-7) -- (-5.5,-4) node[font = \footnotesize, above right] {$\R^m$};
 		\draw[black, very thick, -] (-9,-7) -- (-6.5,-8);
 		\draw[black, very thick, -] (-6.5,-8) -- (-3,-5);
 		\draw[black, very thick, -] (-3,-5) -- (-5.5,-4);
 		\filldraw[black] (-4,-5.5) circle (2pt) node[font = \footnotesize, left] {$\sigma$};
 		
 		\draw [black, very thick, ->] plot [smooth, tension=1] coordinates { (-5,-7.5) (-1.5, -8) (-8 + 9.5,-7.5)};
 		\node[font = \footnotesize] at (-1.5, -8.5) {$e^{\Omega^\rs x}$};
 		
 		\draw[black, very thick, -] (-9 + 10,-7) -- (-5.5+ 10,-4) node[font = \footnotesize, above right] {$\R^m$};
 		\draw[black, very thick, -] (-9 + 10,-7) -- (-6.5+ 10,-8);
 		\draw[black, very thick, -] (-6.5 + 10,-8) -- (-3 + 10,-5);
 		\draw[black, very thick, -] (-3 + 10,-5) -- (-5.5 + 10,-4);
 		\filldraw[black] (-4 + 10,-5.5) circle (2pt) node[font = \footnotesize, above left] {$\sigma$};
 		\filldraw[black] (-4 + 7,-6.5) circle (2pt) node[font = \footnotesize, right] {$e^{\Omega^\rs x}\sigma$};
 		\draw [black, thick, dashed, ->] plot [smooth, tension=1] coordinates { (-4 + 10,-5.5) (-4 + 8, -5.5) (-4 + 7,-6.4)};
 		
 		\draw[black, very thick, -] plot [smooth, tension = 1] coordinates {(-9 + 10,-7 + 7)  (-5.5+ 10,-4 + 7)} node[font = \footnotesize, above right] {$\R^n$};
 		\draw[black, very thick, -] plot [smooth, tension = 1] coordinates {(-9 + 10,-7 + 7) (2.5, -.2) (-6.5+ 10,-8 + 7)} node[font = \footnotesize, below] {$\Phi(x; P(\sigma)) = P\big(e^{\Omega^\rs x} \sigma \big)$};
 		\draw[black, very thick, -] (-6.5 + 10,-8 + 7) -- (-3 + 10,-5 + 7);
 		\draw[black, very thick, -] plot [smooth, tension = 1] coordinates {(-3 + 10,-5 + 7) (-4 + 10, -5 + 8) (-5.5 + 10,-4 + 7)};
		 \filldraw[black] (-4 + 9,-5.5 + 7.5) circle (2pt); 
 		\filldraw[black] (-4 + 6.5,-6.5 + 6.5) circle (2pt) node[font = \footnotesize, right ] {$P(\sigma)$}; 
 		\filldraw[black] (-4+10,-5.5 + 8) circle (2pt) node[font = \footnotesize, left] {$U_*$};
 		\draw [black, thick, dashed, ->] plot [smooth, tension=1] coordinates { (-4 + 6.5,-6.4 + 6.5) (3.5, .65) (-4 + 9,-5.5 + 7.5) };
 				
 		\draw[black, very thick, -] plot [smooth, tension = 1] coordinates {(-9,-7 + 7)  (-5.5,-4 + 7)} node[font = \footnotesize, above right] {$\R^n$};
 		\draw[black, very thick, -] plot [smooth, tension = 1] coordinates {(-9,-7 + 7) (2.5 - 10, -.2) (-6.5,-8 + 7)};
 		\draw[black, very thick, -] (-6.5,-8 + 7) -- (-3,-5 + 7);
 		\draw[black, very thick, -] plot [smooth, tension = 1] coordinates {(-3,-5 + 7) (-4, -5 + 8) (-5.5,-4 + 7)};
 		\filldraw[black] (-4,-5.5 + 8) circle (2pt) node[font = \footnotesize, below] {$U_*$};
 		\filldraw[black] (-4 -2,-5.5 + 7) circle (2pt) node[font = \footnotesize, below] {$P(\sigma)$};
 		
 		\draw [black, very thick, ->] plot [smooth, tension=1] coordinates { (-4 + 1,-5.5 + 8) (.25, 3.2) (-5.5+ 9,-5.5 + 8)};
 		\node[font = \footnotesize] at (.25, 3.2 + .5)  {$\Phi(x; \cdot)$};
 		
 		\draw [black, very thick, ->] plot [smooth, tension=1] coordinates { (-9, -6.5) (-9.5, -3.5)  (-9, -.5)};
 		\node[font = \footnotesize] at (-10, -3.5) {$P$};
 		
 		\draw [black, very thick, ->] plot [smooth, tension=1] coordinates { (-3 + 10,-4.5) (-3 + 10.5, -5 + 3.5)  (-3 + 10,-5.5 + 7)};
 		\node[font = \footnotesize] at (-3 + 11, -5 + 3.5) {$P$};
 	\end{tikzpicture}
 	\caption{The conjugacy described by Equation \eqref{E:invar}.}\label{F:cong}
\end{figure}
 
Let $\sigma \in \R^m$, $\alpha \in \N^m$, and $P_\alpha \in \R^n$. Since $G$ is analytic, one can solve for $P$ written in the form 
\begin{equation}\label{E:Pform}
P(\sigma)  = \sum_{|\alpha| \geq 0} P_{\alpha} \sigma^\alpha, \qquad 
\end{equation}
where $P_{0} = U_*$, and the order one coefficients of $P$ are given by the eigenvectors of $DG(U_*)$, i.e. $P_{e_i} = \hat V_i$, with $ \{e_i\}_{i = 1}^m$ the standard basis vectors for $\R^m$. However, under the assumption that the nonlinear dynamics conjugate to the linear dynamics it is not always the case that one will be able to iteratively solve for $P_\alpha$ at successively higher orders in $\alpha$. This motivates the following. 
\begin{definition}\label{def:ManifoldResonance}
There is a manifold resonance of order $|\alpha|$ if, for any $j = 1, \dots n$,
\begin{equation}\label{E:manifold-resonance}
\alpha_1 \mu_1 + \dots + \alpha_m \mu_m - \mu_j = 0, \qquad |\alpha| \geq 2 .
\end{equation}
\end{definition}
 
\changes{Because the vector field $G$ is analytic, \eqref{E:invar} has an analytic solution $P$ if the eigenvalues are non-resonant in the sense of the above definition \cite{CabreFontichdelaLlave03a}.} Essentially, the reason for this is that, if one takes the expansion \eqref{E:Pform}, plugs it into \eqref{E:param-2}, and tries to solve for successively higher-order coefficients $P_{\alpha}$, then one ends up needing to divide by quantities of the form $\alpha_1 \mu_1 + \dots + \alpha_m \mu_m - \mu_j$.
\begin{remark}\label{R:man-res}
One can adapt this method to a system with an internal manifold resonance by modifying the conjugacy equation \eqref{E:param-1} so that the nonlinear dynamics are conjugated to a nonlinear normal form instead of to the linear dynamics \cite{vandenBergJamesReinhardt16}. Since we do not encounter a manifold resonance, we do not need to do this. \changes{For more information about resonances, see \cite{Homburg06}, and for more information about normal forms for Hamiltonian systems, see \cite{Moser58a, Moser58b}.} 
\end{remark}
\begin{remark}
If some of the eigenvalues are complex, it is easier to consider a power series $P$ with complex coefficients, $P_\alpha \in \C^n$ and $\sigma \in \C^m$. One can argue for a particular choice of complex $\sigma$ that $P$ will map into $\R^n$ \cite{LessardJamesReinhardt14,haro2016parameterization}. This the case that we will have for the Swift-Hohenberg equation.
\end{remark}


\subsection{Parameterization of bundles in the non-resonant case} \label{S:param method for bundles}
Again consider the setting of \eqref{E:param}. Let $U(x)$ be a solution to the initial value problem 
\begin{equation} \label{eq:NonRes_NonlinearDynamics}
U' = G(U), \qquad U(0) = U_0; 
\end{equation} 
in other words, $U(x) = \Phi(x; U_0)$. We are interested in solutions to the variational equation 
\begin{equation}\label{E:var}
V' = DG(U(x)) V.
\end{equation}
Although we present the key ideas in this general setting, it may be helpful to think of $U(x)$ as a solution to \eqref{E:exist}, and in particular to think of $U(x)$ as being the pulse solution that we linearize about. 

Denote the derivative of $\Phi$ with respect to $U$, evaluated at $U_0$, as $M(x):=D_U\Phi(x; U_0)$. The matrix $M(x)$ is the principal fundamental matrix solution of the initial value problem 
\begin{equation*}
M'(x) = DG(U(x)) M(x), \qquad M(0) = I_n. 
\end{equation*}
Suppose $\lim_{x \to \infty} U(x) = U_*$ and, as above, suppose $U_*$ is a hyperbolic equilibrium for $G$ with $DG(U_*)$ diagonalizable and eigenvalues and eigenvectors denoted by $\mu_j$ and $\hat V_j$, respectively, with $j = 1, \dots, n$. In this section we will discuss how to compute the invariant vector bundles over the stable manifold. The discussion for the unstable manifold is virtually identical. The computation over the stable manifold is relevant in \S \ref{Ch:res bundles} and the computation over the unstable manifold is relevant in \S \ref{Ch:nonres bundles}. 
 
As above, define $\Omega^\rs = \text{diag}(\mu_1 , \dots ,\mu_m)$ and $\mathcal V^\rs = \left[V_1 | \dots | V_m\right]$. Suppose that we have a parameterization for the stable manifold $P(\sigma)$ of $U_*$, meaning that $P(\sigma)$ satisfies \eqref{E:param-2}. We define an invariant vector bundle as follows.  
\begin{definition}[\cite{vandenBergJames16}]\label{def: rank 1 invariant vector bundle}
Fix $\delta> 0$ and let $W: B_\delta^m(0) \to \R^n$ be a smooth function. We say that $W$ parameterizes a rank one, forward invariant vector bundle over $P(B_\delta^m(0))$ with exponential rate $\mu$ if $W$ satisfies the equation 
\begin{equation}\label{E:def-bundle} 
M(x) W(\sigma) = e^{\mu x} W \left(e^{\Omega^\rs x} \sigma \right), \qquad \forall \ x > 0, \quad \sigma \in B_\delta^m(0).
\end{equation}
The vector bundle is exponentially contracting if $\mu < 0$ and exponentially expanding if $\mu > 0$. 
\end{definition}
This definition of invariant vector bundles is consistent with the definition formulated within the context of normally hyperbolic manifolds \cite{HirschPughShub77, Robinson95,BatesLuZeng00,eldering2018global}.  Note this is called a rank one bundle because there is one rate of contraction/expansion, given by $\mu$. 
\begin{definition}\label{D:stable-unstable-bundle}
If a rank one vector bundle $W_i(\sigma)$ has growth rate $\mu_i$ with $i = 1, \dots, m$, so that $\mu_i < 0$, we say that $W_i$ is a stable vector bundle over the stable manifold. On the other hand, if $W_i(\sigma)$ has growth rate $\mu_i$ with $i = m+1, \dots, n$, so that $\mu_i > 0$, we say that $W_i$ is an unstable vector bundle over the stable manifold. 
\end{definition}

Recall that Lemma \ref{lemma: param lemma manifolds} provided a way to relate the equation that defined the parameterized invariant manifold, \eqref{E:param-1}, with the more computable characterization \eqref{E:param-2}. We have a similar result for the vector bundles.  

\begin{lemma}[\cite{vandenBergJames16}]\label{lemma: parameterization for bundles} The smooth function $W: B^m_\delta (0) \to \R^n$ parameterizes an exponentially expanding (or contracting) rank one, forward invariant vector bundle with exponential rate $\mu \neq 0$ and satisfies the equation 
\[
M(x)W(\sigma) = e^{\mu x}W(e^{\Omega^\rs x}\sigma), \qquad x > 0
\]
if and only if $W$ is a solution to 
\begin{equation}\label{E:bundle-invar}
DG(P(\sigma)) W(\sigma) = \mu W(\sigma) + DW(\sigma)\Omega^\rs \sigma.
\end{equation}
\end{lemma}
For the proof, see \cite[Lemma 2.4]{vandenBergJames16}. We note that, intuitively, this result holds because if we differentiate \eqref{E:def-bundle} and set $x = 0$, the properties of $M$ and the fact that $U(0) = U_0 = P(\sigma)$ lead to \eqref{E:bundle-invar}.
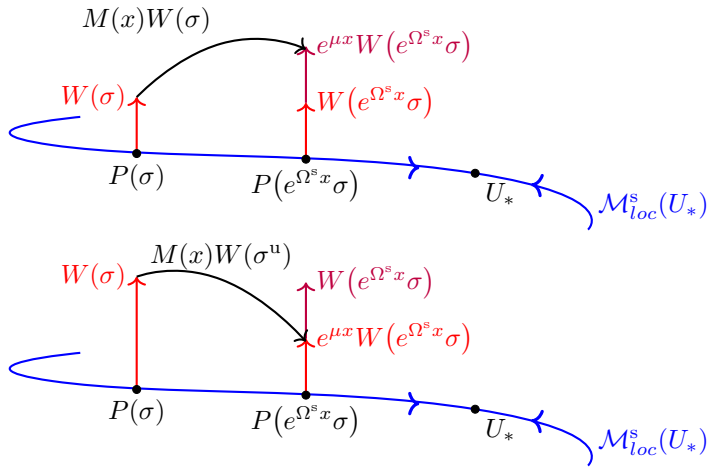
\begin{figure}[H]
	\centering
	\begin{subfigure}[b]{.6\textwidth}
		\centering
		\begin{tikzpicture}[scale=.75]
			\draw[blue, thick] plot [smooth, tension=1] coordinates { (-4,1) (-4.5,.5) (3,0) (5,-1)};
			\draw[blue, very thick, <-] (4,-.2) -- (4.001, -.2001);
			\draw[blue, very thick, ->] (2,.1) -- (2.01, .1);
			\filldraw[blue] (5,-1)  node[above right] {\small $\mathcal{M}_{loc}^\rs(U_*)$};
			
			\draw[red, thick, ->] (-3, .35) -- (-3, 1.35) node[left] {\small$W(\sigma)$};
			\draw[red, thick, ->] (0, .25) -- (0, 1.25) node[right] {\small$W\big(e^{\Omega^\rs x}\sigma\big)$};
			\draw[purple, thick, ->] (0, 1.25) -- (0, 2.25) node[right] {\small$e^{\mu x}W\big(e^{\Omega^\rs x}\sigma\big)$};
			\draw [ thick, ->] plot [smooth, tension=1] coordinates { (-3, 1.35) (-1.5, 1.5) (0, 2.2)};
			\filldraw (-.7, 1.9) node[above left] {\small$M(x)W(\sigma)$};

			\filldraw[black] (3,0) circle (2pt) node[below right] {\small$U_*$};
			\filldraw[black] (-3, .35) circle (2pt) node[below] {\small$P(\sigma)$};
			\filldraw[black] (0, .25) circle (2pt) node[below] {\small$P \big(e^{\Omega^\rs x}\sigma \big)$};
		\end{tikzpicture}
	\end{subfigure}
	
	\begin{subfigure}[b]{.6\textwidth}
		\centering
		\begin{tikzpicture}[scale=.75]
			\draw [blue, thick] plot [smooth, tension=1] coordinates { (-4,1) (-4.5,.5) (3,0) (5,-1)};
			\draw[blue, very thick, <-] (4,-.2) -- (4.001, -.2001);
			\draw[blue, very thick, ->] (2,.1) -- (2.01, .1);
			\filldraw[blue] (5,-1)  node[above right] {\small$\mathcal{M}_{loc}^\rs(U_*)$};
			
			\draw[red, thick, ->] (-3, .35) -- (-3, 2.35) node[left] {\small$W(\sigma)$};
			\draw[purple, thick, ->] (0, .25) -- (0, 1.25) node[right] {\small$e^{\mu x}W\big(e^{\Omega^\rs x}\sigma\big)$};
			\draw[red, thick, ->] (0, 1.25) -- (0, 2.25) node[right] {\small$W\big(e^{\Omega^\rs x}\sigma\big)$};
			\draw [ thick, ->] plot [smooth, tension=1] coordinates { (-3, 2.35) (-1.5, 2.3) (0, 1.2)};
			\filldraw (-1.5, 2.3) node[above] {\small$M(x)W(\sigma^\ru)$};
			
			\filldraw[black] (3,0) circle (2pt) node[below right] {\small$U_*$};
			\filldraw[black] (-3, .35) circle (2pt) node[below] {\small$P(\sigma)$};
			\filldraw[black] (0, .25) circle (2pt) node[below] {\small$P \big(e^{\Omega^\rs x}\sigma \big)$};
		\end{tikzpicture}
	\end{subfigure} 
	\caption{Expanding ($\mu > 0$) and contracting ($\mu < 0$) vector bundles.}\label{fig:expand-contract}
\end{figure}

As in the setting of Lemma \ref{lemma: param lemma manifolds} and Definition \ref{def:ManifoldResonance}, we also have to be aware of the possibility of resonances in the context of vector bundles. 
\begin{definition}\label{def: bundle resonance} There is a vector bundle resonance of order $|\alpha|$ if there exists an $\alpha \in \N^m$ with $|\alpha| \geq  1$ and such that 
\begin{equation}\label{E:vb-res}
\alpha_1 \mu_1 + \dots + \alpha_m \mu_m + \mu - \mu_j = 0
\end{equation}
for some $j \in \{1, \dots, n\}$. 
\end{definition}

Provided there are no resonances, one can find $n$ linearly independent, forward invariant vector bundles $W_i$ by assuming a power series expansion for $W_i$ in terms of $\sigma$, plugging that expansion in to equation \eqref{E:bundle-invar}, and matching orders in $\sigma$. When bundle resonances occur this need no longer be true; we will see that an invariant \emph{or} an analytic rank one bundle may still be obtained, but generically it cannot be both invariant and analytic; \changes{see Remark \ref{R:man-res}.} In the next subsection we describe our method for parameterizing resonant vector bundles, and we apply it to an example two-dimensional system in the following subsection. 


\subsection{Parameterization of bundles in the resonant case}\label{sec:Resonant-Vector-Bundle-Theory}

Consider again the real analytic vector field $G: \R^n \to \R^n$ as in \eqref{E:param} with hyperbolic equilibrium $ U_* \in \R^n$ such that $DG(U_*)$ is diagonalizable. Let $ \mu_1 \dots \mu_n \in \R$ denote the eigenvalues with associated eigenvectors $\hat V_1, \dots \hat V_n \in \R^n$ of $DG(U_*)$.  Assume they are labeled so that $ \mu_1 < \dots \mu_m <0 <\mu_{m+1} < \dots< \mu_n$ and recall $\Omega^\rs = \mathrm{diag}(\mu_1, \dots, \mu_m)$.
Suppose that $P :B^m_\delta(0) \to \R^n$ is a parameterization of the stable manifold which conjugates to the linear flow as in equation \eqref{E:param-1}. 
	
We aim to understand the linear dynamics along the stable manifold. That is, for $\sigma_0 \in B^m_\delta(0)$ and initial condition $P(\sigma_0) \in \mathcal{M}^\rs_{loc}(U_0)$, we aim to understand the solution of 
\begin{equation}
V'(x) = DG( P(e^{\Omega^\rs x} \sigma_0)) V(x) \label{eq:OriginalBundleEquation}
\end{equation}
for general initial condition $V(0) \in \R^n$.  On the one hand we can see this as a non-autonomous differential equation. Alternatively, we may view this as a differential equation on the tangent bundle\footnote{For a more abstract classification of this type of system, see for example the description of linear skew product systems in \cite[Cpt. 4.3]{SellYou02}.} 
\[
T_{\mathcal{M}^\rs}\R^n = \bigcup_{U \in \mathcal{M}^\rs} \{ U\} \times T_U\R^n,
\]
where for $(U,V) \in T_{\mathcal{M}^\rs}\R^n$ we have 
\begin{equation}\label{E:bundle-dynamics}
U'=G(U), \qquad V'= DG(U) V.
\end{equation}
In this context, we present a generalization of Definition \ref{def: rank 1 invariant vector bundle}. 

\begin{definition}\label{def:Rank1UnstableBundle} Let $ \mathcal M \subseteq \R^n$ be a forward invariant manifold of \eqref{eq:NonRes_NonlinearDynamics} and consider the dynamics on the tangent bundle $T_{\mathcal M}\R^n$ given by \eqref{E:bundle-dynamics}. We say that a continuous mapping $ W^{\ru} : \mathcal{M} \to T_\mathcal{M}\R^n $ parameterizes a rank one unstable bundle if there exist $C,\mu >0$ such that, for all $ U_0 \in \mathcal{M}$ and for all $V_0 \in \mathrm{span}\{W^\ru(U_0)\}$, the solution $V(x)$ of \eqref{E:var} with $ V(0) = V_0$ satisfies $ \| V(x) \| \geq C e^{ \mu x} \| V_0\|$ for all $ x \geq 0$. 
\end{definition}
	
We say a rank one unstable bundle is \emph{invariant} if the solution satisfies $ V(x ) \in  \mathrm{span}\{W^\ru(U(x))\}$ for all  $x \geq 0$. We say that a rank one unstable bundle is \emph{smooth/analytic} if $W^{\ru} : \mathcal{M} \to T_\mathcal{M}\R^n $ is smooth/analytic. If $d$ rank one unstable bundles $W^\ru_1 , \dots,  W^\ru_d$ are everywhere linearly independent, then the Whitney sum of their respective vector bundles defines a rank $d$ unstable bundle. 

We aim to parameterize the stable and unstable bundles of $\mathcal{M}^\rs$, and conjugate their linear dynamics to a type of normal form, while respecting the dynamics on the base space. Here the parameterization of the bundles will be given by a map $\mathcal{W}: T_{B^m_\delta(0)}\R^n \to T_{\mathcal{M}^\rs}\R^n$ and the normal form of the dynamics will be determined by $\mathcal{A}: T_{B^m_\delta(0)}\R^n  \to  T_{B^m_\delta(0)}\R^n $. This is summarized in the commutative diagram below,
\begin{center}
	\hfill
		\begin{tikzcd}
		 \mathcal{M}^\rs  \arrow[r, "\Phi^G"] &  \mathcal{M}^\rs  \\
	 B^m_\delta(0)  \arrow[r, "e^{\Omega^\rs x}"] \arrow[u, "P"] &  B^m_\delta(0) \arrow[u, "P"]
	\end{tikzcd}
\hfill
	\begin{tikzcd}
		T_{\mathcal{M}^\rs}\R^n \arrow[r, "\Phi^{DG}"] & T_{\mathcal{M}^\rs}\R^n \\
		T_{B^m_\delta(0)}\R^n \arrow[r, "\Phi^{\mathcal{A}}"] \arrow[u, "\mathcal{W}"] & 	T_{B^m_\delta(0)}\R^n \arrow[u, "\mathcal{W}"]
	\end{tikzcd}
\hfill \null
\end{center}
where the flows $\Phi^{G, DG, \mathcal{A}}$ are those generated by the differential equations $U' = G(U)$, $V' = DG(P(e^{\Omega^\rs x}\sigma_0))V$, and $\tilde V' = \mathcal{A}(e^{\Omega^\rs x}\sigma_0)) \tilde V$, respectively. We can write the domain and range of $\mathcal{W}$ and $\mathcal{A}$ in another format as follows. Let $\GL(\R^n)$ denote the general linear group of real $n\times n$ invertible matrices. Then we aim to find real analytic mappings
\begin{align*}
\mathcal{W}:B^m_\delta(0) \to \GL(\R^n)    && 	\mathcal{A}:B^m_\delta(0) \to \GL(\R^n)
\end{align*}
that satisfy the following conditions. 

\paragraph{Conjugacy Equation.} Suppose $\mathcal{W}$ and $\mathcal{A}$ are defined so that the dynamics commute in the following sense. Fix $\sigma_0 \in B^m_\delta(0)$ and initial condition $ P(\sigma_0) \in \mathcal{M}^\rs_{loc}(U_*)$, and suppose that $V(x)$ solves \eqref{eq:OriginalBundleEquation}  for $x\geq 0$. 
Suppose also that $\tilde{V}(x)$ is defined so that 
\begin{align}\label{eq:Vtilde_Def}
V(x) = \mathcal{W}(e^{\Omega^\rs x} \sigma_0) \tilde{V}(x)
\end{align}
Also suppose that $\tilde{V}(x)$ satisfies the differential equation 
\begin{align}
\tilde{V}'(x) = \mathcal{A}( e^{\Omega^\rs x} \sigma_0) \tilde{V}(x) , \qquad x\geq 0.
\label{eq:NormalFormBundleEquation}
\end{align}
We obtain the conjugacy equation by taking the derivative of \eqref{eq:Vtilde_Def} on both sides. From the left hand side of \eqref{eq:Vtilde_Def} and applying \eqref{eq:OriginalBundleEquation}, we obtain 
\begin{align*}
V'(x) &= DG( P(e^{\Omega^\rs x} \sigma_0)) V(x) \\
&= DG( P(e^{\Omega^\rs x} \sigma_0)) \mathcal{W}(e^{\Omega^\rs x} \sigma_0) \tilde{V}(x) .
\end{align*}
From the right hand side of \eqref{eq:Vtilde_Def}, we obtain 
\begin{align*}
\frac{d}{dx} [\mathcal{W}(e^{\Omega^\rs x} \sigma_0)\tilde{V}(x)]  &= D\mathcal{W}(  e^{\Omega^\rs x} \sigma_0 ) \left(
\frac{d}{dx}e^{\Omega^\rs x} \sigma_0 \right) \tilde{V}(x) +\mathcal{W}(  e^{\Omega^\rs x} \sigma_0 ) \tilde{V}'(x) \\
&= D\mathcal{W}(  e^{\Omega^\rs x} \sigma_0 ) \left(e^{\Omega^\rs x}\Omega^\rs \sigma_0 \right) \tilde{V}(x) +\mathcal{W}(e^{\Omega^\rs x} \sigma_0 )\mathcal{A}( e^{\Omega^\rs x} \sigma_0) \tilde{V}(x). 
\end{align*}
Setting these two sides equal and taking the limit as $ x \to 0$, we obtain the conjugacy equation
\[
DG(P(\sigma_0))\mathcal{W}(\sigma_0) \tilde{V}(0) = D\mathcal{W}(\sigma_0)\left(\Omega^\rs \sigma_0 \right)\tilde{V}(0) +\mathcal{W}(\sigma_0)\mathcal{A}(\sigma_0) \tilde{V}(0).
\]
We want this to hold for all vectors $\tilde{V}(0) \in \R^n$ and all $ \sigma_0 \in B^m_\delta(0)$, so we could equivalently write
\begin{equation}\label{eq:BundleConjugacy}
DG(P(\sigma)) \mathcal{W}(\sigma) = D\mathcal{W}(\sigma)  \Omega^\rs \sigma + \mathcal{W}(\sigma )\mathcal{A}(\sigma).  
\end{equation}
We formalize the above argument in the following.
\begin{lemma}\label{L:conjugacy} Let $\mathcal{A}: B_\delta^m(0) \to \GL(n)$ be given. The following two statements are equivalent.
\begin{enumerate}
\item There exists a function $\mathcal{W}: B_\delta^m(0) \to \GL(n)$ such that \eqref{eq:BundleConjugacy} holds for all $\sigma \in B_\delta^m(0)$.
\item There exists a function $\mathcal{W}: B_\delta^m(0) \to \GL(n)$ such that, if $\tilde M(x) \in \R^{n \times n}$ satisfies
\begin{equation}\label{E:Mtilde}
\tilde M'(x) = \mathcal{A}(e^{\Omega^\rs x} \sigma)\tilde M(x), \qquad \tilde M(0) = I
\end{equation}
for all $x \geq 0$, then $M(x) = \mathcal{W}(e^{\Omega^\rs x}\sigma)\tilde M(x)$ solves 
\begin{equation}\label{E:M}
M'(x) = DG(P(e^{\Omega^\rs x}\sigma))M(x), \qquad M(0) = \mathcal{W}(\sigma)
\end{equation}
for $x \geq 0$.
\end{enumerate}
\end{lemma}
The proof of this Lemma is similar to that of \cite[Lemma 2.4 and Thm 2.6]{vandenBergJames16}. The key difference is that the result in \cite{vandenBergJames16} asserts that, in the absence of any vector bundle resonance, each column of $\mathcal{W}$ can be computed separately and independently of the others. Moreover, \eqref{eq:BundleConjugacy} differs from equation (14) of \cite{vandenBergJames16} in that the $\mathcal{W}(\sigma)\mathcal{A}(\sigma)$ term of the right hand side of \eqref{eq:BundleConjugacy} is replaced by $\mathcal{W}(\sigma)\Omega$, where $\Omega = \mathrm{diag}(\mu_1, \dots, \mu_n)$. In our case, due to the resonance, we cannot independently determine the columns of $\mathcal{W}(\sigma)$, and equation \eqref{E:Mtilde} that we conjugate the variational equation to cannot be taken to be $\tilde M' = \Omega \tilde M$. 
\begin{proof}
First suppose statement (1) holds. Fix $\hat \sigma \in B_\delta^m(0)$ and define $M(x) = \mathcal{W}(e^{\Omega^\rs x}\hat \sigma)\tilde M(x)$. Then $M(0) = \mathcal{W}(\hat \sigma)$. In addition, 
\[
M'(x) = D\mathcal{W}(e^{\Omega^\rs x}\hat \sigma) \Omega^\rs e^{\Omega^\rs x} \hat \sigma \tilde M(x) + \mathcal{W}(e^{\Omega^\rs x}\hat \sigma)\mathcal{A}(e^{\Omega^\rs x} \hat \sigma)\tilde M(x).
\]
On the other hand, by \eqref{eq:BundleConjugacy}, we find
\begin{align*}
DG(P(e^{\Omega^\rs x} \hat \sigma))M(x) & = DG(P(e^{\Omega^\rs x} \hat \sigma))\mathcal{W}(e^{\Omega^\rs x}\hat \sigma)\tilde M(x) \\
&= [D\mathcal{W}(e^{\Omega^\rs x} \hat \sigma)  \Omega^\rs e^{\Omega^\rs x} \hat \sigma + \mathcal{W}(e^{\Omega^\rs x} \hat \sigma )\mathcal{A}(e^{\Omega^\rs x} \hat \sigma) ] \tilde M(x),
\end{align*}
since $e^{\Omega^\rs x} \hat \sigma \in B_\delta^m(0)$ for all $x \geq 0$. 

Next, suppose statement  (2) holds. Differentiating $M(x) = \mathcal{W}(e^{\Omega^\rs x}\sigma)\tilde M(x)$ and using \eqref{E:Mtilde} and \eqref{E:M}, we find
\[
DG(P(e^{\Omega^\rs x}\sigma))M(x) = D\mathcal{W}(e^{\Omega^\rs x}\sigma) \Omega^\rs e^{\Omega^\rs x}\sigma \tilde M(x) + \mathcal{W}(e^{\Omega^\rs x}\sigma)\mathcal{A}(e^{\Omega^\rs x} \sigma)\tilde M(x)
\]
for all $x \geq 0$. Setting $x = 0$ and using the facts that $M(0) = \mathcal{W}(\sigma)$ and $\tilde M(0) = I$, we obtain \eqref{eq:BundleConjugacy}.
\end{proof}

\paragraph{Choice of normal form.} We are now free to choose any $\mathcal{A}: B_\delta^m(0) \to \GL(n)$ and apply Lemma \ref{L:conjugacy}. We want to choose $\mathcal{A}(\sigma)$ to be as simple as possible. In the nonresonant case, the choice 
$\mathcal{A}(\sigma) = \Omega$ for all $\sigma$, where $\Omega = \mathrm{diag}(\mu_1, \dots, \mu_n)$ is the Jordan normal form of $DG(U_*)$, is possible. In the resonant case, we will still be able to choose $\mathcal{A}(0) = \Omega$, but we will need to let $\mathcal{A}$ vary with $\sigma$. Recall that $\hat V_1, \dots, \hat V_n$ are the corresponding eigenvectors of $DG(U_*)$ and set
\begin{align*}
\mathcal{A}(0) &= \begin{pmatrix} \mu_1 &\cdots &0\\ \vdots &\ddots & \vdots \\ 0 & \cdots & \mu_n \end{pmatrix}
&
\mathcal{W}(0) &= \begin{pmatrix} \vdots & &\vdots\\ \hat V_1 &\cdots & \hat V_n \\ \vdots &  & \vdots \end{pmatrix}
\end{align*}
It turns out that we may take the form of $\mathcal{A}(\sigma)$ to be
\begin{equation}\label{E:Aform}
\mathcal{A}(\sigma) = \begin{pmatrix} \mu_1 &a^{1, 2}(\sigma) & a^{1,3}(\sigma) & \cdots & a^{1, n}(\sigma) \\  0 & \mu_2 & a^{2, 3}(\sigma) & \cdots &a^{2, n}(\sigma) \\ \vdots & \vdots &   \vdots & \ddots & \vdots \\ 0 & 0 & \cdots & 0 & \mu_n \end{pmatrix}.
\end{equation}

The upper triangular structure of $\mathcal{A}$ follows from the ordering of the eigenvalues.  To see this, let us define sets of multi-indices which give rise to resonances as per Definition \ref{def: bundle resonance}.  
\begin{definition}\label{def:bundle-resonance--Res_ij} For each fixed $i$ and $j$, define $\mathrm{Res}_{i,j} \subsetneq \N^m$ as the collection of resonances of order $\alpha \in \N^m$ which the $\mu_j$-bundle has in the $\mu_i$ location, formally given as 
\begin{align} \label{E:Res_def}
\mathrm{Res}_{i,j} = \{ \alpha \in \N^m :  | \alpha| \geq 1, \ \alpha_1 \mu_1 + \dots + \alpha_m \mu_m  + \mu_j - \mu_i = 0  \} .
\end{align}
\end{definition}

It follows from the ordering of eigenvalues (i.e. $i<j \implies \mu_i < \mu_j$) that if $ i \geq j$, then $ 	\mathrm{Res}_{i,j} = \emptyset$. If $\mathrm{Res}_{i,j} = \emptyset$, then we may take $ a^{i,j}(\sigma)= \mu_j \delta_{ij}$, where $\delta_{ij}$ denotes the Kronecker delta. However if there are resonances, then it will be necessary to solve for $a^{i,j}(\sigma)$ via
\begin{align*}
a^{i,j}(\sigma) = \sum_{  \alpha \in \mathrm{Res}_{i,j} } a^{i,j}_\alpha \sigma^\alpha. 
\end{align*} 
If the coefficients $a^{i,j}_\alpha$ are defined appropriately, then both $\mathcal{A}$ and $\mathcal{W}$ can be defined as convergent power series that satisfy the bundle conjugacy equation in \eqref{eq:BundleConjugacy}. Due to only a finite number of resonances being possible, then in fact $\mathcal{A}(\sigma)$ may be taken as a finite polynomial. Note that, whenever we are forced to define a particular $a^{i,j}_\alpha \neq 0$, this will lead to a free choice of a coefficient in the expansion of $ \mathcal{W}(\sigma)$.\footnote{This is related to the fact that the unstable bundle over a stable manifold is not unique, as you can just add a bit of the stable solution, and you still get solutions that grow exponentially.} The form of $\mathcal{A}$ given in \eqref{E:Aform} will be justified in Lemma \ref{L:solve-for-AQ} below. First, we introduce a bit more notation.

\paragraph{Homological equation.} We follow the presentation from \cite[\S 3]{vandenBergJames16}, modifying as necessary for our framework. Suppose that we have computed the matrix valued power series 
\begin{equation}\label{E:defB}
DG(P(\sigma)) = \sum_{|\alpha|=0}^\infty \hat {G}_\alpha \sigma^\alpha
\end{equation}
and that we may write $\mathcal{W}$ and $\mathcal{A}$ as a power series
\begin{equation}\label{E:defQA}
\mathcal{W}(\sigma) =  \sum_{|\alpha|=0}^\infty \mathcal{W}_\alpha \sigma^\alpha, \qquad \mathcal{A}(\sigma) =  \sum_{|\alpha|=0}^\infty \mathcal{A}_\alpha \sigma^\alpha.
\end{equation}
Let $\mathcal{W}_0=\mathcal{W}(0)$ be a matrix whose columns are eigenvectors of $DG(0)$, let $\mathcal{A}_0 = \Omega = \mathrm{diag}(\mu_1, \dots, \mu_n)$, and for each $\alpha$ let $\mathcal{A}_\alpha, \mathcal{W}_\alpha \in \R^{n \times n}$. Expanding right hand side of \eqref{eq:BundleConjugacy} we get 
\[
D\mathcal{W}(\sigma)  \Omega^\rs \sigma + \mathcal{W}(\sigma )\mathcal{A}(\sigma) = \sum_{|\alpha|=0}^\infty \left[ (\alpha_1 \mu_1 + \dots + \alpha_m \mu_m) \mathcal{W}_\alpha + (\mathcal{W}*\mathcal{A})_\alpha \right]  \sigma^\alpha,  
\]
where $\mathcal{W}*\mathcal{A}$ denotes a Cauchy product of matrices, see \eqref{E:cprod}. 
Expanding the left hand side of \eqref{eq:BundleConjugacy} we get 
\[
DG(P(\sigma)) \mathcal{W}(\sigma)  = \sum_{|\alpha|=0}^\infty (\hat{G}*\mathcal{W})_\alpha \sigma^\alpha. 
\]
By equating both sides and matching powers, we obtain 
\[
 (\alpha_1 \mu_1 + \dots + \alpha_m \mu_m) \mathcal{W}_\alpha + (\mathcal{W}*\mathcal{A})_\alpha =   (\hat{G}*\mathcal{W})_\alpha. 
\]
In order to isolate the $ \mathcal{W}_\alpha$ term, we write
\[
(\alpha_1 \mu_1 + \dots + \alpha_m \mu_m) \mathcal{W}_\alpha + \mathcal{W}_\alpha  \mathcal{A}_0 + \mathcal{W}_0  \mathcal{A}_\alpha  + (\mathcal{W}\hat{*}A)_\alpha  =  \hat{G}_\alpha  \mathcal{W}_0 + \hat{G}_0 \mathcal{W}_\alpha +  (\hat{G}\hat{*}\mathcal{W})_\alpha,
\]
where $ \hat{*}$ is defined in \eqref{eq: star hat def}. Define 
\begin{equation}\label{E:defS}
\cS_\alpha := \hat{G}_\alpha  \mathcal{W}_0 +  (\hat{G}\hat{*}\mathcal{W})_\alpha - (\mathcal{W}\hat{*}\mathcal{A})_\alpha
\end{equation}
From this we obtain the homological equation 
\begin{equation}\label{eq:ResonantHomologicalEquation}
(\alpha_1\mu_1 + \dots + \alpha_m \mu_m) \mathcal{W}_\alpha + \mathcal{W}_\alpha  \Omega -DG(U_*) \mathcal{W}_\alpha + \mathcal{W}_0  \mathcal{A}_\alpha  = \cS_\alpha
\end{equation}
where we recall that $\hat{G}_0 = DG(U_*)$ and $\Omega  = \mathcal{A}_0$ is the diagonal matrix of eigenvalues. Since $\mathcal{W}_0$ has as its columns the eigenvectors of $DG(U_*)$, we have $DG(U_*) = \mathcal{W}_0 \Omega \mathcal{W}_0^{-1}$, and so we define $\tilde{\mathcal{W}}_\alpha$ and $\tilde \cS_\alpha$ via the change of variables
\[
\mathcal{W}_\alpha = \mathcal{W}_0 \tilde{\mathcal{W}}_\alpha, \qquad \cS_\alpha = \mathcal{W}_0 \tilde \cS_\alpha.
\]
Multiplying \eqref{eq:ResonantHomologicalEquation} on the left by $\mathcal{W}_0^{-1}$, we obtain
\begin{equation}\label{eq:ResonantHomologicalEquation-2}
(\alpha_1\mu_1 + \dots + \alpha_m \mu_m) \tilde{\mathcal{W}}_\alpha + \tilde{\mathcal{W}}_\alpha  \Omega - \Omega \tilde{\mathcal{W}}_\alpha +   \mathcal{A}_\alpha  = \tilde \cS_\alpha .
\end{equation}
In the following lemma we summarize how solving this equation leads to a solution of the conjugacy equation. 

\begin{lemma}\label{L:solve-for-AQ}
Let $\hat G_\alpha$ be defined by \eqref{E:defB} for $|\alpha| \geq 0$. Then the expansions of $\mathcal{W}$ and $\mathcal{A}$ as in \eqref{E:defQA} satisfy \eqref{eq:BundleConjugacy} at each order in $\sigma$ if $\mathcal{A}_0 = \Omega$, $\mathcal{W}_0  = [\hat V_1| \dots | \hat V_n]$, and $\mathcal{W}_\alpha$ and $\mathcal{A}_\alpha$ are defined as follows:

Let $S_\alpha$ be defined in \eqref{E:defS}. Set $\mathcal{W}_\alpha = \mathcal{W}_0 \tilde{\mathcal{W}}_\alpha$ and $S_\alpha = \mathcal{W}_0 \tilde S_\alpha$, and let the entries of $\tilde S_\alpha$, $\tilde{\mathcal{W}}_\alpha$ and $\mathcal{A}_\alpha$ in row $i$ and column $j$ be denoted by $\tilde s_\alpha^{i,j}$, $\tilde w_\alpha^{i,j}$ and $a_\alpha^{i,j}$, respectively. 
Then we define
\[
\tilde w_\alpha^{i,j} =  \frac{1}{\alpha_1 \mu_1 + \dots + \alpha_m \mu_m  + \mu_j - \mu_i  } \tilde s_\alpha^{i,j}, \qquad    a_{\alpha}^{i,j} = 0,  \qquad |\alpha| \geq 1, \qquad \alpha \notin \mathrm{Res}_{i,j},
\]
and 
\[
\tilde w_\alpha^{i,j} = 0, \qquad  a_{\alpha}^{i,j}  =  \tilde s_\alpha^{i,j}, \qquad  |\alpha| \geq 1, \qquad \alpha \in \mathrm{Res}_{i,j}.
\]
Moreover, $a_\alpha^{i, j} = 0$ for all $i > j$, and $a_\alpha^{i,j} = 0$ whenever $|\alpha| > \max\{|\alpha|: \alpha \in \mathrm{Res}_{i,j}\}$.
\end{lemma}

\begin{proof} First notice that $S_\alpha$ be defined in \eqref{E:defS} only depends on $\mathcal{W}_{\tilde \alpha}$ and $\mathcal{A}_{\tilde \alpha}$ for $|\tilde \alpha| < |\alpha|$. Thus, if we construct $\mathcal{W}_\alpha$ and $\mathcal{A}_\alpha$ iteratively at each order in $\alpha$, the procedure outlined in the statement of this Lemma is well defined.  Next, if we introduce the notation
\begin{align*}
\tilde{\mathcal{W}}_\alpha &=  [ \tilde w_\alpha^{*, 1} | \dots | \tilde w_\alpha^{*, j}| \dots \tilde w_\alpha^{*, n}], & \tilde S_\alpha &= [ \tilde s_\alpha^{*, 1} | \dots | \tilde s_\alpha^{*, j} | \dots |\tilde s_\alpha^{*, n}], & \mathcal{A}_\alpha &= [a_\alpha^{*, 1} | \dots | a_\alpha^{*, j} | \dots | a_\alpha^{*, n}],
\end{align*}
then we find the columns of \eqref{eq:ResonantHomologicalEquation-2} yield
\begin{equation} \label{eq:ColumnByColumn}
\left[ (\alpha_1 \mu_1 + \dots + \alpha_m \mu_m  + \mu_j)I -\Omega \right] \tilde w_\alpha^{*, j} + a_\alpha^{*, j} = \tilde s_\alpha^{*, j}.
\end{equation}
Therefore, the asserted formulas for $\tilde w_\alpha^{i,j}$ and $a_\alpha^{i,j}$ will indeed ensure that \eqref{eq:BundleConjugacy} is satisfied at each order in $\sigma$.

Next, recall that the eigenvalues are ordered so that $\mu_1 < \dots < \mu_m < 0 < \mu_{m+1} < \dots < \mu)n$. This implies via Definition \ref{def:bundle-resonance--Res_ij}  that we can only possibly have a resonance if $\alpha_1 \mu_1 + \dots + \alpha_m \mu_m = \mu_i - \mu_j \geq \mu_1 - \mu_n$. On the other hand, since $\alpha \in \N^m$, we find that, $\alpha_1 \mu_1 + \dots + \alpha_m \mu_m \leq |\alpha|\mu_m$. Thus, there cannot be a resonance at all if $|\alpha| > (\mu_1 - \mu_n)/\mu_m$. Moreover, $\max\{|\alpha|: \alpha \in \mathrm{Res}_{i,j}\} \leq (\mu_i - \mu_j)/\mu_m$, which is finite. By definition, $a_\alpha^{i,j} = 0$ whenever $|\alpha| > \max\{|\alpha|: \alpha \in \mathrm{Res}_{i,j}\}$. Moreover, if $i > j$, then $\mu_i - \mu_j > 0$, and since $\alpha_1 \mu_1 + \dots + \alpha_m \mu_m < 0$ we cannot have a resonance for that choice of $i$ and $j$. Thus, $a_\alpha^{i, j} = 0$ for all $i > j$.
\end{proof}

This Lemma justifies the form of $\mathcal{A}(\sigma)$ given in \eqref{E:Aform}. In particular, the entries below the diagonal are all zero, since $i > j$ there. In addition, when $i = j$ the resonance condition becomes $\alpha_1 \mu_1 + \dots + \alpha_m \mu_m = 0$, which does not hold if $|\alpha| > 1$. Thus, $a_{i, i}(\sigma) = \mu_i$ since $\mathcal{A}_0 = \Omega$. Finally, the fact that $\mathcal{A}_0 = \Omega$ implies that $a_{i,j} = \mathcal{O}(\sigma)$ for $i < j$, and the lemma also implies that the expansions of $a_{i, j}(\sigma)$ are finite for $i < j$.  

\begin{remark}
The form of $\mathcal{A}(\sigma)$ given in \eqref{E:Aform} together with \eqref{E:Mtilde}, which describes the fundamental matrix solution corresponding to the normal form for the bundle, implies that in the resonant case we expect the bundles to be sums of both exponentially growing/decaying solutions, plus terms that look like polynomials times exponentially growing/decaying terms. See the example in \S\ref{S:bistable} for an explicit example of this.
\end{remark}


\subsection{Example with vector bundle resonance: the bistable equation}\label{S:bistable}

\changes{In this section we present an example where one can explicitly compute the pulse, the parameterizations of the stable and unstable manifolds, and the solutions to the variational equation about the pulse. We choose this particular example because it has a vector bundle resonance, and we use this to illustrate how such a resonance can prevent one from computing the vector bundles via Lemma \ref{lemma: parameterization for bundles}. In particular, we see below in equation \eqref{E:choice-bad} that one would encounter an inconsistent system of equations if one tries to implement Lemma \ref{lemma: parameterization for bundles}. We note this reflects the fact that, in this example, and more generally in the presence of a resonance, one cannot expect to have a bundle that is both invariant and analytic.}

Consider the scalar bistable equation
\[
u_t = u_{xx} + h(u), \qquad h(u) = u(2u-1)(1-u).
\]
This has three spatially homogeneous solutions $ u \equiv 0,1/2,1$. 
Moreover, stationary solutions satisfy 
\begin{equation}\label{E:nonlin-stationary}
0 = \varphi_{xx} + h(\varphi). 
\end{equation}
Two explicit solutions of \eqref{E:nonlin-stationary} are given by 
\begin{equation}\label{E:explicit} 
\varphi_\pm(x) = \frac{1}{1 + e^{\pm x}},
\end{equation}
which each correspond to a standing front of the original PDE. 

We rewrite equation \eqref{E:nonlin-stationary} using the change of coordinates $U = (u_1, u_2)$ and arrive at the first-order nonlinear, Hamiltonian ODE
\begin{equation}\label{E:bistable}
\begin{pmatrix}  u_1' \\  u_2'
\end{pmatrix} = G(U) := \begin{pmatrix} u_2 \\ -h(u_1) .
\end{pmatrix}.
\end{equation}

\subsubsection{Parameterization of the stable manifold}
We first consider the stable manifold of the fixed point $U_* = (1,0)$ of \eqref{E:bistable}. One can compute the linearization 
\[
DG(U_*) = \begin{pmatrix} 0 & 1 \\ 1 & 0 \end{pmatrix}.
\]
This fixed point is hyperbolic with eigenvalues $\pm 1$ and it has a local one-dimension stable and a local one-dimensional unstable manifold.  Note that $\varphi_-$ coincides with the stable manifold of $(1, 0)$, whereas $\varphi_+$ coincides with the unstable manifold. Since we are interested in the stable manifold  and $\Omega^\rs = \mu_1=-1$, the manifold resonance condition \eqref{E:manifold-resonance} in this case is given by
\[
\alpha_1 (-1) - ( \pm 1) = 0, \qquad |\alpha | > 1.
\]
We see that this condition is never satisfied, and so there is no manifold resonance. Since $n = 2$ and $m = 1$, the expansion \eqref{E:Pform} is given by
\[
P(\sigma) = \sum_{j = 0}^\infty \begin{pmatrix} p_j^1 \\ p_j^2 \end{pmatrix} \sigma^j,
\]
and when this is plugged into \eqref{E:param-2} one obtains
\[
\begin{pmatrix} \sum_{j \geq 0} p_j^2 \sigma^j \\ -\left( \sum_{j \geq 0} p_j^1 \sigma^j \right)\left(2\sum_{k \geq 0} p_k^1 \sigma^k  - 1\right) \left(1 - \sum_{m \geq 0} p_m^1 \sigma^m  \right) \end{pmatrix} = -\sum_{j = 0}^\infty j \begin{pmatrix} p_j^1 \\ p_j^2 \end{pmatrix} \sigma^j.
\] 
At $\mathcal{O}(\sigma^0)$, we find 
\[
\begin{pmatrix} p_0^2 \\ -p_0^1 \left(2p_0^1   - 1\right) \left(1 -  p_0^1  \right) \end{pmatrix} = \begin{pmatrix} 0 \\ 0 \end{pmatrix}. 
\]
Thus, we can set $P_0 = U_* = (1, 0)$ as expected. At $\mathcal{O}(\sigma^1)$, we find $(p_1^1, p_1^2) = c(1, -1)$ for any constant $c$, and so we can choose $c = -1$; see Remark \ref{rem:levers}. This exactly corresponds to the eigenvector associated with the eigenvalue $-1$ for $DG(U_*)$, as expected. At $\mathcal{O}(\sigma^2)$, we have $(p_2^1, p_2^2) = (1, -2)$. In other words, we have found
\begin{equation}\label{E:bistable-mnfld-exp}
P(\sigma) = \begin{pmatrix}1 -  \sigma + \sigma^2 + \mathcal{O}(\sigma^3) \\ \sigma - 2\sigma^2 + \mathcal{O}(\sigma^3) \end{pmatrix}.
\end{equation}
Above it was noted that the stable manifold of $U_*$ corresponds exactly to 
\begin{equation}\label{E:y-}
\begin{pmatrix} \varphi_-(x) \\ \varphi_-'(x) \end{pmatrix} = \begin{pmatrix} \frac{1}{1 + e^{-x}} \\ \frac{e^{-x}}{(1+e^{-x})^2} \end{pmatrix}.
\end{equation}
As $x \to \infty$, $e^{-x}$ is small, and so we notice that if we set $\sigma = e^{-x}$ the above can be expanded as
\begin{equation}\label{E:relate-varphi-P}
\begin{pmatrix} \frac{1}{1 + e^{-x}} \\ \frac{e^{-x}}{(1+e^{-x})^2} \end{pmatrix} = \begin{pmatrix} \frac{1}{1+\sigma}\\ \frac{\sigma}{(1+\sigma)^2} \end{pmatrix} =  \begin{pmatrix} 1 - \sigma + \sigma^2 + \mathcal{O}(\sigma^3) \\ \sigma - 2\sigma^2 + \mathcal{O}(\sigma^3) \end{pmatrix},
\end{equation}
which matches \eqref{E:bistable-mnfld-exp}, as desired.

\begin{remark}\label{rem:levers}
There appear to be three degrees of freedom so far in this example. One is the choice of constant $c$ that was made in arriving at \eqref{E:bistable-mnfld-exp}, and which is connected with the fact that eigenvectors are not uniquely determined (as they can be multiplied by any nonzero scalar). Another is that any fixed translate of $\varphi_-(x)$, given in \eqref{E:explicit}, is also a solution and also corresponds to an explicit formula for the stable manifold. Finally, there is the choice of the parameterization of the stable manifold, which is reflected in the choice of definition for $\sigma$. It turns out that, in fact, these are all in some sense equivalent. To see this, first suppose we were to have made a different choice $c \neq -1$. We would then arrive at 
\[
P_c(\sigma) = \begin{pmatrix}1 + c \sigma + c^2 \sigma^2 + \mathcal{O}(\sigma^3) \\ -c \sigma - 2 c^2 \sigma^2 + \mathcal{O}(\sigma^3) \end{pmatrix}
\]
instead of \eqref{E:bistable-mnfld-exp}. This is equivalent to rescaling the parameter $\sigma$ using $\sigma = -c \tilde \sigma$. Similarly, in \eqref{E:relate-varphi-P} the parameterization  \eqref{E:bistable-mnfld-exp} is related to the explicit formula for $\varphi_-(x)$ in \eqref{E:explicit}. If we were to instead choose a different translate, say $\tilde \varphi_-(x) = \varphi_-(x + \tau)$, and still relate $x$ and $\sigma$ via $\sigma = e^{-x}$, then we would have
\[
\begin{pmatrix} \tilde \varphi_-(x)  \\ \tilde{\varphi}'_-(x) \end{pmatrix} = \begin{pmatrix} \frac{1}{1 + e^{-(x+\tau)}} \\ \frac{e^{-(x+\tau)}}{(1+e^{-(x+\tau)})^2} \end{pmatrix} = \begin{pmatrix} \frac{1}{1+e^{-\tau}\sigma}\\ \frac{e^{-\tau}\sigma}{(1+e^{-\tau}\sigma)^2} \end{pmatrix} =  \begin{pmatrix} 1 - e^{-\tau}\sigma + e^{-2\tau}\sigma^2 + \dots \\ e^{-\tau}\sigma - 2e^{-2\tau}\sigma^2 + \dots \end{pmatrix},
\]
which is equivalent to the choice $c = -e^{-\tau}$. 

This relationship will become useful below. For example, if we want an approximate invariant manifold
\[
P^N(\sigma) = \sum_{j = 0}^N \begin{pmatrix} p_j^1 \\ p_j^2 \end{pmatrix} \sigma^j,
\]
to be a good (in some sense to be made precise below) approximation in a ball of radius $\delta$, meaning for $\sigma \in B_\delta(0) = (-\delta, \delta) \subseteq \R$, then we could either (i) make $\delta$ small; (ii) make $c$ small and allow $\tilde \sigma \in B_1(0)$; or choose $\tau$ large and allow $\tilde \sigma \in B_1(0)$. Below, this will allow us to make convenient but otherwise arbitrary choices of $c$ and $\delta$ (specifically, $\delta = 1$), and then choose an appropriate translate of the underlying wave so that, a posteriori, the previous two choices lead to a good approximation of the manifold. 
\end{remark}

\subsubsection{Parameterization of the stable/unstable bundles}

If we linearize about $U(x):= (\varphi_-(x), \varphi_-'(x))$, we obtain the following scalar variational equation 
\begin{equation}\label{eq: bistable scalar variational equation}
v_{xx} + h'(\varphi_-(x))v = 0
\end{equation}
and corresponding first-order variational system
\begin{equation}\label{eq: bistable variational equation}
\begin{pmatrix}v_1' \\ v_2' \end{pmatrix} = DG(U(x)) \begin{pmatrix} v_1 \\ v_2 \end{pmatrix} = \begin{pmatrix} 0 & 1 \\ -h'(\varphi_-(x)) & 0 \end{pmatrix} \begin{pmatrix} v_1 \\ v_2 \end{pmatrix}.
\end{equation}  
To investigate whether or not there is a vector bundle resonance, we need to investigate equation \eqref{E:Res_def},  which reduces to $0= \alpha_1 \mu_1  + \mu_j - \mu_i$. 
Since $m = 1$ and the two eigenvalues are $\pm 1$, we see that this condition is satisfied with the choice $\alpha = 2$ when $\mu_i = -1$ and $\mu_j = 1$. This is a resonance of order $2$ since $\alpha = 2$.

\begin{remark} We can observe the consequence of this resonance if we attempt to solve the conjugacy equation \eqref{E:bundle-invar} which assumes invariance of the bundle. To see this, we use the expansion for the stable manifold given in \eqref{E:bistable-mnfld-exp} and we attempt to solve for the unstable (since we found the resonance when $\mu_i = 1 > 0$) bundle $W$ via the expansion
\[
W(\sigma) = \sum_{j = 0}^\infty \begin{pmatrix} w_j^1 \\ w_j^2 \end{pmatrix} \sigma^j.
\]
Plugging these into \eqref{E:bundle-invar}, which in this case is given by
\[
DG(P(\sigma))W(\sigma) = W(\sigma) - DW(\sigma)\sigma,
\]
at $\mathcal{O}(\sigma^0)$ we find that $(w_0^1, w_0^2) = (1, 1)$ (or any constant multiple of this), which is exactly the eigenvector associated with the unstable eigenvalue $1$ of $DG(U_*)$, as expected. Subtracting $DG(P(\sigma))W(\sigma)$ from each side and expanding to  first order in $\sigma$ we get  
\[
\begin{pmatrix} 0\\0 \end{pmatrix} = \begin{pmatrix} (-w^2_1)\sigma + \mathcal{O}(\sigma^2)\\ (6-w_1^1 )\sigma + \mathcal{O}(\sigma^2) \end{pmatrix},
\]
which we can solve by taking $w^1_1=6$ and $w^2_1=0$. Expanding to second order in $\sigma$ we then find
\begin{equation}\label{E:choice-bad}
\begin{pmatrix} 0\\0 \end{pmatrix} = \begin{pmatrix} (-w^1_2-w^2_2)\sigma^2 + \mathcal{O}(\sigma^3)\\ (24-w^1_2-w^2_2)\sigma^2 + \mathcal{O}(\sigma^3) \end{pmatrix},
\end{equation}
which yields an inconsistent system of equations. The fact that we run into trouble at $\mathcal{O}(\sigma^2)$ is exactly due to the fact that the resonance occurred at order $2$.
\end{remark}

\begin{figure}
\centering \includegraphics[width=1\linewidth]{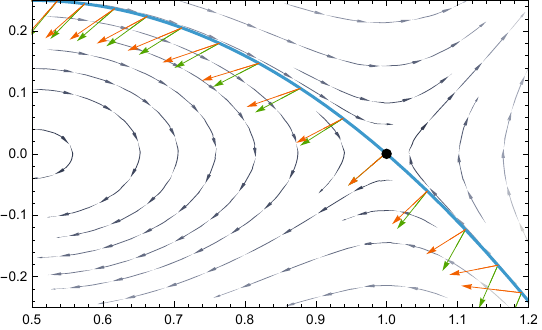}
\caption{Plot of the stable manifold (blue), the analytic bundle (green) and the invariant bundle (orange).}
\label{fig:bundleplot}
\end{figure}

To overcome this obstacle, we aim to find $2 \times 2$ real matrices $\mathcal{W}$, $\mathcal{A}$ that satisfy \eqref{eq:BundleConjugacy}, which we reproduce here:
\[
DG(P(\sigma)) \mathcal{W}(\sigma) = D\mathcal{W}(\sigma)  \Omega^\rs \sigma  + \mathcal{W}(\sigma )\mathcal{A}(\sigma).  
\]
We can write 
\[
\mathcal{W}(\sigma) = \begin{pmatrix} | & | \\ W^\rs(\sigma) & W^\ru(\sigma) \\ | & |  \end{pmatrix}, \qquad \mathcal{A}(\sigma) = \begin{pmatrix} -1 & a^{1,2}(\sigma) \\ 0&1 \end{pmatrix},
\]
where the superscripts $\rs, \ru$ indicate the columns that correspond to the stable and unstable bundle, respectively. We can recover the stable bundle as the derivative of the parameterization of the stable manifold: $W^\rs(\sigma) = \partial_\sigma P(\sigma)$.
Hence  
\begin{equation}\label{E:exp-Ws}
W^\rs(\sigma) =  \begin{pmatrix} \frac{-1}{(1+\sigma)^2} \\ \frac{1-\sigma }{(1+\sigma)^3} \end{pmatrix} =
\begin{pmatrix} -1 +2\sigma -3 \sigma^2 +\mathcal{O}(\sigma^3)\\ 1 -4 \sigma + 9\sigma^2 +\mathcal{O}(\sigma^3) \end{pmatrix}.
\end{equation}
Also note that we have 
\[
DG(P(\sigma)) = \begin{pmatrix} 0&1\\ 1-6 \sigma +12 \sigma^2 +\mathcal{O}(\sigma^3) &0 \end{pmatrix}.
\]
Solving then for the unstable bundle $W^\ru$, we obtain 
\begin{equation}\label{eq:BistableConjugacy}
0=- DG(P(\sigma)) W^\ru(\sigma)   + \partial_\sigma W^\ru(\sigma)(-\sigma)  + W^\ru(\sigma)  + W^\rs(\sigma)a^{1,2}(\sigma).  
\end{equation}
Recall that $W^\ru(0)$ should equal the eigenvector. So let us write 
\[
W^\ru(\sigma) = \begin{pmatrix} 1+ w_1^1 \sigma + w^1_2 \sigma^2 + \mathcal{O}(\sigma^3) \\ 1+ w^2_1 \sigma + w_2^2 \sigma^2 + \mathcal{O}(\sigma^3) \\ \end{pmatrix}, \qquad
a^{1,2}(\sigma) = a^{1,2}_1 \sigma + a^{1,2}_2 \sigma^2 + \mathcal{O}(\sigma^3).
\]
Plugging this into \eqref{eq:BistableConjugacy} and expanding to first order in $\sigma$ we get  
\[
\begin{pmatrix} 0\\0 \end{pmatrix} = \begin{pmatrix} (-w^2_1-a_1^{1,2})\sigma + \mathcal{O}(\sigma^2)\\ (6-w_1^1 +a_1^{1,2})\sigma + \mathcal{O}(\sigma^2) \end{pmatrix},
\]
which we can solve by taking $w^1_1=6$, $w^2_1=0$ and $a^{1,2}_1 = 0$. Note that the choice $a^{1,2}_1 = 0$ implies that $a^{1,2}(\sigma) = \mathcal{O}(\sigma^2)$, which is the order of the resonance. To second order in $\sigma$ we then find
\begin{equation}\label{E:choice}
\begin{pmatrix} 0\\0 \end{pmatrix} = \begin{pmatrix} (-w^1_2-w^2_2-a^{1,2}_2)\sigma^2 + \mathcal{O}(\sigma^3)\\ (24-w^1_2-w^2_2+a^{1,2}_2)\sigma^2 + \mathcal{O}(\sigma^3) \end{pmatrix}.
\end{equation}
The solution to this requires $a^{1,2}_2 = - 12$ and $ w^1_2 + w^2_2 =12 $.  We are free to make a choice here, so we choose $w^1_2 =6 $ and $w^2_2 = 6$; we comment on this choice just after  \eqref{eq:Explicit_Wu}. Then we have that 
\begin{equation}\label{E:exp-Wu}
W^\ru(\sigma) = \begin{pmatrix} 1+ 6 \sigma +  6 \sigma^2 + \mathcal{O}(\sigma^3) \\  1+ 0\sigma+ 6 \sigma^2 + \mathcal{O}(\sigma^3) \\  \end{pmatrix}.
\end{equation}
Since no more resonances may occur, we may take all the higher order coefficients of $a^{1,2}(\sigma)$ to be zero. Thereby $a^{1,2}(\sigma) = -12 \sigma^2$, and hence 
\[
\mathcal{A}(\sigma)= \begin{pmatrix} -1 & -12 \sigma^2 \\ 0&1 \end{pmatrix}.
\]

\subsubsection{Closed forms of the stable/unstable bundles}
We can relate this back to the general framework that was presented above, and to the fact that we can explicitly solve the variational equation for this example. 

First, we  connect the stable solutions $V^\rs(x)$ of \eqref{eq: bistable variational equation} to the stable bundle $ W^\rs(\sigma)$. 
We know that $\varphi_-$ solves \eqref{E:bistable}, so its derivative $(\varphi_-'(x), \varphi_-''(x))$ is a solution of \eqref{eq: bistable variational equation}, and it is given explicitly by
\begin{equation}\label{E:first-var-soln} 
V^\rs(x)= \begin{pmatrix}\varphi_-'(x) \\ \varphi_-''(x) \end{pmatrix} = e^{-x}\begin{pmatrix} \frac{1}{(1+e^{-x})^2} \\  \frac{e^{-x} - 1}{(1+e^{-x})^3} \end{pmatrix}. 
\end{equation}
Thus, ignoring the factor $e^{-x}$, this solution converges as $x \to +\infty$ to an eigenvector of $DG(U_*)$ associated with the eigenvalue $-1$. Up to scalar multiples, this is the unique solution with this property, and so we denote  
\begin{equation}\label{E:first-var-soln-new}
V^\rs(x)=  \begin{pmatrix} v^\rs(x) \\ (v^\rs)'(x) \end{pmatrix}, \qquad v^\rs(x) = \varphi_-'(x).
\end{equation}
We expect this to correspond to the stable bundle $W^\rs(\sigma)$ that we found above. 

In equation \eqref{E:exp-Ws} we saw that
\[
W^\rs(\sigma) =  \begin{pmatrix} -1 +2\sigma -3 \sigma^2 +\mathcal{O}(\sigma^3)\\ 1 -4 \sigma + 9\sigma^2 +\mathcal{O}(\sigma^3) \end{pmatrix}.
\]
This exactly corresponds to the expansion of the solution given in \eqref{E:first-var-soln}, if we ignore the prefactor $e^{-x}$ and recall that $\sigma = e^{-x}$. In particular, 
\begin{equation}\label{E:Ws-reln}
W^\rs(\sigma) = e^{x}V^\rs(-\log \sigma).
\end{equation}

Next, we can solve for a second independent solution $ v^{\ru}(x)$ of the variational equation such that $V^\ru = (v^\ru, (v^\ru)')$ converges as $x \to +\infty$ (after ignoring the ambient exponential growth factor) to an eigenvector of $DG(U_*)$ associated with the eigenvalue $+1$. Doing so, for example using the reduction of order method, produces $(\tilde v^\ru(x), (\tilde v^\ru)'(x))$, where  
\begin{align*}
\tilde v^{\ru}(x) &= \frac{1}{(1+e^{-x})^2} \left[ e^{x} +8  -8e^{-2x} - e^{-3x} + 12 x e^{-x} \right] \\
&=  \frac{e^x}{(1+e^{-x})^2} \left[1 +8e^{-x}  -8e^{-3x} - e^{-4x} \right] + 12x \varphi_-'(x).
\end{align*}
Again we can multiply this by any scalar and preserve this desired asymptotic property, but we can also add to it any scalar multiple of $(v^\rs(x), (v^\rs)'(x))$ and preserve this property; see Remark \ref{rem:NonuniqueUnstableBundle}. Thus, we denote
\begin{equation} \label{E:second-var-soln}
V^\ru(x)   = \begin{pmatrix} v^\ru(x) \\ (v^\ru)'(x) \end{pmatrix}, \qquad v^\ru(x) = \tilde v^\ru(x) + c v^\rs(x).
\end{equation}

This unstable solution in  \eqref{E:second-var-soln} may be used to define an  unstable bundle, similar to that found in \eqref{E:exp-Wu}. 
Indeed, notice that $V^\ru$ can be expanded via $ \sigma = e^{-x}$  as   
\begin{align} \label{eq:SecondSolutionExpand}
V^\ru(x) &=  e^x \left\{ \begin{pmatrix} \frac{1+8 \sigma - 8 \sigma^3 - \sigma^4-12 \sigma^2 \log \sigma }{(1+\sigma)^2} \\ \frac{1+3\sigma +28 \sigma ^2+28 \sigma ^3+3 \sigma ^4+\sigma ^5+12 \sigma ^2 \log (\sigma ) -12 \sigma ^3 \log(\sigma )}{ (1+\sigma)^3} \end{pmatrix} + c \sigma^2 \begin{pmatrix} \frac{1}{(1+\sigma)^2} \\ \frac{\sigma - 1}{(1+\sigma)^3} \end{pmatrix}\right\}\\
&=  e^x\left\{ \begin{pmatrix}1 + 6 \sigma + (c - 13) \sigma^2 - 12 \sigma^2 \log \sigma + \mathcal{O}(\sigma^3) \\ 1 + 0 \sigma + (-c + 25) \sigma^2 + 12 \sigma^2 \log \sigma + \mathcal{O}(\sigma^3)\end{pmatrix}\right\}  \nonumber
\end{align}
The free constant $c$ in the definition of $v^\ru$ can be chosen (e.g. $c=19$) so that, with the exception of the logarithmic term,  and the growth factor $e^x$, this exactly matches the first component in 
$W^\ru$ from \eqref{E:exp-Wu}. The presence of the logarithmic term means that the expansion of $V^\ru$ with respect to $\sigma$ is not analytic. 
However, as this unstable bundle was constructed from a solution to the variational equation, it is by construction invariant. 
This illustrates the incompatibility of invariance and analyticity of vector bundles in the presence of a bundle resonance.

Finally, we aim to obtain a closed form expression for an unstable bundle $ W^\ru$ satisfying the conjugacy equation \eqref{eq:BundleConjugacy}. 
To that end,  recall from \eqref{eq:Vtilde_Def} that any solution $V(x)$ of \eqref{eq: bistable variational equation}  may be rewritten using the change of variables $V(x) = \mathcal{W}(e^{\Omega^\rs x} \sigma_0) \tilde{V}(x)$,  where for this example
\[
\mathcal{W}(e^{-x} \sigma_0) = \begin{pmatrix} | & | \\ W^\rs(e^{-x}\sigma_0) & W^\ru(e^{-x} \sigma_0) \\ | & | \end{pmatrix}.
\]
The function $\tilde{V}(x)$ is itself a solution of the differential equation  \eqref{eq:NormalFormBundleEquation}, which for the bistable equation is given by
\begin{equation}\label{E:Vtilde-ex}
\tilde V' = \begin{pmatrix} -1 & -12(e^{-x}\sigma_0)^2 \\ 0 & 1 \end{pmatrix} \tilde V.
\end{equation}

Set $\sigma_0 = 1$; the computations for general $\sigma_0$ are similar. One solution of \eqref{E:Vtilde-ex} is any scalar multiple of $\tilde V^\rs(x) = (e^{-x}, 0)$, and this corresponds to 
\[
\mathcal{W}(e^{-x})(-\tilde V^\rs(x)) = - e^{-x} W^\rs(e^{-x}) =  V^\rs(x)
\]
via \eqref{E:exp-Ws} and \eqref{E:Ws-reln}.  Another linearly independent solution of \eqref{E:Vtilde-ex} is  $\tilde V^\ru(x) = ( - 12 x   e^{-x},   e^{x})$. 

Now we solve for $W^\ru$. First, by \eqref{eq:Vtilde_Def} we may write $V^\ru$ from \eqref{E:second-var-soln} as the linear combination 
\begin{align} \label{eq:Wu-LinearExpansion}
V^\ru(x) &=\mathcal{W}(e^{-x} ) \left( k_1 \tilde V^\rs(x)  +k_2 \tilde V^\ru(x) \right)
\end{align}
Multiplying the right hand side of \eqref{eq:Wu-LinearExpansion}  by $e^{-x}$ and substituting in $e^{-x} = \sigma$, we obtain   
\begin{align} \nonumber
e^{-x}		\mathcal{W}(e^{-x} ) \left( k_1 \tilde V^\rs(x)  +k_2 \tilde V^\ru(x) \right) &= 
\begin{pmatrix} | & | \\ W^\rs( \sigma) & W^\ru(  \sigma) \\ | & | \end{pmatrix}
\begin{pmatrix}
	k_1 \sigma^2 -12 k_2 \sigma^2 (-\log \sigma )  \\
	k_2 
\end{pmatrix} \\
&=  \sigma^2 \left( k_1 + 12 k_2 \log \sigma  \right)
\begin{pmatrix} 
	\frac{-1}{(1+\sigma)^2} \\
	\frac{1-\sigma }{(1+\sigma)^3} 
\end{pmatrix} 
+ k_2 W^\ru(\sigma),
\label{eq:WU-solve-RHS}
\end{align}
where we used the explicit form of  $W^\rs( \sigma)$ from  \eqref{E:exp-Ws}. 
Likewise, multiplying the left hand side of \eqref{eq:Wu-LinearExpansion}  by $e^{-x}$ and substituting in $e^{-x} = \sigma$, as in  \eqref{eq:SecondSolutionExpand}  we obtain 
\begin{align}\label{eq:WU-solve-LHS}
		e^{-x} V^{\ru}(x)&=
		\begin{pmatrix}
			\frac{1	+8 \sigma - 8 \sigma^3 - \sigma^4-12 \sigma^2 \log \sigma }{(1+\sigma)^2} \\
			\frac{
				1+3\sigma +28 \sigma ^2+28 \sigma ^3+3 \sigma ^4+\sigma ^5+12 \sigma ^2 \log (\sigma ) -12 \sigma ^3 \log
				(\sigma )
			}{ (1+\sigma)^3}
		\end{pmatrix}
		+ c \sigma^2 
		\begin{pmatrix} \frac{1}{(1+\sigma)^2} \\ \frac{\sigma - 1}{(1+\sigma)^3} \end{pmatrix}.
\end{align}
Equating \eqref{eq:WU-solve-RHS} and \eqref{eq:WU-solve-LHS}, we see that we need to take $k_2 = 1$ in order to cancel the log terms.  
Moreover, equating \eqref{eq:WU-solve-RHS} and \eqref{eq:WU-solve-LHS}  gives us an explicit form of $W^\ru(\sigma)$:
 \begin{align} \label{eq:Explicit_Wu}
 		\begin{pmatrix}
 		\frac{1	+8 \sigma - 8 \sigma^3 - \sigma^4  }{(1+\sigma)^2} \\
 		\frac{
 			1+3\sigma +28 \sigma ^2+28 \sigma ^3+3 \sigma ^4+\sigma ^5 
 		}{ (1+\sigma)^3}
 	\end{pmatrix}
 + (c+k_1)\sigma^2 
 \begin{pmatrix} \frac{1}{(1+\sigma)^2} \\ \frac{-1+\sigma }{(1+\sigma)^3} \end{pmatrix}
 &= W^\ru(\sigma)
 \end{align}
Note that the free parameters $ k_1$ and $ c$ allow for the addition of the stable bundle, as described in Remark  \ref{rem:NonuniqueUnstableBundle}.

Furthermore, we can confirm that this solves the conjugacy equation  \eqref{eq:BundleConjugacy}. With the explicit form of $P(\sigma)$ given in \eqref{E:relate-varphi-P}, equation  \eqref{eq:BundleConjugacy} becomes
\begin{equation}\label{E:explicit-1}
\begin{pmatrix} 0 & 1 \\ \frac{\sigma^2 - 4 \sigma + 1}{(1+\sigma)^2} & 0 \end{pmatrix} \mathcal{W}(\sigma) = - \sigma \partial_{\sigma} \mathcal{W}(\sigma) + \mathcal{W}(\sigma)\begin{pmatrix} -1 & a^{1,2}(\sigma) \\ 0 & 1 \end{pmatrix}
\end{equation}
with $a^{1,2}(\sigma) = -12 \sigma^2$. This is satisfied for solutions $ W^\rs(\sigma) ,W^\ru(\sigma)$ as in \eqref{E:exp-Ws}   and  \eqref{eq:Explicit_Wu}. 

\begin{remark} \label{rem:Invariant-Not-Analytic} Despite all the freedom one has in choosing the various scalar constants, one is always forced to include a term of the form $x \varphi'_-(x)$ in the solution to the variational equation that corresponds to the unstable bundle. This term is not analytic in $\sigma$. However, under the change of coordinates given in Lemma \ref{L:conjugacy}, this lack of analyticity is removed from the bundle $\mathcal{W}$ and placed instead into the solution $\tilde M$. In the bistable example, this appears via the fact that the solution $\tilde V^\ru$ of \ref{E:Vtilde-ex} is not analytic in $\sigma$. 

One can also think about this in terms of whether one wants to prioritize an invariant unstable bundle, or an analytic one. This is perhaps best illustrated the example from Remark \ref{rem:NonuniqueUnstableBundle} with $\lambda = 0$, which demonstrates the choice of invariant versus non-invariant bundles, without the complication of analyticity versus non-analyticity. Here the stable and unstable manifolds of the origin are given exactly by the eigendirections, and the variational equation about the origin is the same as the original equation. As $x \to + \infty$, there is a unique solution (up to scalar multiple) that converges to an eigenvector associated with the eigenvalue $-1$ (after removing the ambient decay $e^{-x}$) and it is given by $(e^{-x}, 0)$. Hence, this is the stable bundle, and for all $x$ we see that it lies in $\mathrm{span}\{(1, 0)\}$. But any solution of the form $(c_1 e^{-x}, e^{x})$ is independent and, after removing the ambient growth of $e^x$, converges to an eigenvector associated with the eigenvalue $+1$. But only for the choice $c_1 = 0$ is this invariant in the sense that for any $x$ we would find that it lies in $\mathrm{span}\{(0, 1)\}$. Otherwise, it lies in $\mathrm{span}\{(c_1 e^{-2x}, 1)\}$, which varies as $x$ varies.
	
In the bistable example, choosing an invariant unstable bundle is tantamount to choosing $c_1 = 0$, but with the consequence that the bundle is then not analytic. On the other hand, one can require that the bundle be analytic, which corresponds to choosing $c_1 \neq 0$. For our purposes, we must prioritize analyticity.
\end{remark}


\section{Proving the conjugate points lie in a compact domain}\label{Ch: Lpm}

In this section we identify numerically computable constants $L^\pm_\mathrm{conj}$ and prove that $\mathbb{E}^\mathrm{u}(x; 0) \cap \ell_\mathrm{sand} = \{0\}$ for all $x \notin [-L^-_{\mathrm{conj}}, L^+_{\mathrm{conj}}]$. We begin by recalling some facts about our Lagrangian path $\E^{\mathrm{u}}_-(x; \lambda)$. We will be primarily interested in the case where $\lambda = 0$, and we denote this case via $\mathbb{E}^{\mathrm{u}}_-(x; 0) = \mathbb{E}^{\mathrm{u}}_-(x)$. This is the unstable subspace of solutions to  \eqref{E:eval} with $\lambda = 0$, which we reproduce here for convenience:
\begin{equation}\label{E:eval-0}
Q' = B(x)Q, \qquad B(x) = \begin{pmatrix} 0 & 0 & 0 & 1 \\ 0 & 0 & 1 & - 2\\ -1 + f'(\varphi(x)) & 0 & 0 & 0 \\ 0 & 1 & 0 & 0
\end{pmatrix}.
\end{equation}
We will also be interested in the dynamics of the asymptotic system $Q' = B_\infty Q$ with 
\begin{equation}\label{E:eval-0-infty}
B_\infty := \lim_{|x| \to \infty} B(x) = \begin{pmatrix}0 & 0 & 0 & 1 \\ 0 & 0 & 1 & - 2\\ -1 + f'(0) & 0 & 0 & 0 \\ 0 & 1 & 0 & 0
\end{pmatrix}.
\end{equation}
To determine $L^\pm_\mathrm{conj}$, we will characterize the asymptotic behavior of solutions of \eqref{E:eval-0}. For large $|x|$, the evolution of $\E^{\mathrm{u}}_-(x)$ is determined by the asymptotic matrix $B_\infty$, because of the fact that 
\begin{equation}\label{E:B-decay}
\|B(x) - B_\infty \| \leq K_Be^{-C_B|x|}
\end{equation}
for some $K_B, C_B > 0$. We note that these constants can be determined explicitly from the underlying solution $\varphi$, which we are able to do using validated numerics; see also Theorem \ref{thm: compute L-}.

Since we are interested in obtaining numerically verifiable bounds, we would like to be clear about which norms are being used. For vectors $z \in \C^d$, we use the standard norm $\|z\|^2 = z \cdot \bar z$. For matrices, we use the norm $\|A\| = \sqrt{\mu_\mathrm{max}(A^*A)}$, where $\mu_\mathrm{max}$ denotes the largest eigenvalue of the matrix $A^*A$, which is real and nonnegative. Note that this matrix norm coincides with the operator norm relative to the vector norm we are using, in the sense that $\|A\| = \sup\{ \|Az\|: \|z\| = 1\}$, and so $\|Az\| \leq \|A\| \|z\|$ for all $z$. 

Since $B_\infty$ has complex eigenvalues, it will be convenient to work with complex-valued solutions of \eqref{E:eval-0}. Therefore, we define
\[
\C \E^{\mathrm{u}}_-(x):= \big\{U(x) \in C^1(\R, \C^4) \ : \ U(x) \text{ solves \eqref{E:eval-0} and } \|U(x)\| \to 0 \text{ as } x \to - \infty \big\}
\]
and
\begin{equation}\label{E:sand-c}
\C\ell_\mathrm{sand}  = \left\{ \begin{pmatrix}  0 & 0 \\ 1 & 0 \\ 0 & 1 \\ 0 & 0 \end{pmatrix}u \ : \ u \in \C^2 \right\}. 
\end{equation}
 
\begin{lemma}\label{lem:real-complex-sand} Let $\E^{\mathrm{u}}_-(x)$ and $\C \E^{\mathrm{u}}_-(x)$ be the real- and complex-valued subspaces of unstable solutions to \eqref{E:eval-0}, respectively. Then $\E^{\mathrm{u}}_-(x) \cap \ell_\mathrm{sand}  \neq \{0\}$ if and only if $\C \E^{\mathrm{u}}_-(x) \cap \C\ell_\mathrm{sand} \neq \{0\}$.  
\end{lemma}
\begin{proof} 
Since $B(x)$ is real-valued, this follows from writing $\C \E^{\mathrm{u}}_-(x) = \text{span} \big\{U_1(x), \bar U_1(x) \big\}$ for an appropriate solution $U_1$ of \eqref{E:eval-0}, noting that a basis for $\E^{\mathrm{u}}_-(x)$ is then given by 
\[
\text{span}\{V_1(x), V_2(x)\} = \text{span}\left\{\text{Re}\big(U_1(x) \big), \text{Im}\big( U_1(x) \big) \right\},
\]
and a direct calculation.
\end{proof}
Thus, we can move freely between complex-valued solutions and real-valued solutions when discussing intersections of the complex and real valued unstable subspaces with the reference planes $\C\ell_\mathrm{sand} $ and $\ell_\mathrm{sand}$ respectively.

We also note that the (complex) eigenvalues and (complex-valued) eigenvectors of the asymptotic matrix $B_\infty$ are given by 
\begin{equation}\label{E:B-infty-evals}
\mu_1^\rs = -\sqrt{\rho}e^{i\theta/2}, \qquad \mu_1^\ru = \sqrt{\rho}e^{i\theta/2},  \qquad 
\mu_2^\rs = -\sqrt{\rho}e^{-i\theta/2}, \qquad \mu_2^\ru = \sqrt{\rho}e^{-i\theta/2},
\end{equation}
and
\begin{equation}\label{E:B-infty-evecs}
\begin{split} \check{\mathcal V}^\ru & = \begin{pmatrix} | & | \\ \check V^\ru_1 & \check V^\ru_2 \\ | & | 
\end{pmatrix} =  \begin{pmatrix}\frac{1}{\rho} e^{-i\theta} & \frac{1}{\rho} e^{i\theta} \\ 1 & 1 \\ \sqrt \rho e^{i\theta/2} + \frac{2}{\sqrt \rho} e^{-i\theta/2} &  \sqrt \rho e^{-i\theta/2} + \frac{2}{\sqrt \rho} e^{i\theta/2} \\  \frac{1}{\sqrt \rho} e^{-i\theta/2} &  \frac{1}{\sqrt \rho} e^{i\theta/2}
\end{pmatrix}\\ 
\check{\mathcal V}^\rs & = \begin{pmatrix} | & | \\ \check V^\rs_1 & \check V^\rs_2 \\ | & | 
\end{pmatrix}= \begin{pmatrix}\frac{1}{\rho} e^{-i\theta} & \frac{1}{\rho} e^{i\theta} \\ 1 & 1 \\  -\left( \sqrt \rho e^{i\theta/2} + \frac{2}{\sqrt \rho} e^{-i\theta/2} \right) &  -\left( \sqrt \rho e^{-i\theta/2} + \frac{2}{\sqrt \rho} e^{i\theta/2} \right)  \\  -\frac{1}{\sqrt \rho} e^{-i\theta/2} &  -\frac{1}{\sqrt \rho} e^{i\theta/2}
\end{pmatrix},
\end{split}
\end{equation}
where
\begin{equation}\label{E:def-r-theta}
\tan\theta = - \sqrt{-f'(0)}, \qquad \theta \in \left(\frac{\pi}{2}, \pi \right), \qquad \rho = \sqrt{1 + f'(0)} = \sqrt{1 + \hat \mu}.
\end{equation}
The superscripts $\mathrm{s}$ and $\mathrm{u}$ indicate those that correspond to the stable and unstable subspaces, respectively, and we use the notation $\check V^{\ru, \rs}$ to distinguish from the eigenvectors $\hat V^{\ru, \rs}$ in \S\ref{S:background} that correspond to the system \eqref{E:exist}, which is written in different coordinates.  We note that for $ i=1,2$ then 
\begin{equation}\label{E:evec-bound}
\| \check V_i^{\ru/\rs} \|^2 = 1 + 4 \cos \theta + \frac{1}{\rho^2} +  \frac{5}{\rho} + \rho \leq 1 + \frac{1}{\rho^2} +  \frac{5}{\rho} + \rho,
\end{equation}
since $\theta \in ( \pi/2,\pi)$. 

In \S\ref{S:L-} we determine $L^-_{\mathrm{conj}}$, and in \S\ref{S:L+} we determine $L^+_{\mathrm{conj}}$.


\subsection{Determination of $L^-_{\mathrm{conj}}$}\label{S:L-}

Recall that we seek an $L^-_{\mathrm{conj}}$ such that for all $x < -L^-_{\mathrm{conj}}$, we have that $\mathbb E^\ru_-(x) \cap \ell_\mathrm{sand} = \{0\}$. We derive a numerical condition on $L^-_{\mathrm{conj}}$ such that if this condition is met, then the previous statement is true. We will construct this numerical condition by characterizing the basis solutions of $\C\E^{\mathrm{u}}_-(x)$ that limit to the unstable eigendirections of $B_\infty$. First, we rewrite equation \eqref{E:eval-0} as 
\begin{equation}\label{E:decomp-eval}
 Q' = \big[B_\infty + \tilde B(x)\big]Q
\end{equation}
where 
\begin{equation}\label{E:Btilde}
\tilde B(x) := B(x) - B_\infty = \begin{pmatrix} 0 & 0 & 0 & 0\\ 0 & 0 & 0 & 0 \\ f'(\varphi(x)) - f'(0)& 0 & 0 & 0 \\
0 & 0 & 0 & 0
\end{pmatrix}.
\end{equation}
We seek two solutions, $V^\ru_{1,2}: \R \to \C^4$ of \eqref{E:decomp-eval} that are asymptotic to the complex valued vectors $\check V^\ru_{1,2}$ in \eqref{E:B-infty-evecs} as $x \to -\infty$. Fix $j$ and suppose that $V_j^\ru(x)$ is a solution to \eqref{E:decomp-eval}. Recall that $\mu_{1,2}^\ru$ are defined in \eqref{E:B-infty-evals} and write 
\begin{equation}\label{eq: unstable solns L-}
V_j^\ru(x) = e^{\mu_j^\ru x}W_j^{-,\ru}(x), \qquad j = 1,2.
\end{equation}
The goal is to decompose $W^{-,\ru}_j(x)$ into $\check V^\ru_j$ plus a part that decays exponentially as $x \to - \infty$. To this end, note that $W_j^{-,\ru}(x)$ satisfies the ODE
\begin{equation}\label{eq: decomposed and scaled first order system} 
(W^{-,\ru})' = \big[ B_\infty - \mu_j^\ru I \big]W^{-,\ru} + \tilde B(x)W^{-,\ru}.
\end{equation}
We will use an exponential dichotomy associated with the constant matrix $B_\infty - \mu_j^\ru I$ to construct $W_j^{-,u}$, similar to that done in \cite{SandstedeKapitula98,BeckJaquette22}. The eigenvalues of $B_\infty - \mu_j^\ru I$ are those of $B_\infty$ but shifted by $\mu_j^\ru$. Thus, there are now two eigenvalues with zero real part, and two with negative real part. This will allow us to isolate the solution to \eqref{E:eval-0} characterized by growth rate $\mu_j^\ru$.

Let $\mathcal{P}^{\mathrm{cu}}$ be the projection onto the eigenspace of $B_\infty - \mu_j^\ru$ corresponding to eigenvalues with real equal to zero, and let $\mathcal{P}^\rs$ the projection onto the eigenspace corresponding to eigenvalues of $B_\infty - \mu_j^\ru$ with real part strictly less than zero. For any fixed $L^-_{\mathrm{conj}} > 0$, define the Banach space $X = L^\infty((-\infty, -L^-_{\mathrm{conj}}], \C^4)$ and consider the map $\mathcal{F}_-: X \to X$ defined by
\begin{equation}\label{E:A-}
 \begin{split}
 \mathcal F_-(W)(x) := \check V_j^\ru & + \int_{-\infty}^x \exp \big\{ (B_\infty - \mu_j^\ru I)(x-y) \} \mathcal{P}^{\rs} \tilde B(y) W(y) \ dy \\
& \qquad - \int_{x}^{-L^-_{\mathrm{conj}}} \exp \big\{ (B_\infty - \mu_j^\ru I) (x-y) \big\} \mathcal{P}^{\mathrm{cu}} \tilde B(y) W(y) \ dy.
\end{split}
\end{equation}
\begin{lemma} Any fixed point of $\mathcal F_-$ is a solution to Equation \eqref{eq: decomposed and scaled first order system}. 
\end{lemma}
\begin{proof}
This follows directly by differentiating $\mathcal F_-(W)(x)$ and using the facts that $\mathcal{P}^{\mathrm{cu}} + \mathcal{P}^{\mathrm{cs}} = I$ and $(B_\infty - \mu_j^\ru I)V_j^\ru = 0$.
\end{proof}

Let $\check{\mathcal{V}}$ be a matrix whose columns are comprised of the eigenvectors of $B_\infty$, which we note are also the eigenvectors of $B_\infty - \mu_j^\ru I$, and define 
\begin{equation}\label{E:KQ-}
K = \|\check{\mathcal{V}}^{-1}\|\|  \check{\mathcal{V}}\|.
\end{equation}
Recall from the properties of exponential dichotomies \cite{Coppel78} that there exists an $\hat \eta > 0$ such that
\begin{equation}\label{eq: exp dich decay}
\begin{split} \left\| e^{(B_\infty- \mu_j^\ru)(x-y)}\mathcal P^{\rs} \right\| & \leq K e^{-\hat \eta (x-y)}, \quad y \leq x \leq 0 \\
\left\|e^{(B_\infty - \mu_j^\ru I)(x-y)} \mathcal P^{\mathrm{cu}} \right\| & \leq K , \quad x \leq y \leq 0.
\end{split}
\end{equation}

\begin{proposition}\label{prop: unstable vec convergence} Fix $L^-_{\mathrm{conj}} > 0$ and define 
\[
\tau_{\mathcal F_-}(L^-_{\mathrm{conj}}) = \frac{K K_B}{C_B} e^{-C_B L^-_{\mathrm{conj}}},
\]
where $C_B$ and $K_B$ are as in \eqref{E:B-decay}. If $\tau_{\mathcal F_-}(L^-_{\mathrm{conj}}) < 1$, $\mathcal F_-$ is a contraction on $X = L^\infty((-\infty, -L^-_{\mathrm{conj}}], \C^4)$ and its unique fixed point $W^{-, \ru}_j(x)$ satisfies 
\[
\|W^{-, \ru}_j(\cdot) - \check V_j^{\ru}\|_X \leq \frac{\tau_{\mathcal F_-}(L^-_{\mathrm{conj}})}{1-\tau_{\mathcal F_-}(L^-_{\mathrm{conj}})}\| \check V_j^{\ru} \|.
\]
\end{proposition}
\begin{proof} The proof is similar to that of \cite{SandstedeKapitula98}[Lemma 2.2] and \cite{BeckJaquette22}[Proposition 2.3].
\end{proof}

Note that, although there are two eigenvalues of $B_\infty - \mu_j^\ru I$ with zero real part, the fact that $\mathcal{F}_-$ is defined using $\check V_j^\ru$ allows us to pick out the unique solution that is asymptotic to that direction. Doing so with $j = 1,2$ then allows us to construct the solutions 
\[
V^{-,\ru}_j (x) = e^{\mu_j^{\ru} x} W^{-,\ru}_j(x) = e^{\mu_j^\ru x}[\check V_j^\ru + (W^{-,\ru}_j(x) -\check V_j^\ru)], \qquad j = 1,2
\]
and see that $V^{-,\ru}_j(x)$ are asymptotic to $e^{\mu_j^\ru x}\check V_j^\ru$ as $ x \to -\infty$. Moreover, 
\begin{equation}\label{E:span-reduce}
\C \E^{\mathrm{u}}_-(x) = \text{span} \left\{ V_1^{-,\ru}(x), V_2^{-,\ru}(x) \right\} = \text{span} \left\{ W^{-,\ru}_1(x), W^{-,\ru}_2(x) \right\},
\end{equation}
and Proposition \ref{prop: unstable vec convergence} tells us that $\C \E^{\mathrm{u}}_-(-L^-_{\mathrm{conj}})$ will be close to the asymptotic unstable subspace 
\[
\C \E_{-\infty}^\ru = \text{span} \left\{\check V^\ru_1, \check V^\ru_2 \right\}.
\]
A direct calculation shows that the real valued subspace corresponding to the above does not intersect $\ell_\mathrm{sand}$, $\E^{\ru}_{-\infty} \cap \ell_\mathrm{sand} = \{0\}$, and so we expect $\E^{\mathrm{u}}_-(-L^-_{\mathrm{conj}}) \cap \ell_\mathrm{sand} = \{0\}$ for sufficiently large $L^-_{\mathrm{conj}}$. Via Lemma \ref{lem:real-complex-sand}, we see we can also expect $\C \E^{\mathrm{u}}_-(-L^-_{\mathrm{conj}}) \cap \C\ell_\mathrm{sand} = \{0\}$ for sufficiently large $L^-_{\mathrm{conj}}$, as well. Before proving Proposition \ref{prop: L-_alt}, which assigns a numerical condition to this, we first introduce some matrix notation. 

As with \eqref{E:B-infty-evecs}, we will denote matrices with a calligraphic letter that corresponds to the nomenclature of the columns. For example, denote
\begin{equation}\label{eq: definition of matrix notation}
\mathcal W^{-,\ru}(x)  = \begin{pmatrix} | & | \\ W_1^{-,\ru}(x) & W_2^{-,\ru}(x) \\ | & |
\end{pmatrix} \in \C^{2 \times 4}, \qquad \check{\mathcal V}^\ru  = \begin{pmatrix} | & | \\ 
\check V^\ru_1 &  \check V^\ru_2 \\ | & | \end{pmatrix} \in \C^{2 \times 4}.
\end{equation}
Additionally, we will be particularly interested in the first and fourth row of matrices such as those in \eqref{eq: definition of matrix notation}. This is because the first and fourth components of a frame matrix in $\Lambda(2)$ determine if a particular Lagrangian plane intersects $\ell_\mathrm{sand}$ (see \ref{E:sand-c}). For any matrix $\mathcal{M}$,  we denote the matrix consisting of the first and fourth rows of $\mathcal{M}$ as $\mathcal{M}_{1;4}$. For example, 
\[
\check{\mathcal V}^\ru_{1;4} = \begin{pmatrix} (\check V_1^\ru)_1 & (\check V_2^\ru)_1 \\ (\check V_1^\ru)_4 & (\check V_2^\ru)_4
\end{pmatrix} = \begin{pmatrix} \frac{1}{\rho}e^{-i\theta} & \frac{1}{\rho}e^{i\theta} \\  \frac{1}{\sqrt \rho} e^{-i\theta/2} &  \frac{1}{\sqrt \rho} e^{i\theta/2} \end{pmatrix} \in \C^{2 \times 2}.
\]
With this notation in place, we are now ready to prove the following proposition. Recall the definitions of $\rho$ and $\theta$ in \eqref{E:def-r-theta}. 
\begin{proposition}\label{prop: L-_alt} Suppose that 
\[
\|W_j^{-,\ru}(x) - \check V_j^\ru\|_X \leq \epsilon \|\check V_j^\ru\|.
\]
If $L^-_{\mathrm{conj}}$ is chosen so that 
\begin{equation}\label{eq: Determination of L- condition}
\epsilon < \frac{1}{8 \rho^{3/2}},
\end{equation}
then $\E^{\mathrm{u}}_-(x) \cap \ell_\mathrm{sand} = \{0\}$ for all $x \leq - L^-_{\mathrm{conj}}$.  
\end{proposition}
\begin{proof}
Recall from \eqref{E:span-reduce} that $\C \E^{\mathrm{u}}_-(x) = \text{span} \big\{ W_1^{-,u}(x),  W_2^{-,u}(x)\}$. Suppose by way of contradiction that $x_* \in (-\infty, -L^-_{\mathrm{conj}}]$ is a conjugate point, so there exist $c_1, c_2 \in \C$ such that
\begin{equation}\label{eq: conj pt condition L-}
\mathcal W^{-,u}(x_*) \begin{pmatrix} c_1 \\ c_2 \end{pmatrix} = \begin{pmatrix} | & | \\ W_1^{-,u}(x_*) & W_2^{-,u}(x_*) \\ | & |\end{pmatrix} \begin{pmatrix} c_1 \\ c_2\end{pmatrix} \in \text{span}\left\{ \begin{pmatrix} 0 & 0 \\ 1 & 0 \\ 0 & 1 \\ 0 & 0\end{pmatrix} \right\} = \C\ell_\mathrm{sand}.
\end{equation}
Define $\mathcal{\tilde U} = \mathcal W^{-,\ru} - \check{\mathcal V}^\ru$. By isolating the first and fourth components of \eqref{eq: conj pt condition L-} we find 
\[
\begin{pmatrix} 0 \\ 0 \end{pmatrix}  = \mathcal W^{-,u}_{1; 4}(x_*) \begin{pmatrix} c_1 \\ c_2 \end{pmatrix}  = \left(\check{\mathcal V}^\ru_{1; 4} +  \mathcal {\tilde U}_{1; 4}(x_*) \right)  \begin{pmatrix} c_1  \\ c_2\end{pmatrix}\\
\]
and so
\[
\check{\mathcal V}^\ru_{1; 4} \begin{pmatrix} c_1  \\ c_2 \end{pmatrix}  =  -\mathcal {\tilde U}_{1; 4}(x) \begin{pmatrix} c_1  \\ c_2\end{pmatrix}
\]
Using \eqref{E:B-infty-evecs}, \eqref{E:evec-bound}, and and Proposition \ref{prop: unstable vec convergence}, this implies
\begin{align}
\left \|\check{\mathcal V}^\ru_{1; 4} \begin{pmatrix} c_1  \\ c_2 \end{pmatrix} \right \| &= \left\| \mathcal {\tilde U}_{1; 4}(x) \begin{pmatrix} c_1  \\ c_2 \end{pmatrix}\right\| \leq \| \mathcal {\tilde U}_{1; 4}(x) \| \|c\| \leq \|\mathcal {\tilde U}(x)\| \|c\| \nonumber \\
&\leq \epsilon \left[ \|\check V_1^\ru\| + \|\check V_2^\ru\|\right]\|c\| \nonumber \\
&\leq  2 \epsilon \left( 1    + \rho + \frac{5}{\rho} + \frac{1}{\rho^2} \right)^{1/2} \|c\|. \label{E:L-est1} 
\end{align}
Also note that 
\[
\left \|\check{\mathcal V}^\ru_{1; 4} \begin{pmatrix} c_1  \\ c_2 \end{pmatrix} \right \|^2 \geq  \| c \| ^2  \min_{c \in \C^2}  \frac{\|\check{\mathcal V}^\ru_{1; 4} c\|^2 }{\| c \| ^2  }  =  \| c \| ^2  \min_{c \in \C^2}  \frac{   (\check{\mathcal V}^\ru_{1; 4}c )^*  (\check{\mathcal V}^\ru_{1; 4} c)     }{\| c \| ^2  } = \| c \| ^2   \min_{ \lambda \in \sigma( A) } \lambda, 
\]
where $\mbox{}^*$ is the conjugate transpose and $A=(\check{\mathcal V}^\ru_{1; 4} c)^*  (\check{\mathcal V}^\ru_{1; 4} c)$. 
It follows (in general) from the singular value decomposition of the matrix that $(\check{\mathcal V}^\ru_{1; 4}  )^*  (\check{\mathcal V}^\ru_{1; 4}  )$ and $(\check{\mathcal V}^\ru_{1; 4}  )  (\check{\mathcal V}^\ru_{1; 4}  )^*$ have the same eigenvalues. The latter is easier to compute, with 
\[
(\check{\mathcal V}^\ru_{1; 4}  )  (\check{\mathcal V}^\ru_{1; 4} c)^* = \begin{pmatrix} \frac{2}{\rho^{2}} & 	\frac{2 \cos(\theta/2)}{\rho^{3/2}} \\ \frac{2  \cos(\theta/2)}{\rho^{3/2}} & 	\frac{2}{\rho} \end{pmatrix}
\]
This has eigenvalues $[1+\rho  \pm  \sqrt{1+\rho ^2+2 \rho  \cos (\theta )}]/\rho ^2$. Thereby we obtain 
\begin{align*}
\left \|\mathcal V^\ru_{1; 4} \begin{pmatrix} c_1  \\ c_2 \end{pmatrix} \right \| &\geq \frac{\sqrt{1+\rho  -  \sqrt{1+\rho ^2+2 \rho  \cos (\theta )}}}{\rho }\| c \|  \\
&\geq  \frac{\sqrt{1+\rho  -  \sqrt{1+\rho ^2}}}{\rho } \| c \|.  
\end{align*}
Combining this with \eqref{E:L-est1}, we find that at any conjugate point we necessarily have
\[
\frac{\sqrt{1+\rho  -  \sqrt{1+\rho ^2}}}{\rho }\| c \|     \leq 2 \epsilon \left( 1    + \rho + \frac{5}{\rho} + \frac{1}{\rho^2} \right)^{1/2} \|c\| .
\]
As a result, a conjugate point cannot occur if
\[
\frac{1+\rho  -  \sqrt{1+\rho ^2}}{\rho^2 } >4 \epsilon^2 \left( 1    + \rho + \frac{5}{\rho} + \frac{1}{\rho^2} \right).
\]
Rearranging, we find that we require
\[
\epsilon^2	\leq \frac{1+\rho  -  \sqrt{1+\rho ^2 }}{4( \rho^2    +\rho^3+5 \rho + 1)}
\]
Using the fact that $\rho > 1$, we see that
\[
 \frac{1+\rho  -  \sqrt{1+\rho ^2 }}{4( \rho^2    +\rho^3+5 \rho + 1)} \geq  \frac{2 - \sqrt{2}}{32\rho^3} \geq \frac{1}{64 \rho^3}
\]
and so it suffices to require
\[
\epsilon^2 \leq \frac{1}{64 \rho^3}
\]
Taking a square root yields the result.
\end{proof}


\subsection{Determination of $L^+_{\mathrm{conj}}$}\label{S:L+}

When characterizing $L^-_{\mathrm{conj}}$, we were only interested in the solutions of the ODE that limited to the unstable eigendirections of $B_\infty$ in backward time. When characterizing $L^+_{\mathrm{conj}}$, things are complicated by the presence of $\varphi'$ as an eigenfunction. Because of this, we seek to characterize two solutions: $\varphi'$ and one that limits to an unknown unstable direction. After incorporating this complication, we characterize $L^+_{\mathrm{conj}}$ in a similar fashion as $L^-_{\mathrm{conj}}$. Namely, we seek a numerical condition on $L^+_{\mathrm{conj}}$ such that, if this condition holds, then $\E^{\mathrm{u}}_-(x) \cap \ell_\mathrm{sand} = \{0\}$ for all $x > L^+_{\mathrm{conj}}$. 

To determine $L^+_{\mathrm{conj}}$, we begin with a similar approach as above and consider \eqref{E:decomp-eval}, which we rewrite here,
\[
Q' = [B_\infty + \tilde B(x)]Q
\]
and note that $\tilde B(x) = B(x) - B_\infty$ is given in \eqref{E:Btilde}. We know that $\lambda = 0$ is an eigenvalue of $\mathcal{L}$, defined in \eqref{E:defL}, with eigenfunction $\varphi'$. Moreover, it will follow from our existence proof that $\varphi$ is constructed in the transverse intersection of the stable and unstable manifolds within the zero-energy level set (see, eg \cite{LessardJamesReinhardt14,VANDENBERG2020310}), and hence $\lambda = 0$ is a simple eigenvalue. As a result, 
\[
\C \E^{\mathrm{u}}_-(x) \cap \C \E_+^\rs(x) = \text{span} \left\{(\varphi', \varphi''', \varphi'''' + 2 \varphi'', \varphi'')\right\},
\]
where we have used the symplectic change of variables given in \eqref{E:defq}. Denote $U_{\varphi'}(x) = (\varphi', \varphi''', \varphi'''' + 2 \varphi'', \varphi'')$, where we note that this is real-valued, and write 
\[
\C \E^{\mathrm{u}}_-(x) = \text{span}\big\{U_{\varphi'}(x), U_1(x) \big\},
\]
where $U_1(x)$ is another (possibly complex-valued) basis solution that we will characterize shortly. 

\begin{lemma}\label{lem:lag}
The solutions $U_{\varphi'}(x) \in \R^4$ and $U_1(x) \in \C^4$ satisfy $\langle U_{\varphi'}(x), J U_1(x) \rangle_{\R^4} = 0$.
\end{lemma}
\begin{proof} If it happens that $U_1(x) \in \R^4$, then $U_1(x) \in \E^\ru_-(x)$, and so this follows from the fact that $\E^\ru_-(x)$ is Lagrangian. Similarly, if $U_1(x)$ is purely imaginary, then $\rmi U_1(x) \in \E^\ru(x)$ and the result follows similarly. So suppose $U_1(x)$ is neither real nor imaginary. This in fact means that $U_1(x)$ and $\bar U_1(x)$ form another basis for $\C\E^\ru_-(x)$, since they are independent solutions that decay to zero as $x \to -\infty$. Hence, $\mathrm{Re}U_1(x)$ and $\mathrm{Im}U_1(x)$ form a basis for $\E^\ru_-(x)$. This means that $0 = \langle U_{\varphi'}(x), J\mathrm{Re}U_1(x) \rangle =  \langle U_{\varphi'}(x), J\mathrm{Im}U_1(x) \rangle$, which implies that $\langle U_{\varphi'}(x), J U_1(x) \rangle_{\R^4} = 0$.
\end{proof}

The basis function $U_{\varphi'}$ will converge to something in the stable subspace of $B_\infty$ as $x \to +\infty$, while $U_1$ will converge to something in the unstable subspace (since we know it does not correspond to another eigenfunction). We do not need to know exactly what these basis functions converge to, but we need enough control of them to be able to show that $\C \E^{\mathrm{u}}_-(x)$ is bounded away from $\C\ell_\mathrm{sand}$ for $x \geq L^+_{\mathrm{conj}}$. 

Similar to what was done above for $L^-_{\mathrm{conj}}$, we can characterize the complex-valued solutions that are asymptotic to the eigendirections at $+\infty$ as 
\begin{align*}
V^{+,\rs}_j(x) & = e^{\mu_j^\rs x} W^{+, \rs}_j(x) = e^{\mu_j^\rs x}\big(\check V_j^{\rs} + \tilde W^{+, \rs}_j(x) \big) \\
V^{+,\ru}_j(x) & = e^{\mu_j^\ru x} W^{+, \ru}_j(x) = e^{\mu_j^\ru x}\big(\check V_j^{\ru} + \tilde W^{+, \ru}_j(x) \big), 
\end{align*}
for $j = 1, 2$, where $\check V_j^{\rs/\ru}$ are the eigenvectors of $B_\infty$ defined in \eqref{E:B-infty-evecs}, $W_j^{+, \rs/\ru}(x) := \check V_j^{\rs/\ru} + \tilde W_j^{+, \rs/\ru}(x)$ and $\tilde W^{+, \rs/\ru}(x)$ decay to $0$ as $x \to \infty$. Define $\mathcal F_+$ analogously to \eqref{E:A-}. In addition, recall the definition of $K$ in \eqref{E:KQ-}
We have a similar result to Proposition \ref{prop: unstable vec convergence} characterizing the convergence rate of $W^{+, \rs/\ru}(x)$ to the asymptotic stable eigenvectors, as well as an associated corollary. 
\begin{proposition}\label{prop: contraction W+} Fix $L^+_{\mathrm{conj}} > 0$ and let $K_B$ and $C_B$ be positive constants as in \eqref{E:B-decay}. Define 
\[
\tau_{\mathcal{F}_+}(L^+_{\mathrm{conj}}) = \frac{KK_B}{C_B}e^{-C_BL^+_{\mathrm{conj}}}.
\]
If $\tau_{\mathcal{F}_+}(L^+_{\mathrm{conj}}) < 1$, then $\mathcal F_+$ is a contraction on $X = L^\infty([L^+_{\mathrm{conj}}, \infty), \C^4)$. The functions $W_j^{+,\rs/\ru}(x)$ satisfy
\[
\|W_j^{+,\rs/\ru}(\cdot) - \check V_j^{\rs/\ru}\|_X \leq \frac{\tau_{\mathcal F_+}(L^+_{\mathrm{conj}})}{1-\tau_{\mathcal F_+}(L^+_{\mathrm{conj}})}\|\check V_j^{\rs/\ru}\|.
\]
\end{proposition} 
\begin{corollary}\label{corr: bounding tilde W}
Let 
\begin{equation}\label{eq: eps 0}
\epsilon_0 : = \frac{\tau_{\mathcal F_+}(L^+_{\mathrm{conj}})}{1-\tau_{\mathcal F_+}(L^+_{\mathrm{conj}})}\max_{j,s,u}\|\check V_j^{\rs/\ru}\|.
\end{equation}
Then 
\[
\left\| \tilde W_j^{+, \rs/\ru}(x) \right\| \leq \epsilon_0, \qquad \forall x \geq L^+_{\mathrm{conj}}, \quad \forall j,s,u.
\]
\end{corollary}

The four solutions $V_{1,2}^{+,\rs/\ru}$ form a basis for $\C^4$, and so we can use them to express arbitrary solutions to equation \eqref{E:eval-0} on the interval $[0, + \infty)$. Thus, we can write the basis functions for $\C \E^{\mathrm{u}}_-(x)$ as 
\begin{equation}\label{eq: unstable basis sols as lin comb}
\begin{split} U_{\varphi'}(x) & = \sum_{j = 1}^2 \tilde\eta_j e^{\mu^\rs_j x}  \big(\check V_j^\rs + \tilde W_j^{+,\rs}(x)\big) \\
& =  \sum_{j = 1}^2 \eta_j  \big(\check V_j^\rs + \tilde W_j^{+,\rs}(x)\big)\\
U_1(x) & =  \sum_{j=1}^2 \bigg[\tilde \gamma_j e^{\mu^\ru_j x} \big(\check V_j^\ru + \tilde W_j^{+,\ru}(x) \big) + \tilde\beta_j e^{\mu^\rs_j x} \big(\check V_j^\rs + \tilde W_j^{+,\rs}(x)\big)  \bigg] \\
& = \sum_{j=1}^2 \bigg[\gamma_j \big(\check V_j^\ru + \tilde W_j^{+,\ru}(x) \big) + \beta_j \big(\check V_j^\rs + \tilde W_j^{+,\rs}(x)\big)  \bigg],
\end{split}
\end{equation}
where $\eta_j: = \tilde \eta_j e^{\mu^\rs_j x}, \beta_j: = \tilde \beta_j e^{-\mu^\rs_j x}, \gamma_j: = \tilde \gamma_j e^{\mu^\ru_jx}$ and $\tilde \gamma_j, \tilde \eta_j$ and $\tilde\beta_j$ are constants that are uniquely determined. 

Note that $U_{\varphi'}(x)$ is real valued but $\check V_j^\rs$ and $\tilde W_j^{+,\rs}(x)$ are complex-valued. The coefficients $\tilde \eta_j \in \C$ can be chosen appropriately so that the linear combination of the complex vectors on the right is real-valued. Additionally, since $U_\varphi'(x) \in \C\E^{\mathrm{s}}_+(x)$, we do not need to include the solutions $V_j^{+,\ru}$ in its basis expansion.

To work with a more compact notation, we will write the above in matrix notation. Denote 
\[
\check{\mathcal V}^{\rs/\ru} = \begin{pmatrix} | & | \\ V_1^{\rs/\ru} & V_2^{\rs/\ru} \\ | & | 
\end{pmatrix}, \qquad \tilde{\mathcal W}^{+,\rs/\ru}(x) = \begin{pmatrix} | & | \\ \tilde W_1^{+,\rs/\ru}(x) & \tilde W_2^{+,\rs/\ru}(x) \\ | & | 
\end{pmatrix}.
\]
Let us denote $\Omega^\rs = \mathrm{diag}(\mu_1^\rs, \mu_2^\rs)$ and $\Omega^\ru = \mathrm{diag}(\mu_1^\ru, \mu_2^\ru) = -\Omega^\rs$. Then $e^{\Omega^\ru x} = \text{diag} \big(e^{\mu_1^\ru x},e^{\mu_2^\ru x}\big) = e^{-\Omega^\rs x}$. In addition, for each $x \geq L^+_{\mathrm{conj}}$  denote 
\begin{align}
\eta = (\eta_1, \eta_2), \quad \gamma &= (\gamma_1, \gamma_2), \quad \beta = (\beta_1, \beta_2), \nonumber \\
\tilde \eta = (\tilde \eta_1, \tilde \eta_2), \quad \tilde \gamma &= (\tilde\gamma_1, \tilde\gamma_2), \quad \tilde \beta = (\tilde\beta_1, \tilde\beta_2). \label{E:agb}
\end{align}
Then we can write
\begin{equation}\label{eq: unstable basis sols as lin comb MAT NOTATION} 
\begin{split} U_{\varphi'}(x) & = \big(\check{\mathcal V}^\rs + \tilde{\mathcal W}^{+,s}(x) \big) \eta = \big(\check{\mathcal V}^\rs + \tilde{\mathcal W}^{+,s}(x) \big) e^{\Omega^\rs x} \tilde\eta \\
U_1(x) & =  \big(\check{\mathcal V}^\rs + \tilde{\mathcal W}^{+,s}(x) \big) \beta + \big(\check{\mathcal V}^\ru + \tilde{\mathcal W}^{+,u}(x) \big) \gamma  \\
& = \big(\check{\mathcal V}^\rs + \tilde{\mathcal W}^{+,s}(x) \big) e^{\Omega^\rs x} \tilde \beta + \big(\check{\mathcal V}^\ru + \tilde{\mathcal W}^{+,u}(x) \big) e^{\Omega^\ru x} \tilde\gamma.
\end{split}
\end{equation}
\begin{remark} Since $U_1(x) \notin \C \E^{\mathrm{s}}_+(x)$, it must be the case that $\tilde \gamma \neq 0$, and so $\gamma(x) \neq 0$ for all $x$.
\end{remark}

Using this characterization of the basis vectors, we seek to show there is an $L^+_{\mathrm{conj}}$ such that $\C E_u^-(x) \cap \C\ell_\mathrm{sand} = \{0\}$ for all $x \geq L^+_{\mathrm{conj}}$. We will do this by constructing two linear subspaces; one determined by the Lagrangian condition and the other determined by the geometric implications of $\C \E^{\mathrm{u}}_-(x) \cap \C\ell_\mathrm{sand} \neq \{0\}$. Then, we will use these two subspaces to construct an analytic inequality that holds if $\C \E^{\mathrm{u}}_-(x) \cap \C\ell_\mathrm{sand} \neq \{0\}$. 

First, notice the Lagrangian condition from Lemma \ref{lem:lag} says 
\begin{align*}0 & = \omega(U_{\varphi'}(x), U_1(x)) \\
&= \omega\left(\big(\check{\mathcal V}^\rs + \tilde{\mathcal W}^{+,s}(x) \big) \eta, \big(\check{\mathcal V}^\rs + \tilde{\mathcal W}^{+,s}(x) \big) \beta + \big(\check{\mathcal V}^\ru + \tilde{\mathcal W}^{+,u}(x) \big) \gamma \right) \\
& = \omega\left(\big(\check{\mathcal V}^\rs + \tilde{\mathcal W}^{+,s}(x) \big) \eta, \big(\check{\mathcal V}^\rs + \tilde{\mathcal W}^{+,s}(x) \big) \beta \right)   + \omega\left(\big(\check{\mathcal V}^\rs + \tilde{\mathcal W}^{+,s}(x) \big) \eta, \big( \check{\mathcal V}^\ru + \tilde{\mathcal W}^{+,u}(x) \big) \gamma \right) \\
& = \omega\left(\big(\check{\mathcal V}^\rs + \tilde{\mathcal W}^{+,s}(x) \big) \eta, \big(\check{\mathcal V}^\ru + \tilde{\mathcal W}^{+,u}(x) \big) \gamma \right),
\end{align*}
where the last line follows from the fact that $\big(\check{\mathcal V}^\rs + \tilde{\mathcal W}^{+,s}(x) \big) \eta, \big(\check{\mathcal V}^\rs + \tilde{\mathcal W}^{+,s}(x) \big) \beta  \in \C \E^{\mathrm{s}}_+(x)$, and so using an argument as in the proof of Lemma \ref{lem:lag}, the symplectic form is zero when applied to these two vectors. 

For any matrix $\mathcal M$, recall that we denote the submatrix consisting of rows $i$ and $j$ of $\mathcal M$ as $\mathcal M_{i; j}$. Using this notation and \eqref{E:B-infty-evecs}, we have
\begin{equation}\label{eq: L+ block notation for basis sols}
\begin{split}\check{\mathcal V}_{1;2}^\ru & = \check{\mathcal V}_{1;2}^\rs = \begin{pmatrix} \frac{1}{\rho}e^{-i\theta} & \frac{1}{\rho}e^{i\theta} \\ 1 & 1
\end{pmatrix} \\
\check{ \mathcal V}_{3;4}^\ru & = -\check{\mathcal V}_{3;4}^\rs = \begin{pmatrix} \sqrt \rho e^{i\theta/2} + \frac{2}{\sqrt \rho} e^{-i\theta/2} & \sqrt \rho e^{-i\theta/2} + \frac{2}{\sqrt \rho} e^{i\theta/2} \\
 \frac{1}{\sqrt \rho}e^{-i\theta/2} & \frac{1}{\sqrt \rho} e^{i\theta/2}
\end{pmatrix}.
\end{split}
\end{equation}

With this block notation, we can write the Lagrangian condition derived above as 
\begin{align*} 0 & = \omega\left(\big(\check{\mathcal V}^\rs + \tilde{\mathcal W}^{+,s}(x) \big) \eta,\big( \check{\mathcal V}^\ru + \tilde{\mathcal W}^{+,u}(x) \big) \gamma \right) \\
& = \left\langle \begin{pmatrix} \check{\mathcal V}^\rs_{1;2} + \tilde{\mathcal W}_{1;2}^{+,s}(x)\\ \check{\mathcal V}^\rs_{3;4} + \tilde{\mathcal W}_{3;4}^{+,s}(x)
\end{pmatrix}\eta, \begin{pmatrix} 0 & -I_2 \\ I_2 & 0
\end{pmatrix} \begin{pmatrix} \check{\mathcal V}_{1;2}^\ru + \tilde{\mathcal W}_{1;2}^{+,u}(x) \\ \check{\mathcal V}_{3;4}^\ru + \tilde{\mathcal W}_{3;4}^{+,u}(x)
\end{pmatrix} \gamma \right\rangle  \\
& = -\left\langle \big(\check{\mathcal V}^\rs_{1;2} + \tilde{\mathcal W}_{1;2}^{+,s}(x) \big) \eta, \big(\check{\mathcal V}_{3;4}^\ru + \tilde{\mathcal W}_{3;4}^{+,u}(x) \big) \gamma \right\rangle + \left\langle \big(\check{\mathcal V}^\rs_{3;4} + \tilde{\mathcal W}_{3;4}^{+,s}(x)\big)\eta , \big(\check{\mathcal V}_{1;2}^\ru + \tilde{\mathcal W}_{1;2}^{+,u}(x) \big) \gamma \right\rangle \\
& = -\left\langle\eta,  \big(\check{\mathcal V}^\rs_{1;2} + \tilde{\mathcal W}_{1;2}^{+,s}(x) \big)^* \big(\check{\mathcal V}_{3;4}^\ru + \tilde{\mathcal W}_{3;4}^{+,u}(x) \big) \gamma \right\rangle  + \left\langle\eta,  \big(\check{\mathcal V}^\rs_{3;4} + \tilde{\mathcal W}_{3;4}^{+,s}(x)\big)^* \big(\check{\mathcal V}_{1;2}^\ru + \tilde{\mathcal W}_{1;2}^{+,u}(x) \big) \gamma \right\rangle \\
& = \left\langle \eta, \left[\big(\check{\mathcal V}^\rs_{3;4} + \tilde{\mathcal W}_{3;4}^{+,s}(x)\big)^* \big(\check{\mathcal V}_{1;2}^\ru + \tilde{\mathcal W}_{1;2}^{+,u}(x) \big)  - \big(\check{\mathcal V}^\rs_{1;2} + \tilde{\mathcal W}_{1;2}^{+,s}(x) \big)^* \big(\check{\mathcal V}_{3;4}^\ru + \tilde{\mathcal W}_{3;4}^{+,u}(x) \big) \right] \gamma \right\rangle.
\end{align*}
Expanding and using the fact that $\check{\mathcal V}_{1;2}^\rs = \check{\mathcal V}_{1;2}^\ru$ and $\check{\mathcal V}_{3;4}^\rs = - \check{\mathcal V}_{3;4}^\ru$, we can rewrite the previous expression to obtain
\changes{
\begin{align}
0 & = \bigg\langle \eta, \left[ \left(\check{\mathcal V}^\rs_{3;4}\right)^*\check{\mathcal V}^\ru_{1;2} + \left(\check{\mathcal V}^\ru_{1;2}\right)^*\check{\mathcal V}^\rs_{3;4} \right]\gamma \bigg\rangle \nonumber  \\
& \qquad + \bigg\langle \eta,  \left[ \left(\check{\mathcal V}_{3;4}^\rs \right)^* \tilde{\mathcal W}^{+,u}_{1;2}(x) - \left(\check{\mathcal V}_{1;2}^\rs \right)^* \tilde{\mathcal W}^{+,u}_{3;4}(x) \right. \label{E:lag-cond}\\
& \qquad \qquad \left. + \left(\tilde{\mathcal W}_{3;4}^{+,s} \right)^*\big(\check{\mathcal V}_{1;2}^\ru + \tilde{\mathcal W}_{1;2}^{+,u}(x) \big) - \left(\tilde{\mathcal W}_{1;2}^{+,s} \right)^*\big(\check{\mathcal V}_{3;4}^\ru + \tilde{\mathcal W}_{1;2}^{+,s}(x) \big)  \right]\gamma \bigg\rangle. \nonumber
\end{align}
}
Let us write
\changes{
\begin{align}
M_1 & = \left(\check{\mathcal V}^\rs_{3;4}\right)^*\check{\mathcal V}^\ru_{1;2} + \left(\check{\mathcal V}^\ru_{1;2}\right)^*\check{\mathcal V}^\rs_{3;4} \label{E:def-A1-Q} \\ 
M_3 & = \left(\check{\mathcal V}_{3;4}^\rs \right)^* \tilde{\mathcal W}^{+,u}_{1;2}(x) - \left(\check{\mathcal V}_{1;2}^\rs \right)^* \tilde{\mathcal W}^{+,u}_{3;4}(x) +  \left(\tilde{\mathcal W}_{3;4}^{+,s} \right)^*\big(\check{\mathcal V}_{1;2}^\ru + \tilde{\mathcal W}_{1;2}^{+,u}(x) \big)  - \left(\tilde{\mathcal W}_{1;2}^{+,s} \right)^*\big(\check{\mathcal V}_{3;4}^\ru + \tilde{\mathcal W}_{1;2}^{+,s}(x) \big),
\nonumber
\end{align}
}
where we note that $M_1$ is constant whereas $M_3$ depends on $x$. Thus, \eqref{E:lag-cond} implies
\begin{equation}\label{E:alpha-perp-1}
\eta(x) \perp \Gamma_1, \qquad \Gamma_1 = \{ (M_1 + M_3(x)) \gamma(x)\}.
\end{equation}

Suppose now there is an intersection of $\C \E^{\mathrm{u}}_-(x)$ with the subspace $\C\ell_\mathrm{sand}$. Since this intersection requires that the first and fourth components of the vector vanish, we introduce the following projection operator onto these coordinates: 
\[
 \pi: \C^{4} \to \C^4, \quad \pi(x_1, x_2, x_3, x_4) = (x_1, 0, 0, x_4).
\]
Using this notation, we see that, if $\C \E^{\mathrm{u}}_-(x) \cap \C\ell_\mathrm{sand} \neq \{0\}$, then 
\begin{align}
0 & = \det \big[\pi U_{\varphi'}(x), \pi U_1(x)\big] \nonumber \\
& = \det \left[ \big(\check{\mathcal V}_{1;4}^\rs + \tilde{\mathcal W}_{1;4}^{+,s}(x)\big)\eta, \big(\check{\mathcal V}_{1;4}^\ru + \tilde{\mathcal W}_{1;4}^{+,u}(x)\big)\gamma + \big(\check{\mathcal V}_{1;4}^\rs + \tilde{\mathcal W}_{1;4}^{+,s}(x) \big)\beta \right] \label{E:det-cond}\\
& = \det\big[\check{\mathcal V}_{1;4}^\rs + \tilde{\mathcal W}_{1;4}^{+,s}(x) \big] \det \left[\eta, \big(\check{\mathcal V}_{1;4}^\rs + \tilde{\mathcal W}_{1;4}^{+,s}(x) \big)^{-1}\big(\check{\mathcal V}_{1;4}^\ru + \tilde{\mathcal W}_{1;4}^{+,u}(x)\big)\gamma + \beta  \right]. \nonumber
\end{align}
Using the definition of $\check{\mathcal V}^{\rs/\ru}$ and the fact that $\theta \in \left(\frac{\pi}{2}, \pi \right)$, one can see that $\check{\mathcal V}_{1;4}^\rs$ and $\check{\mathcal V}_{1;4}^\ru$ are invertible. In particular, we can write $\check{\mathcal V}_{1;4}^\rs + \tilde{\mathcal W}_{1;4}^{+,\rs}(x) = \check{\mathcal V}_{1;4}^\rs[I + (\check{\mathcal V}_{1;4}^\rs)^{-1}\tilde{\mathcal W}_{1;4}^{+,s}(x)]$. By Corollary \ref{corr: bounding tilde W}, we can assume that $x$ is sufficiently large so that $ \| (\check{\mathcal V}^\rs_{1;4})^{-1} \tilde{\mathcal W}_{1;4}^{+,\rs}(x) \|< 1$. (Note this assumption is precisely captured in equation \eqref{E:epsilon0-cond} of Proposition \ref{prop: L+ intersection condition 1} below.) Hence, in order for \eqref{E:det-cond} to hold, it must be the case that
\begin{equation}\label{E:det-cond-mod}
0 = \det \left[\eta, \big(\check{\mathcal V}_{1;4}^\rs + \tilde{\mathcal W}_{1;4}^{+,s}(x) \big)^{-1}\big(\check{\mathcal V}_{1;4}^\ru + \tilde{\mathcal W}_{1;4}^{+,u}(x)\big)\gamma + \beta  \right].
\end{equation}
Next, we define 
\begin{equation}\label{E:def-A2}
M_2 = \left( \check{\mathcal V}_{1;4}^\rs\right)^{-1} \check{\mathcal V}_{1;4}^\ru
\end{equation}
and write
\begin{align*}
\big(\check{\mathcal V}_{1;4}^\rs + & \tilde{\mathcal W}_{1;4}^{+,s}(x) \big)^{-1}  \left(\check{\mathcal V}_{1;4}^\ru + \tilde{\mathcal W}_{1;4}^{+,u}(x)\right) \\
 & = \left(I_2 + \left( \check{\mathcal V}_{1;4}^\rs\right)^{-1} \tilde{\mathcal W}_{1;4}^{+,s}(x)  \right)^{-1} M_2 \left(I_2 + \left(\check{\mathcal V}_{1;4}^\ru\right)^{-1}\tilde{\mathcal W}^{+,u}_{1;4}(x) \right) \\
& = \left( I_2 + \sum_{k=1}^\infty (-1)^k \left(\left( \check{\mathcal V}_{1;4}^\rs\right)^{-1} \tilde{\mathcal W}_{1;4}^{+,s}(x) \right)^k  \right)M_2 \left(I_2 + \left(\mathcal V_{1;4}^\ru\right)^{-1}\tilde{\mathcal W}^{+,u}_{1;4}(x) \right) \\ 
& = M_2 +  M_2 \left(\check{\mathcal V}_{1;4}^\ru\right)^{-1}\tilde{\mathcal W}^{+,u}_{1;4}(x) \\
& \qquad +\left(\sum_{k=1}^\infty (-1)^k \left(\left( \check{\mathcal V}_{1;4}^\rs\right)^{-1} \tilde{\mathcal W}_{1;4}^{+,s}(x) \right)^k  \right)M_2 \left(I_2 + \left(\check{\mathcal V}_{1;4}^\ru\right)^{-1}\tilde{\mathcal W}^{+,u}_{1;4}(x) \right).
\end{align*}
Define 
\begin{equation}\label{eq: frame Gamma2}
M_4 = M_2 \left(\check{\mathcal V}_{1;4}^\ru\right)^{-1}\tilde{\mathcal W}^{+,u}_{1;4}(x)  +\left(\sum_{k=1}^\infty (-1)^k \left(\left( \check{\mathcal V}_{1;4}^\rs\right)^{-1} \tilde{\mathcal W}_{1;4}^{+,s}(x) \right)^k  \right)M_2 \left(I_2 + \left(\check{\mathcal V}_{1;4}^\ru\right)^{-1}\tilde{\mathcal W}^{+,u}_{1;4}(x) \right)
\end{equation}
and, due to \eqref{E:det-cond-mod}, we see that if $\C \E^{\mathrm{u}}_-(x) \cap \C\ell_\mathrm{sand} \neq \{0\}$ it must be the case that
\begin{equation}\label{E:alpha-perp-2}
\eta(x) \in \mathrm{span} \Gamma_2, \qquad \Gamma_2  = \left\{\beta(x) + (M_2 + M_4)\gamma(x) \right\}.
\end{equation}
To summarize, at a conjugate point we must have both \eqref{E:alpha-perp-1} and \eqref{E:alpha-perp-2}. This can only happen if $\Gamma_1 \perp \Gamma_2$. In other words, we have the following lemma. 

\begin{lemma} \label{lem:gammas-perp}Let the vectors $\Gamma_1$ and $\Gamma_2$ in $\C^2$ be defined as in \eqref{E:alpha-perp-1} and \eqref{E:alpha-perp-2}. Then $\C \E^{\mathrm{u}}_-(x) \cap \C\ell_\mathrm{sand} \neq \{0\}$ if and only if $\Gamma_1 \perp \Gamma_2$. 
\end{lemma}

We will use this Lemma to derive an inequality. Then, we can take the contrapositive to conclude that if the inequality does not hold, then there is not an intersection between $\C \E^{\mathrm{u}}_-(x)$ and $\C\ell_\mathrm{sand}$. 

\begin{proposition}\label{prop: L+ intersection condition 2}
Recall the definitions of $M_1, M_2, M_3$, and $M_4$ given in \eqref{E:def-A1-Q}, \eqref{E:def-A2}, and \eqref{eq: frame Gamma2}, and let
\begin{equation}\label{eq: sigma min def}
\sigma_{\min}  = \min \bigg\{ \sqrt{|\mu|} \ : \mu \mbox{ is an eigenvalue of } M_1^*M_2M_2^*M_1 \bigg\}.
\end{equation}
If $L^+_{\mathrm{conj}}$ is sufficiently large so that for any $x \geq L^+_{\mathrm{conj}}$ we have
\begin{equation}\label{eq: Determination of L+ condition}
\sigma_{\min} > \left(\|M_1\| + \|M_3\| \right) [\|\beta\|\|\gamma\|^{-1} + \|M_4\|]+ \|M_2\|\|M_3\|,
\end{equation}
then a conjugate point cannot occur for $x \geq L^+_{\mathrm{conj}}$.
\end{proposition}

\begin{proof}
First notice that, by Lemma \ref{lem:gammas-perp}, we have a conjugate point if and only if 
\[
0  = (\beta + M_2\gamma + M_4\gamma)^* (M_1\gamma + M_3\gamma). 
\]
Expanding and rearranging, we have,
\[
-\gamma^*M_2^*M_1\gamma  = (\beta + M_4\gamma)^*(M_1\gamma + M_3\gamma) + (M_2\gamma)^* M_3\gamma.
\]
Using the triangle and Cauchy-Schwarz inequalities, we obtain
\begin{align*}
| \gamma^*M_2^* M_1\gamma | &= | \beta^*(M_1\gamma +M_3\gamma) + \gamma^* [M_4^* (M_1 + M_3) + M_2^* M_3]\gamma |\\
& \leq \left\| \beta^*(M_1 + M_3) \right\| \|\gamma\| + \|\gamma\|^2 \|M_2^* M_3 + M_4^*(M_1 + M_3)\|_M.
\end{align*}
Notice that $|\gamma^* M_2^*M_1\gamma| \geq \sigma_{\min}\|\gamma\|^2$, and so
\[
\sigma_{\min}\|\gamma\|^2 \leq  \left\| \beta^*(M_1 + M_3) \right\| \|\gamma\| + \|\gamma\|^2 \|M_2^* M_3 + M_4^*(M_1 + M_3)\|_M.
\]
Since $\|\gamma\| \neq 0$, we can divide to obtain 
\begin{equation}\label{eq: L+ intersection condition 1}
\begin{split} \sigma_{\min}& \leq  \left\| \beta^*(M_1 + M_3) \right\| \|\gamma\|^{-1} +  \|M_2^* M_3 + M_4^*(M_1 + M_3)\| \\
& \leq \|\beta\|\big(\|M_1\| + \|M_3\|  \big) \|\gamma\|^{-1} + \|M_2^*\| \|M_3\| + \|M_4^*\|\big(\|M_1\| + \|M_3\|\big)
\end{split}
\end{equation}
and the desired result.
\end{proof}

Note that all of the norms on the right hand side of \eqref{eq: Determination of L+ condition} depend on $x$, except for $\|M_1\|$ and $\|M_2\|$. Furthermore, each of $\|\beta^*\|, \ \|M_3\|$ and $\|M_4\|$ can be made small for large $x$. Lastly, $\|\gamma\|$ grows with $x$, but the negative exponent means that the entire term gets small, as well. 
Using these ideas, we now bound the right hand side of equation \eqref{eq: Determination of L+ condition}.

\begin{proposition}\label{prop: L+ intersection condition 1}
Recall $\epsilon_0 = \epsilon_0(L^+_{\mathrm{conj}})$ given in \eqref{eq: eps 0} and the definitions of $M_1, M_2$, $M_3$ and $M_4$ given in \eqref{E:def-A1-Q}, \eqref{E:def-A2}, and \eqref{eq: frame Gamma2}. Define the following constants: 
\begin{align} \epsilon_\beta & = e^{-\mathrm{Re}(\mu_1^\ru)L^+_{\mathrm{conj}}}\|\tilde \beta\|, \label{eq: epsilon beta} \\
C_{M_4} & = \|M_2\|\left[ \|(\check{\mathcal V}_{1;4}^\ru)^{-1}\| + \frac{\|(\check{\mathcal V}_{1;4}^\rs)^{-1}\| (1 + \epsilon_0 \|(\check{\mathcal V}_{1;4}^\ru)^{-1}\|)}{1-\epsilon_0\|(\check{\mathcal V}_{1;4}^\rs)^{-1}\| }\right]  , \label{eq: CP} \\
\epsilon_\gamma & = e^{-\mathrm{Re}(\mu_1^\ru) L^+_{\mathrm{conj}}} \|\tilde \gamma\|^{-1}, \label{eq: epsilon gamma} \\
C_{M_3} & = 2 \left( \|\check{\mathcal V}_{3;4}^\rs \| + \|\check{\mathcal V}_{1;2}^\rs \| \right) + 2\epsilon_0. \label{eq: CQ}   
\end{align}
If $L^+_{\mathrm{conj}}$ is chosen so that 
\begin{equation}\label{E:epsilon0-cond}
\epsilon_0 \|(\check{\mathcal V}^\rs_{1;4})^{-1}\| < 1
\end{equation}
and so that
\begin{equation}\label{E:sigma-cond}
\sigma_\mathrm{min} >  \epsilon_0C_{M_3} \|M_2\| + (\epsilon_0C_{M_3} + \|M_1\|)(\epsilon_0C_{M_4} + \epsilon_\beta \epsilon_\gamma)
\end{equation}
then there cannot be any conjugate points for $x \geq L^+_{\mathrm{conj}}$.
\end{proposition}

\begin{proof}
From \eqref{eq: unstable basis sols as lin comb MAT NOTATION}, recall that 
 \[
 \gamma = e^{\Omega^\ru x} \tilde \gamma = \begin{pmatrix} e^{\mu_1^\ru x} & 0 \\ 0 & e^{\mu_2^\ru x} \end{pmatrix}\begin{pmatrix} \tilde \gamma_1 \\ \tilde \gamma_2
 \end{pmatrix}, \qquad 
 \beta = e^{-\Omega^\ru x}\tilde \beta = \begin{pmatrix} e^{-\mu_1^\ru x} & 0 \\ 0 & e^{-\mu_2^\ru x} \end{pmatrix}\begin{pmatrix} \tilde \beta_1 \\ \tilde \beta_2
 \end{pmatrix}.
 \]
It follows that
 \[
 \|\beta\| = \|e^{-\Omega^\ru x} \tilde \beta\|  \leq \left\|e^{-\Omega^\ru x} \right\|\|\tilde \beta\| = e^{-\text{Re}(\mu_1^\ru)x}\|\tilde \beta\| = \epsilon_B.
 \]
and
\[
\|\gamma \| \geq e^{\text{Re}(\mu_1^\ru)x} \|\tilde\gamma\| = \epsilon_\gamma^{-1}.
\]
Next, we want to bound the norms of $M_3$ and $M_4$. To do so, note that $\epsilon_0$ defined in \eqref{eq: eps 0} also provides an upper bound for $\tilde{\mathcal W}^{+, s/u}_{i,j}(x)$, and so we find
\[
\|M_3\|  \leq 2\epsilon_0\left(\|\check{\mathcal V}_{3;4}^\rs\| + 2\sqrt{2}\|\check{\mathcal V}^\rs_{1;2}\| + \epsilon_0\right) = \epsilon_0 C_{M_3}. 
\]
Similarly, 
\[
\|M_4\| \leq \epsilon_0\|M_2\|\left[ \|(\check{\mathcal V}_{1;4}^\ru)^{-1}\| + \frac{\|(\check{\mathcal V}_{1;4}^\rs)^{-1}\| (1 + \epsilon_0 \|(\check{\mathcal V}_{1;4}^\ru)^{-1}\|)}{1-\epsilon_0\|(\check{\mathcal V}_{1;4}^\rs)^{-1}\| }\right] = \epsilon_0 C_{M_4}.
\]
Thus, we find
\begin{align*}
\left(\|M_1\|_M + \|M_3\| \right) &[\|\beta\|\|\gamma\|^{-1} + \|M_4\|]+ \|M_2\|\|M_3\| \\
&\qquad \leq \epsilon_0C_{M_3} \|M_2\| + (\epsilon_0C_{M_3} + \|M_1\|)(\epsilon_0C_{M_4} + \epsilon_\beta \epsilon_\gamma)
\end{align*}
Comparing with Proposition \ref{prop: L+ intersection condition 2}, this completes the proof. 
\end{proof}

In order to verify conditions \eqref{E:epsilon0-cond} and \eqref{E:sigma-cond} for $L^+_{\mathrm{conj}}$, several quantities must be determined numerically. We summarize these in Table \ref{table: Lplus quantities necessary}. 

\begin{table}[H]
\centering
\begin{tabular}{|c|c|c|} 
\hline
Term & Definition (Eq \#) & Necessary Quantities \Tstrut\Bstrut \\
\hline 
\hline 
$M_1$ & \eqref{E:def-A1-Q} & $\check{\mathcal V}^{\rs, \ru}$ \Tstrut\Bstrut \\
\hline
$M_2$ & \eqref{E:def-A2} &  $\check{\mathcal V}^{\rs, \ru}$  \Tstrut\Bstrut \\
\hline
$\sigma_\mathrm{min}$ & \eqref{eq: sigma min def}  & $M_1, M_2$ \Tstrut\Bstrut \\ 
\hline
$\epsilon_0$ & \eqref{eq: eps 0} & $\check{\mathcal V}^{\rs, \ru}$, $K$, $C_B$, $K_B$  \Tstrut\Bstrut \\ 
\hline
$\tilde\gamma$ & \eqref{eq: unstable basis sols as lin comb MAT NOTATION} & $U_1(x)$ \Tstrut\Bstrut \\
\hline
$\tilde \beta $ &  \eqref{eq: unstable basis sols as lin comb MAT NOTATION}  & $U_1(x)$ \Tstrut\Bstrut \\ 
\hline
$C_{M_3}$ & \eqref{eq: CQ} & $\epsilon_0, M_1, M_2$\Tstrut\Bstrut \\ 
\hline
$C_{M_4}$ & \eqref{eq: CP} & $\epsilon_0, M_1, M_2$  \Tstrut\Bstrut \\
\hline
\end{tabular}
\caption{The quantities used to determine the constant $L^+_{\mathrm{conj}}$.}\label{table: Lplus quantities necessary}
\end{table}
Recall that $\check{\mathcal V}^{\rs, \ru}$ are defined in \eqref{E:B-infty-evecs} and depend only on the choice of the parameters $\nmu$ and $\nnu$ in the nonlinearity of the Swift-Hohenberg equation \eqref{E:SH}. The quantities $M_1, M_2$ and $\sigma_\mathrm{min}$ in turn only depend on $\check{\mathcal{V}}^{\rs, \ru}$. The quantities $\epsilon_0$ depends on  
$\check{\mathcal V}^{\rs, \ru}$ and on $K$ from \eqref{E:KQ-}, which can be determined via the asymptotic matrix $B_\infty$ that is known explicitly; it also depends on the constants $K_B$ and $C_B$ from \eqref{E:B-decay}, which can be computed when one carries out a validated numerical existence proof of the underlying pulse. The constants $C_{M_3,M_4}$ depend only on quantities we have already described. This leaves us with just $\tilde \beta$ and $\tilde \gamma$, which are defined in \eqref{E:agb}.
Computing these requires computing the second basis solution $U_1(x)$ of $\mathbb{E}^\ru_-(x)$. This will be the subject of \S \ref{Ch:res bundles}. 


\section{Parameterization of Invariant Manifolds}\label{S:mflds}

As discussed in \S\ref{S:background}, we will use the parameterization method to compute the Taylor coefficients of the stable and unstable manifolds associated with the Swift-Hohenberg equation. Carrying this out will be the focus of this section. We note that we do not encounter any resonances in the manifolds, and so the computations in this section are somewhat standard, although lengthy. They follow the framework originally developed in \cite{CabreFontichdelaLlave03a,CabreFontichdelaLlave03b,CabreFontichdelaLlave05}.


\subsection{Analytic framework}\label{S:analytic-framework}

First, we introduce some additional notation that builds on the notation introduced in \S\ref{S:background}. Given $\delta > 0$, define the normed space $\ell^1_{\delta,m}$ as
\[
\ell^1_{\delta,m} = \{ b \in \mathcal{S}_m: \|b\|_{\delta, m}^1 < \infty\}, \qquad \|b\|^1_{\delta,m}   = \sum_{|\alpha| = 0}^\infty |b_\alpha| \delta^{|\alpha|}= \sum_{\alpha_1 = 0}^\infty \cdots \sum_{\alpha_m = 0}^\infty \big|b_{\alpha_1,\dots,\alpha_m} \big| \delta^{\alpha_1 + \dots + \alpha_m}. 
\]
One can check the set $\ell_{\delta,m}^1$ with the product $*$ defined in \eqref{E:cprod} is a Banach algebra in the norm $\| \cdot \|_{\delta,m}^1$. In other words, $\|b \ast \tilde b\|_{\delta,m}^1 \leq \|b\|_{\delta,m}^1\|\tilde b\|_{\delta,m}^1$ for all $b, \tilde b \in \ell_{\delta,m}^1$. 

This is the function space that the coefficients for our parameterizations of the invariant manifolds and vector bundles will live in. For example, if we have
\[
P(\sigma) = \sum_{|\alpha| \geq 0} P_\alpha \sigma^\alpha, \qquad \sigma \in \C^m, \qquad \alpha \in \N^m, \qquad P_\alpha \in \C^n,
\]
then $\{P_\alpha\} \in \mathcal{S}_m$ and we can require also that $\{P_\alpha\} \in \ell^1_{\delta, m}$. This implies that the radius of convergence of the above power series is $\delta$, and this requires a corresponding decay rate in the coefficients $P_\alpha$ as $|\alpha| \to \infty$. Heuristically, multiplication in function spaces corresponds to the Cauchy product in the sequence space that is related to the Taylor expansions. Below it will be convenient to choose $\delta = 1$ (see Remark \ref{rmk: scaling the coefficients}), and in the setting of the Swift-Hohenberg equation we will have $m = 2$ and $n = 4$. 

In practice, we will only compute finitely many coefficients of the Taylor expansion, usually up to order $N$ for some fixed $N <\infty$. To construct bounds on the error between the truncated series and the infinite series, we need to define additional sequence spaces and products. Define the linear subspaces 
\begin{align*} X_{\delta,m}^N & = \big\{b = \{b_\alpha\} \in \ell^1_{\delta,m} \ : \ b_\alpha = 0 \text{ when } |\alpha| \geq N+1 \big\} \\
X^\infty_{\delta, m} & = \big\{ b = \{b_\alpha\} \in \ell^1_{\delta,m} \ : \ b_\alpha = 0 \text{ when } 0 \leq |\alpha| \leq N \big\}.
\end{align*}
Additionally, define the projection operators associated to these spaces as $\Pi^N: \ \ell^1_{\delta, m} \to X_{\delta,m}^N $ and $\Pi^\infty: \ell^1_{\delta,m} \to X^\infty_{\delta,m}$ by 
\begin{align*} 
\Pi_N (b)_\alpha  = \begin{cases} b_\alpha & \text{ if } 0 \leq |\alpha| \leq N \\
0 & \text{ if } |\alpha| \geq N+1
\end{cases}\quad \text{ and } \quad \Pi_\infty(b)_\alpha & = \begin{cases} 0 & \text{ if } 0 \leq |\alpha| \leq N \\ 
b_\alpha & \text{ if } |\alpha| \geq N+1
\end{cases}.
\end{align*}

Since $\ell^1_{\delta,m} = X^N_{\delta,m} \oplus X^\infty_{\delta , m}$, every element $b \in \ell^1_{\delta,m}$ can be written as $b = \Pi_N(b) + \Pi_\infty(b)$. Furthermore $X^N_{\delta,m}$ and $X^\infty_{\delta,m}$ are closed linear subspaces of $\ell^1_{\delta,m}$, so they are Banach spaces and inherit the norm $\|\cdot\|^1_{\delta,m}$. While $X^N_{\delta,m}$ is a Banach space, it does not inherit the Banach algebra structure of $\ell^1_{\delta,m}$. This is because an arbitrary pair of sequences $b,\tilde b\in X^N_{\delta,m}$ has product $b*\tilde b \in \ell^1_{\delta,m}$ with nonzero projection into $X^\infty_{\delta,m}$. This motivates us to define a truncated and tail Cauchy product on $X^N_{\delta,m}$. For $b^N, \tilde b^N \in X^N_{\delta,m}$, the truncated Cauchy product $(\cdot * \cdot)^N : X_{\delta,m}^N \oplus X^N_{\delta, m} \to X^N_{\delta, m}$ is defined component-wise by 
\[
(b^N * \tilde b^N)^N_\alpha  = \begin{cases} \sum_{\beta \leq \alpha} b_{\alpha - \beta} \tilde b_{\beta} & \text{ if } 0 \leq |\alpha| \leq N \\
0 & \text{ if } |\alpha| \geq N+1
\end{cases}. 
\]
The space $X^N_{\delta,m}$ is a Banach algebra when equipped with this operation. We define the tail Cauchy product $(\cdot * \cdot)^\infty : X_{\delta,m}^N \oplus X^N_{\delta, m} \to X^\infty_{\delta, m}$ component-wise  via 
\[
(b^N * \tilde b^N)^\infty_\alpha   = \begin{cases} 0 & \text{ if } 0 \leq |\alpha| \leq N \\ 
\sum_{\beta \leq \alpha} b_{\alpha - \beta} \tilde b_\beta & \text{ if } N+1 \leq |\alpha| \leq 2N \\
0 & \text{ if } |\alpha| \geq 2N+1
\end{cases}.
\]
Finally, since we will be solving for the coefficients of the Taylor series recursively, we introduce the products $(\cdot \ \hat * \ \cdot)^N: X^N_{\delta, m} \oplus X^N_{\delta,m} \to X^N_{\delta,m}$ and $(\cdot \ \hat * \ \cdot)^\infty: X^N_{\delta,m} \oplus X^N_{\delta,m} \to X^\infty_{\delta,m}$, which are defined analogously to $(\cdot * \cdot)^N$ and $(\cdot * \cdot )^\infty$ via $(\cdot \hat * \cdot)$ in \eqref{eq: star hat def}.


\subsection{Coefficients of the Taylor expansion via recursion}\label{Ss: derivation of functional eq}

Recall that $\Phi$ is the flow associated with the vector field \eqref{E:exist}, which we reproduce here:
\begin{align}
u_1' &= u_2 \nonumber \\
u_2' &= u_3 \nonumber \\
u_3' &= u_4 \label{E:exist2} \\
u_4' &= -2u_3 -u_1 + f(u_1). \nonumber
\end{align}
Also recall that, since $\varphi$ is a pulse with $\lim_{|x| \to \infty} \varphi(x) = 0$, and so its corresponding vector-valued function $(\varphi, \varphi_x, \varphi_{xx}, \varphi_{xxx})$ is bi-asymptotic to the origin, we are interested in computing the local stable and unstable manifolds of the origin. Moreover, in the notation of \S\ref{S:background}, the right-hand side of \eqref{E:exist2} is exactly the vector field $G$, and we have  
\[
DG(0) = \begin{pmatrix}0 & 1 & 0 & 0 \\ 0 & 0 & 1 & 0\\ 0 & 0 & 0 & 1 \\ -1+f'(0) & 0 & -2 & 0 \end{pmatrix}.
\]
Note that by the change of variables in \eqref{E:defq}  it follows that  $DG(0)$ has the same eigenvalues as $B_\infty$, given in \eqref{E:eval-0-infty}, and moreover   $B_\infty = S DG(0) S^{-1}$. 
The eigenvalues can be found in \eqref{E:B-infty-evals} - \eqref{E:def-r-theta} and are denoted $\mu_{1,2}^{\rs, \ru}$. However, the eigenvectors are different. Since the eigenvectors of $B_\infty$ were denoted $\check V_{1,2}^{\rs, \ru}$, we denote the eigenvectors of $DG(0)$ as 
\begin{equation}\label{E:dg-evecs}
\hat V_{1,2}^{\rs, \ru} = S^{-1} \check V_{1,2}^{\rs, \ru} =  \begin{pmatrix} 1 \\ \mu_{1,2}^{\rs, \ru} \\ (\mu_{1,2}^{\rs, \ru})^2 \\ (\mu_{1,2}^{\rs, \ru})^3 \end{pmatrix},
\end{equation}
which is consistent with the general notation of \S \ref{S:background}. 

We seek an expression for $P(\sigma)$ having the form given in \eqref{E:Pform} that satisfies \eqref{E:param-2}. In particular, since we have
\[
\sigma \in B_\delta^2(0) \subset \C^2, \qquad \alpha \in \N^2, \qquad P_\alpha \in \C^4, \qquad P: B_\delta^2(0)\subset \C^2 \to \C^4,
\]
we can write
\begin{equation}\label{E:Pform-2}
P(\sigma_1,\sigma_2) = \sum_{m=0}^\infty \sum_{n=0}^\infty P_{mn} \sigma_1^m \sigma_2^n = \sum_{m=0}^\infty \sum_{n=0}^\infty \begin{pmatrix} p^1_{mn} \\ p^2_{mn} \\ p^3_{mn} \\ p^4_{mn} 
\end{pmatrix} \sigma_1^m \sigma_2^n,
\end{equation}
where 
\[
P(0) = 0, \qquad DP(0) = [\hat V_1^\rs, \hat V_2^\rs] = [P_{10}, P_{01}], \qquad \Omega^\rs = \mathrm{diag}(\mu_1^\rs, \mu_2^\rs).
\]
We define a recursive formula for the coefficients $P_{mn}$ that allows for the computation of these coefficients up to any order. We utilize the fact that the nonlinear terms in $G$ are polynomials. By plugging the above expansion into the invariance equation \eqref{E:param-2}, we obtain 
\[
\sum_{m=0}^\infty \sum_{n=0}^\infty (m\mu_1^\rs + n\mu_2^\rs) P_{mn} \sigma_1^m \sigma_2^n = G\left( \sum_{m=0}^\infty \sum_{n=0}^\infty P_{mn} \sigma_1^m \sigma_2^n \right).
\]
In particular,
\[
\big( m \mu_1^\rs + n \mu_2^\rs) \begin{pmatrix} p^1_{mn} \\ p^2_{mn} \\ p^3_{mn} \\ p^4_{mn}  \end{pmatrix} = \begin{pmatrix} p^2_{mn} \\ p^3_{mn} \\ p^4_{mn}  \\ - 2 p^3_{mn}- (\nmu + 1)p^1_{mn}+ \nnu (p^1 *p^1)_{mn} - (p^1 *p^1 *p^1)_{mn} \end{pmatrix}.
\]
Note that as $ p^1_{00} = 0$, it follows that $  0= \frac{\partial}{\partial p_{mn}} \left( \nnu (p^1 *p^1)_{mn} - (p^1 *p^1 *p^1)_{mn} \right)$, and furthermore $ p^1  * p^1 = p^1 \hat * p^1$. Hence we may solve for the coefficients $p_{mn}$ in terms of the lower order coefficients as follows: 
\begin{equation}\label{E:rec-mnfld}
p_{mn} =  \begin{pmatrix} p^1_{mn} \\ p^2_{mn} \\ p^3_{mn} \\ p^4_{mn}  \end{pmatrix}  = \left(DG(0)  - (m\mu_1^\rs + n\mu_2^\rs)I\right)^{-1}  \begin{pmatrix} 0 \\ 0 \\ 0 \\
-\nnu (p^1 * p^1)_{mn} + (p^1 * p^1 * p^1)_{mn}
\end{pmatrix}.
\end{equation}
Note that the matrix on the right hand side that is being inverted is invertible exactly when there is no manifold resonance. Moreover, the vector the inverse acts on is a nonlinear function of the coefficients $P_{m'n'}$ with $m' + n' < m + n$, and it depends on the nonlinearities in the vector field $G$ and on $\mu_{1,2}^\rs$. 
We can solve this equation to any finite order to obtain a polynomial approximation for the local stable manifold given by $P^{N}(\sigma) = \sum_{0 \leq m+n \leq N} P_{mn} \sigma_1^{m}\sigma_2^n$.   We are able to obtain a real-submanifold from the complex manifold we parameterize. 
	
	This is summarized in Lemma \ref{lem:real-valued}, see  \cite{mireles2013rigorous,james2017validated,vandenBergJamesReinhardt16}, which  is essentially due to the fact that the coefficients our our parameterization satisfy $ p_{mn}= (p_{nm})^*$. 

\begin{lemma} \label{lem:real-valued}
Let $P$ be the parameterization for the invariant manifold of the Swift-Hohenberg equation whose coefficients are determined by \eqref{E:Pform-2}, and let $\tilde P(\mathrm{Re}(\sigma), \mathrm{Im}(\sigma)) = P(\sigma, \bar \sigma)$. Then $\tilde P: \R^2 \to \R^4$.  
\end{lemma}

\begin{remark}\label{rmk: scaling the coefficients}
In general, there are two ways to control the size of the region of the invariant manifold that we parameterize. The first is by changing the size of the radius of the ball $B_\delta(0)$ in parameter space in which we allow $\sigma$ to vary. The second method fixes $\delta = 1$ so $\sigma \in B_1(0)$ and scales the coefficients of the parameterization with a diagonal matrix. See \cite{BredenLessardJames16,james2017validated} for more details. We note this corresponds to Remark \ref{rem:levers}, with the diagonal matrix playing the role that $c$ played previously. 
\end{remark}


\subsection{Construction of radii polynomials and a posteriori validation}\label{Ss: construction of rad poly}

In practice, we only compute finitely many coefficients of $P$, meaning that our Taylor expansion is just an approximation. We will use a fixed point argument to bound the error incurred by this truncation. Recalling that the stable manifold is $2$ dimensional and the phase space is $4$ dimensional, endow $(\ell^1_{\delta,2})^4 = \ell^1_{\delta,2} \oplus \ell^1_{\delta,2} \oplus \ell^1_{\delta,2} \oplus  \ell^1_{\delta,2}$ with the norm  
\[
\|(b_1, \dots, b_4) \|_{(\ell^1_{\delta,2})^4} = \sum_{1 \leq i \leq 4} \|b_i\|^1_{\delta,2}.
\]
Recalling that the coefficients of $P$ are given by the recursive relationship \eqref{E:rec-mnfld}, define the operator $\Psi:(\ell^1_{\delta,2})^4  \to (\ell^1_{\delta,2})^4$ via
\begin{equation}\label{E:def-psi}
\Psi(p^1, \dots, p^4)_{mn} := \big[DG(0)  - (m\mu_1^\rs + n\mu_2^\rs)I\big]^{-1} \begin{pmatrix} 0 \\ 0 \\ 0 \\ -\nnu (p^1 * p^1)_{mn} + (p^1 * p^1 * p^1)_{mn} \end{pmatrix}.
\end{equation}
By definition, a fixed point $P = (p^1, \dots, p^4)$ of $\Psi$ solves the recursive relation at all orders; keep in mind that this fixed point is really an entire sequence, $\{P_\alpha\} = \{(p_\alpha^1, p_\alpha^2, p_\alpha^3, p_\alpha^4)\}$. To prove that $\Psi$ has a fixed point, we consider the terms of order less than or equal to $N$ and the tail terms separately. Using the projection operators defined in \S \ref{S:analytic-framework}, we define the two operators
\begin{align*}
\Psi^N&: (X^N_{\delta,2})^4  \to (X^N_{\delta,2})^4 \qquad &\Psi^N & = \Pi^N \circ \Psi \\
\Psi^\infty&: (X^N_{\delta,2})^4 \oplus (X^\infty_{\delta,2})^4 \to (X^\infty_{\delta,2})^4 \qquad &\Psi^\infty & = \Pi^\infty \circ \Psi.
\end{align*}
Using these operators, we can decompose $\Psi$ as
\begin{equation}\label{E:separate}
\Psi(p^1, \dots, p^4) = \Psi^N\left((p^1)^N, \dots, (p^4)^N\right) + \Psi^\infty\left((p^1)^N, \dots, (p^4)^N, (p^1)^\infty, \dots, (p^4)^\infty \right).
\end{equation}
Note that the nonzero components of $\Psi^N$ are the recursive relations for $0 \leq |\alpha| \leq N$. 

Suppose that $\bar P^N = ((\bar p^1)^N, \dots, (\bar p^4)^N)$ is obtained by solving these equations up to order $N$. We note that this can be done with interval arithmetic to enclose the true solution to $ \Psi^N(P^N)=P^N$. We define the map $T: (X^\infty_{\delta,2})^4 \to (X^\infty_{\delta,2})^4$ via 
\begin{equation}\label{eq: tail operator T}
T\left((p^1)^\infty, \dots, (p^4)^\infty\right) =  \Psi^\infty\left((\bar p^1)^N, \dots, (\bar p^4)^N, (p^1)^\infty, \dots, (p^4)^\infty \right).
\end{equation}
Since solving the recursive relations to order $N$ gives a fixed point of $\Psi^N$, we only need to prove the existence of a fixed point of the map $T$ to prove there is a fixed point of the full map $\Psi$. We do this via the following so-called radii polynomial theorem (see, for example,  \cite{DayLessardMischaikow07,hungria2016rigorous,vandenBergJamesReinhardt16}). 

\begin{theorem}\label{thm: Radii Polynomial approach for Contraction Mapping Theorem} 
Suppose that $X$ is a Banach space, $T: X \to X$ is a Fr\'echet differentiable mapping, and $\bar x \in X$. Let $Y_0 > 0$ and $Z: (0, \infty) \to [0, \infty)$ satisfy
\[
\|T(\bar x) - \bar x \|_X  \leq Y_0, \qquad \sup_{x \in \overline{B_r(\bar x)}}\| DT(x)\|_{L(X)}  \leq Z(r).
\]
If there exists an $r_0 > 0$ such that
$\mathfrak p(r_0) < 0$, where
\begin{equation}\label{def: general radii poly}
\mathfrak p(r) = Z(r)r - r + Y_0,
\end{equation}
then there exists a unique $x_* \in \overline{B_{r_0}(\bar x)}$ so that $T(x_*) = x_*$. 
\end{theorem}
In our setting, we will have  $X = (X^\infty_{\delta,2})^4$ and $\bar x = 0$, since our initial guess at a solution is $\bar P^N + 0$, and $T$ will be as defined in \eqref{eq: tail operator T}. One important consequence of this Theorem is that it tells us that the tail of the power series that we didn't compute, $x_*$, lives in the ball of radius $r_0$, and hence is small in this sense. This will allow us to quantify the error made in only computing the first $N$ terms of the series. We also note that $\mathfrak p$ is often referred to as a radii polynomial.

To apply this Theorem to our context, suppose we use \eqref{E:rec-mnfld} to calculate the coefficients $P_{mn}$ up to order $m+n = N$. The next step is to bound the error incurred by this truncation. First, we formulate the Newton-like operator for the stable manifold of this system. Define 
\[
\hat G^0_{mn} = \big[Dg(0)  - (m\mu_1^\rs+n\mu_2^\rs)I\big]^{-1}
\]
so that \eqref{E:rec-mnfld} becomes 
\[
P_{mn}  = \hat G^0_{mn}  \big[0, 0, 0, -\nnu (p^1 * p^1)_{mn} + (p^1 * p^1 * p^1)_{mn} \big]^T 
\]
and we can write $\Psi$ from \eqref{E:def-psi} as
\[
\Psi(a,b,c,d)_{mn} = \hat G^0_{mn}  \big[0, 0, 0, -\nnu (p^1 * p^1)_{mn} + (p^1 * p^1 * p^1)_{mn} \big]^T.
\]
Focusing on the tail terms, we can write our Newton-like operator from \eqref{eq: tail operator T} as
\begin{equation}\label{eq: T} 
T\left((p^1)^\infty, \dots, (p^4)^\infty\right)_{mn} =  \hat G^0_{mn}\begin{pmatrix} 0 \\ 0 \\ 0 \\ \hat T^4((\bar p^1)^N, (p^1)^\infty)_{mn} \end{pmatrix},
\end{equation}
where
\begin{align*}
& \hat T^4((\bar p^1)^N, (p^1)^\infty)_{mn} = \\
& \quad -\nnu[(\bar p^1)^N +(p^1)^\infty] * [(\bar p^1)^N +(p^1)^\infty] _{mn} + [(\bar p^1)^N +(p^1)^\infty] * [(\bar p^1)^N +(p^1)^\infty]   * [(\bar p^1)^N +(p^1)^\infty] _{mn}
\end{align*}
Given the data $((\bar p^1)^N, (\bar p^2)^N, (\bar p^3)^N, (\bar p^4)^N) \in \overline{B_\delta(0)} \subset (X^N_{1,2})^4$, we wish to show that $T$ is a contraction in a small neighborhood of the origin. We have the following lemma that characterizes the bounds given in Theorem \ref{thm: Radii Polynomial approach for Contraction Mapping Theorem}. We fix $\delta = 1$ as per Remark \ref{rmk: scaling the coefficients}. 
\begin{lemma}\label{lemma: SH radii poly for mflds}
Define  
\begin{align*}
 K_N & := 
 \frac{\|\hat{\mathcal{V}}\|}{N\sqrt{\rho} \cos(\theta/2)}
\left( 
\max_{1 \leq i \leq 4 }
\big| (
\hat{\mathcal{V}}^{-1})^{i,4} \big| 
\right) \\
Y_0 & := K_N \left( \nnu\sum_{m+n = N+1}^{2N}  \big| ((\bar p^1)^N * (\bar p^1)^N)_{mn} \big| + \sum_{m+n =N+1}^{3N} \big| ((\bar p^1)^N  * (\bar p^1)^N  * (\bar p^1)^N)_{mn} \big| \right) \\
Z_1 & = K_N \left( 2 \nnu \sum_{1 \leq m+n \leq N} \left|(\bar p^1)^N_{mn} \right| 
+ 3\sum_{1 \leq m+n \leq 2 N} \left| \left( (\bar p^1)^N  * (\bar p^1)^N\right)_{mn} \right|   \right) \\
Z_2(r) & = K_N \left(2\nnu  + 3r + 6 \sum_{1 \leq m+n \leq N } \left|(\bar p^1)^N \right| \right),
\end{align*}
where $\hat{\mathcal{V}}$ denotes the matrix whose columns are the eigenvectors of $DG(0)$. If $T$ is given by \eqref{eq: T}, then 
\begin{align}
\|T(0) - 0\|_{(X_{1,2}^\infty)^4} &  \leq Y_0 \nonumber  \\
\sup_{P^\infty \in \overline{B_\delta(0)}} \| DT\left(P^\infty\right)\|_{L((X_{1,2}^\infty)^4)} & \leq Z_1 + Z_2(r)r \label{E:radii-bounds}
\end{align}
\end{lemma} 
\begin{proof} Denote the entry of the matrix $\hat G^0_{mn}$ in row $i$ and column $j$ as $(\hat G^0_{mn})^{i,j}$ and define the operator $(\hat G^0)^{ij}: X_{1,2}^\infty \to X_{1,2}^\infty$ componentwise via 
\[
((\hat G^0)^{i,j}b)_{mn} = \begin{cases} 0 & \text{ if } 0 \leq m+n \leq N \\ 
(\hat G^0_{mn})^{i,j} b_{mn} & \text{ if } m+n \geq N+ 1
\end{cases}.
\]
We can then write $T = (T^1, T^2, T^3, T^4)$ as 
\begin{align*}
&T^j \left((p^1)^\infty, \dots, (p^4)^\infty\right) = (\hat G^0)^{j,4}  \hat T^4((\bar p^1)^N, (p^1)^\infty)_{mn} =   \\
& \quad  (\hat G^0)^{j,4} \bigg( -\nnu \big((\bar p^1)^N \hat * (\bar p^1)^N \big)^\infty - 2\nnu \big( (\bar p^1)^N \hat * (p^1)^\infty \big) - \nnu\big((p^1)^\infty \hat * (p^1)^\infty \big) + \big( (\bar p^1)^N \hat * (\bar p^1)^N \hat * (\bar p^1)^N \big)^\infty \\
		& \qquad \quad +3\big( (\bar p^1)^N \hat *(\bar p^1)^N \hat * (p^1)^\infty \big) + 3 \big( (\bar p^1)^N \hat * (p^1)^\infty \hat * (p^1)^\infty \big) + \big( (p^1)^\infty \hat * (p^1)^\infty \hat * (p^1)^\infty \big) \bigg).
\end{align*} 
 
In addition, let $\hat{\mathcal{V}}$ denote the matrix of eigenvectors associated with $DG(0)$, so that $DG(0) = \hat{\mathcal{V}} \mathrm{diag}(\mu_1^\rs, \mu_2^\rs, \mu_1^\ru, \mu_2^\ru)\hat{\mathcal{V}}^{-1}$. Then we obtain 
\[
\hat G^0_{mn} = [DG(0) - (m \mu_1^\rs + n \mu_2^\rs)I]^{-1} = \hat{\mathcal{V}}[\mathrm{diag}(\mu_1^\rs, \mu_2^\rs, \mu_1^\ru, \mu_2^\ru) - (m \mu_1^\rs + n \mu_2^\rs)I]^{-1} \hat{\mathcal{V}}^{-1}.
\]
To bound $(\hat G^0)^{j,4}$ we may write
	\[
	(\hat G^0_{mn})^{j,4}
	=
	\left(
	\begin{smallmatrix}
		0\, \cdots \,1  \,\cdots \,0
	\end{smallmatrix}
	\right)
	\hat{\mathcal{V}}[\mathrm{diag}(\mu_1^\rs, \mu_2^\rs, \mu_1^\ru, \mu_2^\ru) - (m \mu_1^\rs + n \mu_2^\rs)I]^{-1} \hat{\mathcal{V}}^{-1}
	\left(\begin{smallmatrix}
		0 \\ 
		0 \\
		0 \\
		1
	\end{smallmatrix}\right),
	\]
where $\left(\begin{smallmatrix} 0\, \cdots \,1  \,\cdots \,0 \end{smallmatrix} \right)$ denotes the row vector with a $1$ in the $j^{\mathrm{th}}$ spot, and zeros elsewhere. One may then compute  
	\[
	\big|(\hat G^0_{mn})^{j,4} \big| \leq 
	\left( 
	\max_{1 \leq k \leq 4 }
	\big| (
	\hat{\mathcal{V}})^{j,k} \big| 
	\right)
	\left(
		\sup_{m+n \geq N+1} \max_{i, \rs, \ru} 
		\frac{1}{|\mu_i^{\rs, \ru} - (m \mu_1^\rs + n \mu_2^\rs)| }
		\right)
		\left( 
		\max_{1 \leq i \leq 4 }
		\big| (
		\hat{\mathcal{V}}^{-1})^{i,4} \big| 
		\right).
	\]
To bound the middle diagonal matrix, recall that $\mu_1^\rs = -\sqrt{\rho}e^{i\theta/2}$, $ \mu_2^\rs =  -\sqrt{\rho}e^{-i\theta/2}$, and  $\mu_i^\ru = - \mu_i^\rs $ for $ i=1,2$. Using the fact that $|z|^2 \geq (\mathrm{Re}z)^2$, we find
	\begin{align}
		\inf_{m+n \geq N+1} \min_{i, \rs, \ru} |\mu_i^{\rs, \ru} - (m \mu_1^\rs + n \mu_2^\rs)| &\geq \inf_{m+n \geq N+1} |\sqrt{\rho}(m+n)\cos (\theta/2) \pm \sqrt \rho \cos(\theta/2)|
		\nonumber \\
		& \geq N\sqrt{\rho} \cos(\theta/2). \label{eq:K_N_bound}
	\end{align}
We thereby obtain

\begin{align*}
		\sum_{1\leq j \leq 4} \|(\hat G^0)^{j,4}\|_{L( X^\infty_{1,2})} 
	&\leq 
	\sum_{1\leq j \leq 4} 
	\left( 
	\max_{1 \leq k \leq 4 }
	\big| (
	\hat{\mathcal{V}})^{j,k} \big| 
	\right)
	\frac{1}{N\sqrt{\rho} \cos(\theta/2)}
	\left( 
	\max_{1 \leq i \leq 4 }
	\big| (
	\hat{\mathcal{V}}^{-1})^{i,4} \big| 
	\right) \\
	& = \frac{\|\hat{\mathcal{V}}\|}{N\sqrt{\rho} \cos(\theta/2)}
		\left( 
	\max_{1 \leq i \leq 4 }
	\big| (
	\hat{\mathcal{V}}^{-1})^{i,4} \big| 
	\right) = K_N,
\end{align*}
where we recall the $\ell^1$ operator norm of a finite matrix $ A$ is given as $ \| A \|  = \max_j \sum_i  |A^{i,j}|$, and we can interchange the sup and the sum because, as given in \eqref{E:dg-evecs}, the   $j^{\mathrm{th}}$ component of each eigenvector in  has the same absolute value.

We can now prove the first bound in \eqref{E:radii-bounds}. We have  
\begin{align*} 
&\|T(0,0,0,0)  - 0\|_{(X_{1,2}^\infty)^4} = \sum_{1 \leq j \leq 4} \| (\hat G^0)^{j,4} \hat T^4(0,0,0,0)\|_{X^\infty_{1,2}} \\
&  \qquad  =  \sum_{1 \leq j \leq 4} \left\| (\hat G^0)^{j,4} \right\|
 \left\|    -\nnu \big( (\bar p^1)^N  * (\bar p^1)^N \big)^\infty + \big( (\bar p^1)^N * (\bar p^1)^N * (\bar p^1)^N \big)^\infty   \right\|_{X^\infty_{1,2}} 
 \\
& \qquad \leq K_N
\left( \nnu \left\| \left((\bar p^1)^N * (\bar p^1)^N \right)^\infty \right\|_{1,2}^1 + \left\| 
\left((\bar p^1)^N  * (\bar p^1)^N  * (\bar p^1)^N \right)^\infty \right\|_{1,2}^1  \right) \\
& \qquad = K_N \left( \nnu\sum_{|\alpha| = N+1}^{2N}  \big| ((\bar p^1)^N  * (\bar p^1)^N)_{\alpha} \big| + \sum_{|\alpha| =N+1}^{3N} \big| ((\bar p^1)^N  * (\bar p^1)^N  * (\bar p^1)^N)_{\alpha} \big| \right) \\
& \qquad =Y_0.
\end{align*}

To prove the second bound in \eqref{E:radii-bounds}, we must characterize the Fr\'echet derivatives of $T$. Looking at the components of $T^j$, the partial derivatives corresponding to $(\bar p^{2,3,4})^N$ will all vanish, so all we have to do is compute the partial derivative associated with $(\bar p^1)^N$. If $h \in X_{1,2}^\infty$, this partial derivative will have the action
\begin{align*}
& D_1 T^j \left((p^1)^\infty, \dots, (p^4)^\infty\right)  h = \\
& \qquad  (\hat G^0)^{j,4} \bigg( -2\nnu((\bar p^1)^N * h) - 2\nnu (h * (p^1)^\infty) + 3\big( (\bar p^1)^N * (\bar p^1)^N * h \big) +6\big((\bar p^1)^N * h  * (p^1)^\infty \big)  \\
& \qquad \qquad \qquad \qquad +3(h * (p^1)^\infty * (p^1)^\infty) \bigg),
\end{align*}
where we have used the fact that $*$ is a commutative operation. In addition, note that the partials are all continuous, since they quadratic in the operation $\ast$. For any $ \left((p^1)^\infty, \dots, (p^4)^\infty\right)  \in \overline{B_r(0)} \subseteq \left( X_{1,2}^\infty \right)^4$, we have 
\begin{align*}
& \| DT\left((p^1)^\infty, \dots, (p^4)^\infty\right) \|_{L((X_{1,2}^\infty)^4)}  = \max_{1 \leq k \leq 4} \sum_{j=1}^4 \|D_k T^j\left((p^1)^\infty, \dots, (p^4)^\infty\right)\|_{L(X_{1,2}^\infty)} \\
& \qquad  = \sum_{1 \leq j \leq 4} \sup_{\|h\| = 1} \|D_1T^j\left((p^1)^\infty, \dots, (p^4)^\infty\right)h \|_{X_{1,2}^\infty}\\
& \qquad \leq K_N \Big( 2 \nnu \| (\bar p^1)^N\|_{X_{1,2}^N}+ 2\nnu \|(p^1)^\infty\|_{X_{1,2}^\infty} + 3 \| (\bar p^1)^N *  (\bar p^1)^N \|_{X_{1,2}^N}  \\
& \qquad \qquad \qquad \qquad \qquad \qquad \qquad  +6\|(p^1)^\infty\|_{X_{1,2}^\infty}\| (\bar p^1)^N \|_{X_{1,2}^N} + 3 \|(p^1)^\infty\|_{X_{1,2}^\infty}^2 \Big) \\
& \qquad \leq  K_N \left( 2 \nnu \| (\bar p^1)^N \|_{X_{1,2}^N}+ 2\nnu r + 3 \| (\bar p^1)^N *  (\bar p^1)^N  \|_{X_{1,2}^N}   + 6r \| (\bar p^1)^N \|_{X_{1,2}^N} + 3 r^2 \right) \\
& \qquad = Z_1 + Z_2(r)r.
\end{align*} 
\end{proof}

We now explain how this Lemma allows us to conclude that the assumptions of Theorem \ref{thm: Radii Polynomial approach for Contraction Mapping Theorem} hold, and hence why we can conclude the existence of a fixed point of the map $T$, and thus also of the full map $\Psi$. With $Y_0$ and $Z(r)$ as in the statement of Lemma \ref{lemma: SH radii poly for mflds}, we can define
\[
\mathfrak p(r) = (Z_1+Z_2r)r -r+Y_0 = Z_2r^2 + (Z_1-1)r + Y_0.
\]
Notice that $K_N \to 0$ as $N \to \infty$. Making $K_N$ small will also make $Z_1$, $Y_0$, and the scaling factor in $Z_2$ small. In other words, 
\[
\mathfrak p(r) = K_N\left( \frac{Z_2(r)}{K_N} r^2 + \frac{Z_1}{K_N} r + \frac{Y_0}{K_N} \right) - r,
\]
where everything inside the parentheses stays bounded as $K_N$ gets small. This allows us to force the existence of an $r_0 > 0$ such that $\mathfrak p(r_0) < 0$, and hence also force the existence of the desired fixed point, by choosing $N$ sufficiently large. Moreover, by making $N$ larger, we can make $r_0$ smaller, and thus ensure that our truncated Taylor series is as good an approximation as we wish. 


\section{The nonresonant unstable bundle over the unstable manifold}
\label{Ch:nonres bundles}
\label{S:nonres methods}

In this section we focus on the blue dot in Figure \ref{fig:computation-depiction}. In particular, we seek to determine a high-order approximation of $\mathbb{E}^\ru_-(-L^-_{\mathrm{conj}}; 0)$. This subspace is exactly the unstable vector bundle over the unstable manifold, evaluated along the solution $\varphi$ that lies in the unstable manifold. To compute its Taylor expansion, we will use the parameterization method as outlined in \S\ref{S:param method for bundles}. We will break things up into two steps. First, we will determine how to choose $L^-_{\mathrm{conj}}$ sufficiently large so that both the conditions of \S\ref{S:L-} are satisfied. Next, we will determine a Taylor expansion for the vector bundle at $-L^-_{\mathrm{conj}}$. 


\subsection{Determining $L^-_{\mathrm{conj}}$}

There were two key requirements on $L^-_{\mathrm{conj}}$ that were obtained in \S\ref{S:L-}, and we review them now. First, in Proposition \ref{prop: unstable vec convergence}, the quantity
\[
\tau_{\mathcal F_-}(L^-_{\mathrm{conj}}) = \frac{K K_B}{C_B} e^{-C_B L^-_{\mathrm{conj}}},
\]
was defined, where $C_B$ and $K_B$ are as in \eqref{E:B-decay} and $K$ was defined in \eqref{E:KQ-}. Moreover, this Proposition required that $L^-_{\mathrm{conj}}$ was chosen sufficiently large so that $\tau_{\mathcal F_-}(L^-_{\mathrm{conj}}) < 1$. Second, in Proposition \ref{prop: L-_alt} it was required that $L^-_{\mathrm{conj}}$ was chosen sufficiently large so that 
\[
\frac{\tau_{\mathcal F_-}(L^-_{\mathrm{conj}})}{1-\tau_{\mathcal F_-}(L^-_{\mathrm{conj}})} <  \frac{1}{8 \rho^{3/2}}.
\]
In this section, we determine a computationally tractable way to bound $\tau_{\mathcal F_-}(L^-_{\mathrm{conj}})$, so that it can be verified that both of these conditions hold. 

In the below, $\bm{\varphi}$ denotes the vector-valued pulse, which is a solution of \eqref{E:SH-vf}.

\begin{theorem}\label{thm: compute L-}
Suppose that $\mathcal{M}^\ru_{loc}(0)$ is parameterized by the function $P(\sigma)$ with $P:B_1(0) \subseteq \C^2 \to \C^4$ and that there exists a $\sigma \in B_1(0)$ such that $\bm{\varphi}(0) = P(\sigma)$. Then 
\begin{equation}\label{E:tau-bound}
\tau_{\mathcal F_-}(L^-_{\mathrm{conj}})  \leq  \frac{Ke^{-|\mathrm{Re}\mu_1^\ru|L^-_{\mathrm{conj}}}}{|\mathrm{Re}\mu_1^\ru|} \|P\|_{1,2}^1\left(2|\nnu| + 6  \|P\|_{1,2}^1\right),
\end{equation}
where $\|P\|_{1,2}^1$ is the norm of the sequence of coefficients for the Taylor series of $P$. 
\end{theorem}

Note that, similar to what was discussed in Remarks \ref{rem:levers} and \ref{rmk: scaling the coefficients}, we have the freedom to choose the translate of the pulse that we work with. This allows us to choose a translate so that we have $\bm{\varphi}(0) \in P(B_1(0))$, and hence a $\sigma$ as required in the above Theorem exists. This fixes the translate of the pulse. The point of the above theorem is that, once we fix such a translate, we can further choose $L^-_{\mathrm{conj}}$ large enough so that we can make the right hand side of \eqref{E:tau-bound} as small as we need to ensure that any conditions on $\tau_{\mathcal{F}_-}$ are met. We also note that \cite{BeckJaquettePieperStorm25} will contain a validated existence proof of $\bm{\varphi}$ using a boundary value problem that will enforce that $\bm{\varphi}(0) = P(\sigma)$ for some $\sigma \in B_1(0)$. 

Recall that $\Phi(x; P(\sigma)) = P\left(e^{\Omega^\rs x} \sigma \right)$, where $\Phi$ is the flow associated with \eqref{E:exist2}. If $\sigma_0$ is chosen so that $P(\sigma_0) = \bm{\varphi}(x_0)$ for some $x_0$, meaning we have chosen a point on the manifold that coincides with a point on the pulse, then by the properties of the flow, for any $\tilde x \leq 0$, we have
\[
\bm{\varphi}(x_0 + \tilde x) = \Phi(\tilde x, \bm{\varphi}(x_0)) = \Phi(\tilde x, P(\sigma_0)) = P\left(e^{\Omega^\rs \tilde x} \sigma_0 \right) = \sum_{m = 0}^\infty\sum_{n = 0}^\infty P_{mn} \left(e^{\mu_1^\ru \tilde x} \sigma_{0, 1} \right)^m \left(e^{\mu_2^\ru \tilde x} \sigma_{0, 2} \right)^n.
\]
Note that the sum on the right has no explicit dependence on $x_0$; the value of $x_0$ is encoded in $\sigma_0 = (\sigma_{0, 1}, \sigma_{0, 2})$. Let $x_0 = 0$ and $\tilde x = -L^-_{\mathrm{conj}} + x$, for $x \leq 0$. This gives
\begin{equation}\label{E:param-pulse}
\bm{\varphi}(\tilde x) = \sum_{m = 0}^\infty\sum_{n = 0}^\infty P_{mn} \left(e^{\mu_1^\ru (x - L^-_{\mathrm{conj}})} \sigma_{0, 1} \right)^m \left(e^{\mu_2^\ru (x - L^-_{\mathrm{conj}})} \sigma_{0, 2} \right)^n.
\end{equation}
This will allow us to view $L^-_{\mathrm{conj}}$ as setting some initial bound on the size of the solution, with $x \leq 0$ providing further decay.  

\begin{proof} First recall that $C_B$ and $K_B$ are as in \eqref{E:B-decay}. We are interested in determining values of these constants such that in particular
\[
\| B(x) - B_\infty\| \leq K_B e^{-C_B|x|}, \qquad x \leq -L^-_{\mathrm{conj}}.
\]
Using the definition of $B(x)$ and $B_\infty$ in \eqref{E:eval-0} and \eqref{E:eval-0-infty}, we find
\[
\| B(x) - B_\infty \|  = | f'(\varphi(x)) - f'(0)|  \leq \sup_{|c| \leq |\varphi(x)|} |f''(c)||\varphi(x)| \leq (2|\nnu| + 6|\varphi(x)|)|\varphi(x)|. 
\]
To bound $|\varphi(x)|$, we use \eqref{E:param-pulse}. We then find
\[
|\varphi(x)| \leq \|\bm{\varphi}(x)\| = \left| \sum_{m = 0}^\infty\sum_{n = 0}^\infty P_{mn} \left(e^{\mu_1^\ru (x + L^-_{\mathrm{conj}})} \sigma_{0, 1} \right)^m \left(e^{\mu_2^\ru (x + L^-_{\mathrm{conj}})} \sigma_{0, 2} \right)^n\right|.
\]
Recall that $p_{00} = 0$, so each term in the above sum has at least one factor of $e^{\mu_1^\ru x}$ or $e^{\mu_2^\ru x}$. The real parts of $\mu_1^\ru, \mu_2^\ru$ are equal and positive, and $|\sigma| \leq 1$, so we can pull out a factor of $e^{\mathrm{Re}\mu_1^\ru x}$ and the remaining terms will all be bounded. In other words, 
\[
|\varphi(x)| \leq \left| e^{\mathrm{Re}\mu_1^\ru (x+L^-_{\mathrm{conj}})}\sum_{m = 0}^\infty\sum_{n = 0}^\infty P_{mn}\right|  \leq e^{\mathrm{Re}\mu_1^\ru (x+L^-_{\mathrm{conj}})} \|P\|_{1,2}^1.
\]
Thus, we have 
\[
\| B(x) - B_\infty \| \leq e^{\mathrm{Re}\mu_1^\ru (x+L^-_{\mathrm{conj}})} \|P\|_{1,2}^1\left(2|\nnu| + 6  e^{\mathrm{Re}\mu_1^\ru (x+L^-_{\mathrm{conj}})} \|P\|_{1,2}^1\right).
\]
Therefore, if we define $C_B = \mathrm{Re}\mu_1^\ru$ and 
\[
K_B = \|P\|_{1,2}^1\left(2|\nnu| + 6  \|P\|_{1,2}^1\right),
\]
we have
\[
\|B(x) - B_\infty\| \leq K_B e^{-C_B |x+L^-_{\mathrm{conj}}|}, \qquad x \leq -L^-_{\mathrm{conj}}.
\]
Finally, from Proposition \ref{prop: unstable vec convergence}, we see that 
\[
\tau_{\mathcal A_-}(L^-_{\mathrm{conj}})  \leq \frac{K K_B}{C_B}e^{-C_B L^-_{\mathrm{conj}}}  = \frac{Ke^{-\mathrm{Re}\mu_1^\ru L^-_{\mathrm{conj}}}}{\mathrm{Re}\mu_1^\ru} \|P\|_{1,2}^1\left(2|\nnu| + 6  \|P\|_{1,2}^1\right).
\]
\end{proof}


\subsection{$\mathbb E^\ru_-(-L^-_{\mathrm{conj}})$ and its associated contraction operator}\label{sec: param method for Eu-}

We can now use the parameterization method for invariant vector bundles described in \S \ref{S:param method for bundles} to obtain $\E^{\mathrm{u}}_-(-L^-_{\mathrm{conj}})$. Recall that $P: \C^2 \to \C^4$ is the parameterization of the unstable manifold,
\[
P(\sigma) = \sum_{m, n = 0}^\infty P_{mn} \sigma_1^{m}\sigma_2^n, \qquad P_{mn} = \begin{pmatrix} p^1_{mn} \\ p^2_{mn} \\ p^3_{mn} \\ p^4_{mn} \end{pmatrix}.
\]
This parameterizes the unstable manifold of the vector field in \eqref{E:exist2}, and for the associated flow $\Phi$, we have $\Phi(x; P(\sigma)) = P\left(e^{\Omega^\rs x} \sigma\right)$.  

The unstable bundle of the unstable manifold corresponds to solutions of the variational equation that grow exponentially fast as $x \to -\infty$. We can obtain an expansion for this bundle simply by differentiating the above expansion for the manifold itself. More specifically, we have that the two-dimensional unstable bundle is spanned by $W_i^{u,-}(\sigma): \C^2 \to \C^4$ with expansions 
\begin{align*}
W_1^{\ru, -}(\sigma)  &= \partial \sigma_1 P(\sigma)  & W_2^{\ru,-}(\sigma) &= \partial \sigma_2 P(\sigma).\\
&=\sum_{m, n = 0}^\infty mP_{mn} \sigma_1^{m-1}\sigma_2^n & &=\sum_{m, n = 0}^\infty n P_{mn} \sigma_1^{m}\sigma_2^{n-1}
\end{align*} 
Then we may define independent real spanning vectors $ W_\mathrm{Re}^{\ru, -}(\sigma) = \mathrm{Re}( W_1^{\ru, -}(\sigma))$ and $ W_\mathrm{Im}^{\ru, -}(\sigma) = \mathrm{Im}( W_1^{\ru, -}(\sigma))$.
\begin{remark}
We note that the above notation is consistent with the notation of \S\ref{S:background} for rank one vector bundles; see, for example, Definition \ref{D:stable-unstable-bundle}. We use similar notation in \S\ref{Ch: Lpm}, for example in \eqref{eq: unstable solns L-}. This is a slight abuse of notation since, we are using different coordinates and parameterizations in this section and in \S\ref{Ch: Lpm}. Essentially, however, these objects are the same. 
\end{remark}

Using these expressions and the conjugacy relation $\Phi(x; P(\sigma)) = P\left(e^{\Omega^\rs x} \sigma\right)$, we can determine the evolution in $x$ of the unstable bundle, which corresponds to the evolution of solutions to the variational equation. We note that this method will in general work to determine the parameterization of an unstable bundle over a parameterized unstable manifold, or to determine the parameterization of a stable bundle over a parameterized stable manifold. However, to determine the parameterization of an unstable bundle over a parameterized stable manifold (or vice versa), one must use another method. This will be relevant in \S \ref{Ch:res bundles}, below. 

\section{The nonresonant stable bundle and resonant unstable bundle over the stable manifold} \label{Ch:res bundles}

In this section we aim to understand the linear dynamics about the stable manifold in a neighborhood of the origin. In particular, we use the parameterization method to compute its stable and unstable bundles. 

Recall that, if we write $\C \E^{\mathrm{u}}_-(x) = \text{span}\big\{U_{\varphi'}(x), U_1(x) \big\},$ then our $L^+_{\mathrm{conj}}$ calculation requires us to write $U_1(x)$ in eigen-coordinates, as per \eqref{eq: unstable basis sols as lin comb}. That is, 
\begin{align} \label{eq:EigenCoordinates-Application}
U_1(x)  &=  \sum_{j=1}^2 \left( \tilde \gamma_j   e^{\mu^\ru_j x}  W_j^{+,\ru}(x)+ \tilde\beta_j  e^{\mu^\rs_j x}  W_j^{+,\rs}(x) \right) ,
\end{align}
where each $V_j^{+,\ru/\rs}(x) = e^{\mu^{\ru/\rs}_j x}  W_j^{+,\ru/\rs}(x)$ is a solution to the linearized equation about $\bm{\varphi}$,  and each $W_j^{+,\ru/\rs}(x)$ limits to the eigenvector $\check V_j^{\ru/\rs}$  as $ x \to +\infty$. 
Proposition \ref{prop: L+ intersection condition 1} requires an explicit calculation of the coefficients  $\tilde\beta_j, \tilde\gamma_j$.  This requires us to, for a given $ x_0 \in \R$, write $U_1(x_0) $ in the basis given by $\{ W_j^{+,\ru/\rs}(x_0)\}_{j=1,2}$. 
This section details how we may calculate this basis with high precision. 

We apply the methodology in Section \ref{sec:Resonant-Vector-Bundle-Theory} to compute both the stable and unstable bundles, and obtain \emph{a posteriori} bounds in a manner largely analogous to Section \ref{S:mflds}. We note that the computation of the unstable bundle depends on the stable bundle. While it is possible to apply the methodology in Section \ref{sec: param method for Eu-} to compute the (non-resonant) stable bundle, the approach we take here computes and validates all bundles at once.

This section is where we will encounter a resonance when computing the unstable vector bundle.  Like in Section \ref{S:mflds}, we work in $(u, u_x, u_{xx}, u_{xxx})$ coordinates, and then at the end we convert to $(u, u_{xx}, u_{xxx} + 2 u_x, u_x)$ coordinates.


\subsection{Computing the bundles}

Consider the vector field $ \dot{U} = G(U)$ in \eqref{E:exist2}, and fix a parameterization $ P:B^2_\delta(0) \to \R^4$ of the stable manifold. (We acknowledge a slight abuse of notation in using $P$ for both the stable and unstable manifolds.) For the flow $\Phi$ associated with \eqref{E:exist2}, we have $\Phi(x; P(\sigma_0)) = P\left(e^{\Omega^\rs x} \sigma_0 \right)$. For solutions $U(x) =  \Phi(x,P(\sigma))$ on the stable manifold, we wish to solve the variational equation 
\begin{align} \label{eq:VariationalEq-Resonance}
V'(x) = DG(U(x)) V(x) .
\end{align}
This generalizes the problem from just the linearization about the pulse, to the linearization about any solution on the stable manifold. 
 
To solve the linearized equation \eqref{eq:VariationalEq-Resonance} we shall use Lemma \ref{L:conjugacy}. That is, we find a bundle $\mathcal{W}:B^2_\delta(0) \to \GL_4(\R)$ and a normal form of the linear dynamics $ \mathcal A:B^2_\delta(0) \to \GL_4(\R)$ which solve the bundle conjugacy equation \eqref{eq:BundleConjugacy}, repeated below:
\begin{align} \label{eq:BundleConjugacy-SwiftHohenberg}
DG(P(\sigma)) \mathcal{W}(\sigma) = D\mathcal{W}(\sigma)  \Omega^\rs \sigma + \mathcal{W}(\sigma )\mathcal A(\sigma)  .
\end{align}
Note that $\mathcal{W} = [W_1|W_2|W_3|W_4]$ describes a frame bundle over the stable manifold, where we fix $ \cW(0) = \hat{\mathcal{V}}$ as a matrix of eigenvectors of $ DG(0)$. 
The stable bundle is given by $ [W_1|W_2]$ and an unstable bundle is given by $ [W_3|W_4]$. Furthermore  Lemma \ref{L:conjugacy} ensures that if \eqref{eq:BundleConjugacy-SwiftHohenberg} is satisfied, then the dynamics of \eqref{eq:VariationalEq-Resonance} are conjugate (via $\cW$) to the linear system of $ \tilde{V}'(x)  = \mathcal A(e^{\Omega^\rs x}\sigma) \tilde{V}(x)$.  In the case where there are no resonances $\mathcal A$ is a constant coefficient matrix. Otherwise,  $\mathcal A$ is a matrix-valued polynomial in $\sigma$. 

To solve the conjugacy equation we employ the parameterization method, solving for the Taylor coefficients of $ \cW$ and $\mathcal A$ order by order. Suppose that we may write $\cW$ and $\mathcal A$ as power series
\begin{equation}\label{E:defWA_InPractice}
	\cW(\sigma) =  \sum_{|\alpha|=0}^\infty \cW_\alpha \sigma^\alpha, \qquad \mathcal A(\sigma) =  \sum_{|\alpha|=0}^\infty \mathcal A_\alpha \sigma^\alpha .
\end{equation}
Similarly, fix coefficients $ P = \{ P_\alpha \} \in (\ell^1_{\delta,2})^4$ of the stable manifold, and let us write the  Taylor series   
\[
	DG(P(\sigma)) = \sum_{|\alpha|=0}^\infty \hat {G}_\alpha \sigma^\alpha.
\]
As detailed in Section \ref{sec:Resonant-Vector-Bundle-Theory}, matching Taylor coefficients in the conjugacy equation yields the homological equation \eqref{eq:ResonantHomologicalEquation}  repeated below: 
\begin{equation}\label{eq:ResonantHomologicalEquation2}
	(\alpha_1\mu_1 + \alpha_2 \mu_2) \mathcal{W}_\alpha + \mathcal{W}_\alpha  \Omega -DG(0) \mathcal{W}_\alpha + \mathcal{W}_0  \mathcal A_\alpha  = \cS_\alpha,
\end{equation}
where we define  
\begin{equation}\label{E:defS-Swift-Hohenberg}
	\mathcal{S}_\alpha := \hat {G}_\alpha  \cW_0 +  (\hat {G}\hat{*}\cW)_\alpha - (\cW\hat{*}\mathcal A)_\alpha 
\end{equation}
as in \eqref{E:defS}. Since $\cW_0 = \cW(0)$ has as its columns the eigenvectors of $DG(0)$, $DG(0) = \cW_0 \Omega \cW_0^{-1}$, and  we define $\tilde \cW_\alpha$ and $\tilde \cS_\alpha$ via the change of variables
\[
\cW_\alpha = \cW_0 \tilde \cW_\alpha, \qquad \cS_\alpha = \cW_0 \tilde \cS_\alpha.
\]
as we did in \S\ref{sec:Resonant-Vector-Bundle-Theory}. Multiplying the homological equation  \eqref{eq:ResonantHomologicalEquation2} on the left by $\cW_0^{-1}$ yields equation \eqref{eq:ResonantHomologicalEquation-2}, and then restricting to the $j$th column yields the equation 
\begin{equation} \label{eq:ColumnByColumn_InPractice}
	\left[ (\alpha_1 \mu_1^\rs + \alpha_2 \mu_2^\rs) + \mu_j)I -\Omega \right] \tilde w_\alpha^{*, j} + a_\alpha^{*, j} = \tilde s_\alpha^{*, j},
\end{equation}
as in \eqref{eq:ColumnByColumn}. As detailed in Lemma \ref{L:solve-for-AQ}, we are able to solve this equation  order by order for $  \tilde w_\alpha^{*, j} $ and $  a_\alpha^{*, j}$. In the case where there is no resonance, we set $  a_\alpha^{i, j}=0$  and solve for $ \tilde w_\alpha^{i, j} $. In the case where there is a resonance, this is done \emph{visa versa}. 

As in Definition \ref{def:bundle-resonance--Res_ij}, for each fixed $i$ and $j$ we define $\mathrm{Res}_{i,j}$ as the collection of resonances of order $\alpha \in \N^2$ which the $j$th bundle has with the $ i$th bundle, formally given as 
\begin{align}
	\mathrm{Res}_{i,j} = \{ \alpha \in \N^2| \alpha_1 \mu_1^\rs +   \alpha_2 \mu_2^\rs  + \mu_j - \mu_i = 0 \} .
\end{align}
In general there can only be finitely many resonances. For the Swift-Hohenberg equation in the parameter regime we consider  there are just four \footnote{This follows from our eigenvalues coming in a complex quartet of $ \{ \mu_1, \mu_2 ,\mu_3 , \mu_4 \} = \{ \mu ^\rs_1 , (\mu ^\rs_1)^*,-\mu ^\rs_1 , -(\mu ^\rs_1)^*\}$.}; the only non-trivial resonances we have are 
\begin{align*}
	\mathrm{Res}_{1,3} &= \{ (2,0) \} & 	
	\mathrm{Res}_{1,4} &= \{ (1,1) \} \\
	\mathrm{Res}_{2,3} &= \{ (1,1) \} & 	
	\mathrm{Res}_{2,4} &= \{ (0,2) \}. 
\end{align*}
 
For any given order, we only need to solve for  one of   $ \tilde w_\alpha^{i, j} $ or $  a_\alpha^{i, j}$, and set the other to zero. For economy of notation we can represent them both by a single series $ b_\alpha^{i, j}$. 
For each $j$ and $b^{i,j} \in \ell^1_{\delta,2}$ we  associate  $ b^{i,j} \mapsto ( w^{i, j}, a^{i,j}) \in \ell^1_{\delta,2} \times \ell^1_{\delta,2}$ by 
\begin{align} \label{eq:Resonant_Split_Space}
	\tilde w_\alpha^{i,j} &= 
	\begin{dcases}
		0 & \mbox{ if } \alpha \in \mathrm{Res}_{i,j} \\
		b_\alpha^{i,j} & \mbox{ if } \alpha \notin \mathrm{Res}_{i,j} 
	\end{dcases} ,
&
a_{\alpha}^{i,j} & =
	\begin{dcases}
	b_\alpha^{i,j} & \mbox{ if } \alpha \in \mathrm{Res}_{i,j} \\
	0 & \mbox{ if } \alpha \notin \mathrm{Res}_{i,j} \\
\end{dcases} 
\end{align}

As described in Lemma \ref{L:solve-for-AQ}, the solutions for $ \tilde w_\alpha^{i, j} $ or $  a_\alpha^{i, j}$ may be solved recursively in terms of the lower order terms.  
Moreover, if $\mathrm{Res}_{*,j} = \emptyset$ then $b^{*,j}$ only depends on the lower order terms of $ b^{*,j}$. Whereas if  $\mathrm{Res}_{*,j} \neq  \emptyset$, then $b^{*,j}$ may also depend on the lower order terms of $ b^{*,j'}$ for $ j' < j$. 
In particular, we already know from \eqref{E:Aform} that $\mathcal{A}$ will have a  special form, namely
\[
\mathcal{A}(\sigma) = \begin{pmatrix} \mu_1^\rs & a^{1,2}(\sigma) & a^{1,3}(\sigma) & a^{1,4}(\sigma) \\ 0 & \mu_2^\rs & a^{2,3}(\sigma) & a^{2,4}(\sigma) \\ 0 & 0 & \mu_1^\ru & a^{3,3}(\sigma) \\ 0 & 0 & 0 & \mu_2^\ru \end{pmatrix}.
\]
Furthermore, by Lemma \ref{L:solve-for-AQ} and our enumeration of  the resonances listed above, $\mathcal{A}$ may be more explicitly given as  
\begin{align} \label{eq:SH_A_form}
	\mathcal{A}(\sigma_1, \sigma_2) = \begin{pmatrix} \mu_1^\rs & 0 & a^{1,3}_{2,0} \sigma_1^2 & a^{1,4}_{1,1}\sigma_1\sigma_2 \\ 0 & \mu_2^\rs & a^{2,3}_{1,1} \sigma_1 \sigma_2 & a^{2,4}_{0,2} \sigma_2^2 \\ 0 & 0 & \mu_1^\ru & 0 \\ 0 & 0 & 0 & \mu_2^\ru \end{pmatrix}.
\end{align}
Thus, if $ j=3,4$ then $(\cW\hat{*}\mathcal A)^j$ will only depend on $(\cW\hat{*}\mathcal A)^{j'}$ for $j'=1,2$.

These solutions to the homological equation \eqref{E:defS-Swift-Hohenberg} may be seen to be the unique fixed point of the  operator $\Psi: (\ell^1_{\delta,2})^{4\times 4}  \to (\ell^1_{\delta,2})^{4 \times 4}$ we present in Definition   \ref{def:BundlePsi}. 
Before defining $\Psi$ we first define the following linear operators.  
\begin{definition}
For $ i,j \in \{1,\dots 4\}$, we define the linear operators $\kappa^{i,j}:\ell^1_{\delta,2} \to \ell^1_{\delta,2}$, $\cK^j:  (\ell^1_{\delta,2})^{4 }  \to (\ell^1_{\delta,2})^{4 }$, and $\cK:(\ell^1_{\delta,2})^{4\times 4}  \to (\ell^1_{\delta,2})^{4 \times 4}$ as follows. 
\begin{itemize}
\item Fix $ i,j \in \{1,\dots 4\}$.  We define $\kappa^{i,j}(b)$ for $ b \in\ell^1_{\delta,2}$ and $ (m,n) \in \N^2$ by  
\begin{align} \label{eq:KappaDef}
\left(\kappa^{i,j} b\right)_{m,n}
&:= \begin{dcases}
b_{m,n} & \mbox{ if } (m,n) \in \mathrm{Res}_{i,j} \\
\frac{b_{m,n}}{m \mu_1^\rs + n \mu_2^\rs  + \mu_j - \mu_i  } & \mbox{ if } (m,n) \notin \mathrm{Res}_{i,j} 
\end{dcases}.
\end{align} 
\item Fix $ j \in \{ 1, \dots, 4\}$. We define $\cK^j(B^j)$ for $B^j = ( b^{1,j} , \dots , b^{4,j})^T \in (\ell^1_{\delta,2})^{4 }$ by  
\begin{align*}
\cK^j B^j  &:= ( \kappa^{1,j} b^{1,j} , \dots , \kappa^{4,j}  b^{4,j})^T.
\end{align*}
\item We define $\cK(\cB)$ for $\mathcal{B}= (B^1 , \dots , B^4) =  \{b^{i,j}\}_{i,j=1}^4 \in  (\ell^1_{\delta,2})^{4\times 4} $ by 
\begin{align*}
\cK  \odot \cB := (	\cK^1	 B^1 , \dots, \cK^4	 B^4   )  =  \{\kappa^{i,j} b^{i,j}\}_{i,j=1}^4.
\end{align*}
\end{itemize}
\end{definition}
Note that $\kappa^{i,j}$ acts diagonally on a sequence $ b \in \ell^1_{\delta,2}$. For each $ \alpha \in \N^2$, we may represent $ \cK^j_\alpha$ as a diagonal matrix, and its action $\cK^jB^j $ is given by a matrix -- column vector product. To describe the action of $ \cK$, consider some matrix valued sequence $ \mathcal{B} \in  (\ell^1_{\delta,2})^{4\times 4} $  and fix $ \alpha \in \N^2$. Then we may consider both $\cK_\alpha$ and $\cB_\alpha$ as $4 \times 4$ matrices and the action $ \cB_\alpha \mapsto (\cK  \odot \cB)_\alpha$ is given by their element-wise multiplication. 

\begin{definition} \label{def:BundlePsi}
Define a map $ \Psi = ( \Psi^1,\dots \Psi^4)= \{\psi^{i,j}\}_{i,j=1}^4 : (\ell^1_{\delta,2})^{4\times 4}  \to (\ell^1_{\delta,2})^{4 \times 4}$ for $ \mathcal{B}= \{b^{i,j}\}_{i,j=1}^4 \in  (\ell^1_{\delta,2})^{4\times 4} $ via
\begin{align} \label{eq:Bundle_FixedPoint}
\Psi (\cB)_{\alpha} :=  
\begin{dcases}
I & \mbox{if } | \alpha|=0 \\
\cK_\alpha \odot   \left( \cW_0^{-1} \mathcal{S}_\alpha \right) 
& \mbox{if }  | \alpha| \geq 1.
\end{dcases}
\end{align}
where $I$ is the $ 4 \times 4 $ identity matrix, $\cS_\alpha = \cS_\alpha(\mathcal B)$ is given as in \eqref{E:defS-Swift-Hohenberg}, we associate  $ b^{i,j} \mapsto ( \tilde w^{i, j}, a^{i,j})$  as in \eqref{eq:Resonant_Split_Space}, and $\cW_\alpha = \cW_0 \tilde \cW_\alpha$. 
\end{definition}

Note we may expand $ \cW_0^{-1} \cS_\alpha $ more explicitly in terms of $\cW$ below
\begin{align}   
\cW_0^{-1} \cS_\alpha  &=   \cW_0^{-1} \left( \hat {G}_\alpha  \cW_0 +  (\hat {G}\hat{*}\cW)_\alpha - (\cW\hat{*} \mathcal A)_\alpha \right) \nonumber \\
&=  \left( \cW_0^{-1}	\hat {G}_\alpha  \cW_0 +    ( \cW_0^{-1} \hat {G} \cW_0\hat{*}   \tilde \cW )_\alpha -  (\tilde \cW\hat{*}\mathcal A)_\alpha \right) \label{eq:StildeExpand}
\end{align}

\begin{theorem} \label{thm:Swift-Hohenberg-Bundle-Fixedpoint} For $\mathcal{B}= \{b^{i,j}\}_{i,j=1}^4 \in  (\ell^1_{\delta,2})^{4\times 4}$ write $ b^{i,j} \mapsto ( w^{i, j}, a^{i,j})$  as per \eqref{eq:Resonant_Split_Space} and $\mathcal{W},A:B^m_\delta(0) \to \GL_4(\R)$ as per \eqref{E:defWA_InPractice}. If $ \cB = \Psi( \cB)$,   then  the conjugacy equation \eqref{eq:BundleConjugacy-SwiftHohenberg} is satisfied. 
\end{theorem}
\begin{proof}
This follows from Lemma \ref{L:solve-for-AQ}. 
\end{proof}

By calculating the fixed points $ \cB = ( \cB^1 , \dots , \cB^4)$ of  $ \Psi^j$ for $1\leq j \leq 4$, we are thus able to compute the frame bundle $ \cW$ about the stable manifold, which then yields the stable and unstable bundles. It is important to note that, while the stable bundle is an invariant bundle, due to the resonances the unstable bundle, which depends analytically on $ \sigma$, is not invariant; see Remark \ref{rem:Invariant-Not-Analytic}. 

We now explain how the bundle $\cW$ allows us to compute the coefficients $\tilde \beta_j$ and $\tilde \gamma_j$ in \eqref{eq:EigenCoordinates-Application}. We can write this as 
\[
U_1(x) = \tilde \gamma_1 V_1^{+, \ru}(x) + \tilde \gamma_2V_2^{+, \ru}(x) + \tilde \beta_1 V_1^{+, \rs}(x) + \tilde \beta_2 V_2^{+, \rs}(x), 
\]
where $V_{1,2}^{+, \rs, \ru}$ are solutions of the variational equation written in the symplectic coordinates that are asymptotic to the eigendirections in those coordinates. In other words, $V_{1,2}^{+, \rs, \ru}(x) \sim \rme^{\mu_{1,2}^{\rs, \ru}x} \check V_{1,2}^{\rs, \ru}$ as $x \to +\infty$. Once we know the bundle $\mathcal W$, we can use the relation $V(x) = \cW(\rme^{\Omega^\rs x}\sigma) \tilde{V}(x)$ (and a change of coordinates putting everything into the symplectic coordinates) to relate solutions $V(x)$ to solutions $\tilde V$ of $ \tilde{V}'(x)  = \mathcal A(\rme^{\Omega^\rs x}\sigma) \tilde{V}(x)$; see Lemma \ref{L:conjugacy}. 

First, let us charecterize a basis of solutions $\{ \tilde{V}_i\}$. Furthermore, assume that $ \bm{ \varphi}(x_0) = P(\sigma_0)$ and write $\sigma_0 = (\sigma_1, \sigma_2)$. We wish to find a fundamental matrix solution such that 
\[
\tilde M'(x) = \mathcal{A}(\rme^{\Omega^\rs x}\sigma_0)\tilde M, \qquad \tilde M(x_0) = I.
\]
Using the explicit form of $\mathcal A$ given in \eqref{eq:SH_A_form}, we find the columns of $\tilde M$ to be
\begin{align*}
\tilde{V}_1^{\rs}(x) =	&\begin{pmatrix}e^{\mu ^\rs_1 (x-x_0)} \\ 0\\ 0\\ 0 \end{pmatrix},&
\tilde{V}_2^{\rs}(x) =	&\begin{pmatrix} 0\\ e^{\mu ^\rs_2 (x-x_0)} \\ 0\\ 0 \end{pmatrix}, \\
\tilde{V}_1^{\ru}(x) =	&\begin{pmatrix} a^{1,3}_{2,0} (x-x_0) e^{  \mu ^\rs_1 (x-x_0)}\sigma_1^2\\ a^{2,3}_{1,1} (x-x_0)  e^{ \mu^\rs_2 (x-x_0)} \sigma_1\sigma_2\\ e^{\mu ^\ru_1 (x-x_0)} \\ 0 \end{pmatrix},&
\tilde{V}_2^{\ru}(x) =	&\begin{pmatrix} a^{1,4}_{1,1} (x-x_0) e^{ \mu ^\rs_1 (x-x_0)}\sigma_1\sigma_2\\ a^{2,4}_{2,0} (x-x_0) e^{  \mu^\rs_2 (x-x_0) }\sigma_2^2\\ 0\\ e^{\mu ^\ru_2 (x-x_0)} \\ \end{pmatrix}.
\end{align*}
Note the first two exponentially decay, and the second two exponentially grow. 

Next we describe how  we obtain the functions $ V_{1,2}^{+, \rs, \ru}(x) $ from this basis of solutions. Recall the relation $V(x) = \cW(\rme^{\Omega^\rs x}\sigma_0) \tilde{V}(x)$. As $x \to \infty$, $\rme^{\Omega^\rs x}\sigma_0 \to (0, 0)$. Since $\cW(0)$ has columns given by the eigenvectors of $DG(0)$, ie $\cW(0) = [\hat V_1^\rs | \cdots | \hat V_2^\ru]$, then as $x \to \infty$, the solutions $V^{\ru, \rs}_{1,2}(x) =\cW(\rme^{\Omega^\rs x}\sigma_0) \tilde V^{\ru, \rs}_{1,2}(x)$ will be given to leading order by $\rme^{\mu^{\ru, \rs}_{1,2}x} \hat V^{\ru, \rs}_{1,2}$. To write this in symplectic coordinates, recall the change of variables $S$ from \eqref{E:defq}, by which $ \check V_{1,2}^{\rs, \ru} = S \hat V_{1,2}^{\rs, \ru}$. 
Hence, if we let 
\[
V_{1,2}^{+, \rs/\ru}(x) = S\cW(\rme^{\Omega^\rs x}\sigma_0) \tilde V_{1,2}^{\rs/\ru}(x)
\]
then we obtain exactly the asymptotic behavior that we need.  Thus, we may write $U_1$ in $ \tilde \beta , \tilde \gamma$ coordinates as follows
\begin{align}
U_1(x) &= \tilde \beta_1 V_1^{+,\rs}(x) + \tilde \beta_2 V_2^{+,\rs}(x) + 
\tilde \gamma_1 V_1^{+,\ru}(x) + \tilde \gamma_2 V_2^{+,\ru}(x) \nonumber \\
&=
\tilde \beta_1 S\cW(\rme^{\Omega^\rs \tilde x}\sigma_0) \tilde V_1^\rs(x) + 
\tilde \beta_2S\cW(\rme^{\Omega^\rs \tilde x}\sigma_0) \tilde V_2^\rs(x)+ 
\tilde \gamma_1  S\cW(\rme^{\Omega^\rs \tilde x}\sigma_0) \tilde V_1^\ru(x) + 
\tilde \gamma_2 S\cW(\rme^{\Omega^\rs \tilde x}\sigma_0) \tilde V_2^\ru(x) \nonumber \\
&= S\mathcal{W}(\rme^{\Omega^\rs \tilde x}\sigma_0) \tilde{M}(x) \begin{pmatrix} \tilde \beta_1 \\ \tilde \beta_2 \\ \tilde \gamma_1 \\ \tilde \gamma_2 \end{pmatrix} \label{E:U1-form}.
\end{align}

Theorem \ref{thm:Swift-Hohenberg-Bundle-Fixedpoint} asserts that, if we prove that a fixed point of $\Psi$ exists, this fixed point will correspond to the bundle $\mathcal{W}$, and to the matrix $\mathcal{A}$ that leads to the solutions that comprise the columns of $\tilde M$. The matrix $\mathcal{A}$ is a polynomial in $\sigma$, so it can be compute explicitly, and we can compute $\mathcal{W}$ to any order of accuracy we wish. Thus, if we have numerically computed $U_1(x)$ for large $x$, we can compute the coefficients $\tilde \gamma_j$ and $\tilde \beta_j$ via. \eqref{E:U1-form}, or localized at $ x_0$ via
\changes{
\[
\begin{pmatrix} \tilde \beta_1 \\ \tilde \beta_2 \\ \tilde \gamma_1 \\ \tilde \gamma_2 \end{pmatrix} = \mathcal{W}( \sigma_0) ^{-1} S^{-1} U_1(x_0).
\]
}

\subsection{\emph{A posteriori} validation of the bundles}\label{sec:ValidateBundles} 

In practice, we only compute finitely many coefficients of $\cW$, meaning that our Taylor expansion is just an approximation. We will use a fixed point argument to bound the error incurred by this truncation.
Like in Section \ref{S:mflds}, we will do this with computer-assisted proofs to show that  $\Psi$ has a fixed point, by considering the terms of order less than or equal to $N$ and the tail terms separately. 

First let us imbue $ (\ell^1_{\delta,2})^{4\times 4}$ with the ``$p=1$'' operator norm; e.g. for $ \cB\in  (\ell^1_{\delta,2})^{4\times 4}$ then we define 
\[
\| \cB\| = \max_{1 \leq j \leq 4} \sum_{1\leq i \leq 4} \| b^{i,j} \|_{\ell^1_{\delta,2}}
\] 
Using the projection operators defined in \S \ref{S:analytic-framework}, we define the two operators
\begin{align*}\Psi^N: (X^N_{\delta,2})^{4\times4} & \to (X^N_{\delta,2})^{4\times4} 
	& \Psi^\infty:  (\ell^1_{\delta,2})^{4\times 4} &  \to (X^\infty_{\delta,2})^{4\times4}  \\
	\Psi^N & = \Pi_N \circ \Psi 
	& \Psi^\infty & = \Pi_\infty \circ \Psi.
\end{align*}
Note $\Pi_N \circ \Psi = \Pi_N \circ \Psi \circ \Pi_N$. 
Using these operators, for $\cB \in (\ell^1_{\delta,2})^{4\times 4} $ we can decompose $\Psi$ as
\begin{equation}\label{E:separate_bundle}
	\Psi(\cB) = \Psi^N\left( \cB^N\right) + \Psi^\infty\left(  \cB^N + \cB^\infty\right).
\end{equation}

Now suppose that $\bar \cB^N  \in (X^N_{\delta,2})^{4\times 4}$ is a fixed point of $ \Psi^N$. This may be solved numerically by solving finitely many equations; here interval arithmetic is able to bound all of the numerical error. Define the map $T: (X^\infty_{\delta,2})^{4 \times 4} \to (X^\infty_{\delta,2})^{4 \times 4}$ via 
\begin{equation}\label{eq: tail operator T_bundle}
	T(\cB^\infty ) :=  \Psi^\infty\left( \bar \cB^N + \cB^\infty\right).
\end{equation}
Note for $N>2$ there are no more resonant terms, so by \eqref{eq:Resonant_Split_Space} we have  $ \cB_\alpha \equiv \tilde \cW_\alpha$. Let us also write $ \tilde{\cW}^N  = \bar  \cB^N$ and $ \bar{\cW}^N  =\cW_0 \bar  \cB^N$.
 
Hence, we have for $\cW^\infty\in (X^\infty_{\delta,2})^{4\times 4}$ that 
\[
T(\tilde \cW^\infty) = \Pi_\infty \circ \Psi( \tilde{\mathcal{W}}^N + \tilde{\cW}^\infty)= \Pi_\infty \mathcal{K} \odot   \left( \cW_0^{-1} \, \mathcal{S}    \right).
\]
Using \eqref{eq:StildeExpand} to expand $\mathcal{S} \equiv\mathcal{S} ( \tilde{\mathcal{W}}^N + \tilde{\cW}^\infty ) $, we obtain 
\begin{align*}
T(\tilde \cW^\infty) &=   \Pi_\infty \mathcal{K} \odot  \left( \cW_0^{-1}	\hat {G}  \cW_0 +     \cW_0^{-1} \hat {G} \cW_0\hat{*} ( \tilde{\mathcal{W}}^N + \tilde{\cW}^\infty)  -    \tilde{\mathcal{W}} \hat{*}\mathcal A \right)   \\
&= \Pi_\infty \mathcal{K} \odot  \left( \cW_0^{-1}	\hat {G}  \cW_0 + \cW_0^{-1} \hat {G} \hat{*} \bar{\cW}^N +  \cW_0^{-1} \hat {G} \cW_0\hat{*} \tilde{\cW}^\infty  -  ( \tilde{\mathcal{W}}^N + \tilde{\cW}^\infty)\hat{*}\mathcal A \right) .
\end{align*}
To simplify, note that  $\bar{\cW}^N_0 =\cW_0 $ and $\Pi_\infty \bar{\cW}^N  = 0$, whereby
\[
\Pi_\infty \left( \hat {G}  \cW_0 +   \hat {G} \hat{*} \bar{\cW}^N \right)_\alpha = \Pi_\infty \left( \hat G_\alpha  \bar{\cW}^N_0 + \hat G_0  \bar{\cW}^N_\alpha +(\hat G \hat* \bar{\cW}^N)_\alpha \right)
= \Pi_\infty ( \hat G * \bar{\cW}^N)_\alpha	 .
\]
Thereby, we obtain
\begin{align}
	\label{eq:T_expand}
T(\tilde \cW^\infty) &=    \Pi_\infty \mathcal{K} \odot  \left( 
	\cW_0^{-1}  \hat G * \bar{\cW}^N	
	+( \cW_0^{-1} \hat G \cW_0 ) \hat{*}\tilde \cW^\infty  -  ( \bar{\mathcal{W}}^N + \tilde{\cW}^\infty)\hat{*}\mathcal A \right) 
\end{align}
Since $\mathcal A$ is fixed after $|\alpha| >2$, $T$ is an affine linear mapping on $ \tilde\cW^\infty$. 

\begin{remark} If we consider the $j$th column of \eqref{eq:T_expand}, then we obtain 
\begin{align*}
T^j(\tilde \cW^\infty) &= \Pi_\infty \cK^j   \left( \cW_0^{-1}  \hat G * (\bar{\cW}^{N})^{j}+( \cW_0^{-1} \hat G \cW_0 ) \hat{*}\tilde \cW^{\infty,j} - \left((\tilde{\cW}^N+\tilde{\cW}^\infty) \hat{*} \mathcal A \right)^j\right).  
\end{align*}
The product $ (\tilde{\mathcal{W}}  \hat * \mathcal{A})^j$  may be simplified using the form of $ \mathcal{A}$ in \eqref{eq:SH_A_form} as 
\begin{align*}
({\tilde \cW} \hat{*} \mathcal A)^j_\alpha &= \sum_{i<j} 	(\tilde{\cW}^{i} \hat{*} a^{i,j} )_\alpha  = \sum_{i<j} \sum_{ 0 < |\beta | < | \alpha| }  {\tilde \cW}^{i}_{\alpha-\beta}   a^{i,j}_\beta = \sum_{i<j} \sum_{ \substack{ \beta \in \mathrm{Res}_{i , j}  \\	|\beta| \neq |0|, |\alpha|}}  	{\tilde \cW}^{i}_{\alpha-\beta}   a^{i,j}_\beta.
\end{align*}

Here, we were able to restrict the sum over $i$ to  $ i< j$, due to the upper triangular structure of \eqref{eq:SH_A_form}; we were able to restrict the sum over $ \beta $ to $\mathrm{Res}_{i , j} $, as that is the index set for which $a^{i,j}_\beta$ is nonzero. Since the set $\mathrm{Res}_{i , j}$ contains elements, if any, only of order $2$, it follows that  for $|\alpha| > 2$ we obtain 
\begin{align} \label{eq:ResonantProductExpansion}
({\tilde \cW} \hat{*} \mathcal A)^j_\alpha &	= \sum_{i<j} \sum_{ \substack{ \beta \in \mathrm{Res}_{i , j}   }}  	
{\tilde \cW}^{i}_{\alpha-\beta}   a^{i,j}_\beta .
\end{align}
Note that if $ j=1,2$ then  $ \mathrm{Res}_{i , j} = \emptyset $, whereby the sum vanishes, and $	T^j(\tilde \cW^\infty)$ only depends on the  $\tilde{\cW}^{j}$. This is in fact the nonresonant case, and like in \cite{vandenBergJames16}, these bundles can be independently solved for. On the other hand if $ j=3,4$, then $ \mathrm{Res}_{i , j} \neq \emptyset$ for $ i=1,2$. That is to say that $T^j(\tilde \cW^\infty)$ depends both on $\tilde \cW^{j}$ as well as $\tilde \cW^{i}$ for $ i=1,2$.
\end{remark}
 
Note that  $T$ acts as a self map  on the tail space $(X^\infty_{\delta, 2})^4$. Since solving the recursive relations to order $N$ gives a fixed point of $\Psi^N$, we only need to prove the existence of a fixed point of the map $T$ to prove there is a fixed point of the full map $\Psi$. We do this via the parameterized Newton-Kantorovich Theorem,  Theorem \ref{thm: Radii Polynomial approach for Contraction Mapping Theorem}. 

Given the data $\bar \cB^N \in  (X^N_{1,2})^{4\times 4}$, we wish to show that $T$ is a contraction in a small neighborhood of the origin,  $\overline{B_r(0)} \subset (X^\infty_{1,2})^{4\times 4}$. For this we need to compute bounds   $Y_0 > 0$ and $Z: (0, \infty) \to [0, \infty)$ such that
\[
\|T(\bar x) - \bar x \|_X  \leq Y_0, \qquad \sup_{x \in \overline{B_r(\bar x)}}\| DT(x)\|_{L(X)}  \leq Z(r).
\]
Note here that $ \bar{x} =0$. Further, since $T$ is affine linear, then $ Z(r) = Z$ is a constant. We prove lemmas before consolidating everything together. 

The first issue to address is that   our parameterization method is defined in terms of the coefficients $ \hat{G}$, which in turn depend on the coefficients $P$ of the stable manifold. By \S \ref{S:mflds} we can   compute finitely many coefficients of $P$ and bound the tail. Consequently, this will yield us finitely many of the coefficients of $\hat G$ and a bound on the remainder. 
 
Suppose $ P = P^N + P^ \infty \in (X^N_{\delta,2})^4 \oplus (X^\infty_{\delta,2})^4$ are the coefficients of the stable manifold. We define $ \hat G^N$ as the Taylor coefficients of $ DG( P^N(\sigma))$, and $ \hat {G}^\infty = \hat G - \hat G^N$.  We warn the reader that with this definition $ \hat G^N \neq \pi_N\hat G$, however it is that case that $ \pi_{2N} \hat G^{N} = \hat G^{N}$ and $ \pi_N \hat G^{\infty} =0$. This follows from the fact that the nonlinearity in SH is cubic, and hence $DG$ is quadratic.
 
\begin{lemma}\label{lem:ghatinfty} If $ \| P^\infty \| \leq \epsilon$ then $\| \hat G^\infty\| \leq   2\nnu \epsilon + 3( 2 \| p^1\| \epsilon + \epsilon^2)  =: \epsilon_\infty.$
\end{lemma}
 
\begin{proof}
	Recall that  $DG(U)$ is given by 
	\begin{align*}
		DG(U) & = \left( \begin{smallmatrix}0 & 1 & 0 & 0 \\ 0 & 0 & 1 & 0\\ 0 & 0 & 0 & 1 \\ -1+f'(u) & 0 & -2 & 0 \end{smallmatrix} \right),
		&
		f'(u) &=- \nmu+2\nnu u - 3u^2  .
	\end{align*}
	where $ u= (U)_1$. 
Hence, 
	\begin{align*}
		\hat{G}  = 
		\left( \begin{smallmatrix}
			0 & 1 & 0 & 0 \\ 
			0 & 0 & 1 & 0\\ 
			0 & 0 & 0 & 1 \\ 
			-1 - \nmu & 0 & -2 & 0
		\end{smallmatrix} \right)
		+ \left(
		2 \hat{\nu} p^1  -3 ( p^1 * p^1)
		\right) 
		\left(\begin{smallmatrix}
			0 & 0 & 0 & 0 \\ 
			0 & 0 & 0 & 0\\ 
			0 & 0 & 0 & 0 \\ 
			1& 0 & 0 & 0  \end{smallmatrix}\right)
	\end{align*}
where $ P = (p^1 , p^2,p^3,p^4)$ are the coefficients of the stable manifold.  Let $p^1 =  (p^1)^N + (p^1)^\infty$ and suppose that $ \| (p^1)^\infty \| \leq \epsilon$. Let us define 
	\begin{align*}
		\hat{G}^N  = 
		\left( \begin{smallmatrix}
			0 & 1 & 0 & 0 \\ 
			0 & 0 & 1 & 0\\ 
			0 & 0 & 0 & 1 \\ 
			-1 - \nmu & 0 & -2 & 0
		\end{smallmatrix} \right)
		+ \left(
		2 \hat{\nu} (p^1)^N  -3 ( (p^1)^N * (p^1)^N)
		\right) 
		\left(\begin{smallmatrix}
			0 & 0 & 0 & 0 \\ 
			0 & 0 & 0 & 0\\ 
			0 & 0 & 0 & 0 \\ 
			1& 0 & 0 & 0  \end{smallmatrix}\right).
	\end{align*}
Then we have that 
	\begin{align*}
	\hat {G}^\infty  =   \left(
		2 \hat{\nu} (p^1)^\infty  -3 ( 2 (p^1)^N * (p^1)^\infty+(p^1)^\infty * (p^1)^\infty)
		\right) 
		\left(\begin{smallmatrix}
			0 & 0 & 0 & 0 \\ 
			0 & 0 & 0 & 0\\ 
			0 & 0 & 0 & 0 \\ 
			1& 0 & 0 & 0  \end{smallmatrix}\right)
	\end{align*}
Hence $ \| \hat {G}^\infty  \| \leq 2\nnu \epsilon + 3( 2 \| p^1\| \epsilon + \epsilon^2) =2 \epsilon (\nnu + 3 \| p^1 \| ) +3 \epsilon^2$.	
\end{proof}

\begin{lemma} \label{prop:K_N_bundle} For $j \in \{1,\dots,4\}$ consider $ \Pi_\infty  \cK^j : (X^\infty_{\delta,2})^4 \to (X^\infty_{\delta,2})^4$ and define the constant 
\[
K_N^j :=
\begin{dcases}
\frac{1}{(N+1) \sqrt{\rho} \cos(\theta/2)  } & \mbox{ if } j=1,2 \\
\frac{1}{N \sqrt{\rho} \cos(\theta/2)  } & \mbox{ if } j=3,4
\end{dcases}.
\]
If $N>2$ then $\| \Pi_\infty  \cK^j \| \leq K_N^j$. 
\end{lemma}
\begin{proof}
From \eqref{eq:KappaDef} we have the operator norm 
\begin{align*}
\| \Pi_\infty  \cK^j \| &=	\sup_{m+n \geq N + 1}   \max_{i, \rs/\ru} \left\{ \frac{1}{|\mu_i^{\rs/\ru} - m \mu_1^\rs - n \mu_2^\rs - \mu_j|} \right\}. 
\end{align*}
The result follows by a calculation analogous to the Lemma \ref{lemma: SH radii poly for mflds}, in particular equation \eqref{eq:K_N_bound}.
\end{proof}

\begin{lemma} \label{prop:Bundle_YBound} For $j \in \{1,\dots,4\}$, define 
\begin{align*}
Y^{j,a}_0 &:=		\sum_{N+1 \leq |\alpha| \leq 3N} \Big|
		\cK^j_{\alpha} \cW_0^{-1}    \left(\hat {G}^N * (\bar{\cW}^{N})^{j}	 \right)_\alpha  
		\Big| \delta^{|\alpha|} \\		
		Y^{j,b}_0 &:=
		K^j_N
		\| 
		\cW_0^{-1}  \| \epsilon_\infty \|  (\bar{\cW}^{N})^{j}	 \| 	\\
		Y^{j,c}_0 &:=	
		\begin{dcases}
			0 & \mbox{ if } j=1,2 \\
			\sum_{i=1,2}
			\sum_{\beta \in \mathrm{Res}_{i , j}} 
			|a^{i,j}_\beta| 
			\sum_{|\alpha| = N+1}^{N+2} \Big|
			\cK^j_\alpha   	({\tilde \cW}^{N})^{i}_{\alpha-\beta}   
			\Big| \delta^{|\alpha|} 
			 & \mbox{ if } j=3,4 \\
		\end{dcases}
\end{align*}
If we define  $ Y^{j}_0 = Y^{j,a}_0 + Y^{j,b}_0 + Y^{j,c}_0$,  then $\|T^j(0) \|  \leq Y_0^{j}$.
\end{lemma}

\begin{proof}
From \eqref{eq:ResonantProductExpansion} 	and expanding out $ \hat{G}= \hat{G}^N + \hat{G}^\infty$,
		\begin{align*}
			T^j(0) &= \Pi_\infty \cK^j   \left(
			\cW_0^{-1}  (\hat{G}^N + \hat{G}^\infty) * \bar{\cW}^{N,j}	
			- (\tilde{\cW}^N \hat{*} \mathcal{A})^j
			\right)  
		\end{align*}
We will obtain bounds $Y^{j,a}_0 ,Y^{j,b}_0 , Y^{j,c}_0$ such that 
	\begin{align*}
		\| \Pi_\infty \cK^j   \left(
		\cW_0^{-1}  \hat G^N * (\bar{\cW}^{N})^{j}		\right) \| & \leq Y^{j,a}_0	
		\\
		\| \Pi_\infty \cK^j   \left(
		\cW_0^{-1}  \hat G^\infty * (\bar{\cW}^{N})^{j}		\right) \| & \leq Y^{j,b}_0	
				\\
		\| \Pi_\infty \cK^j   \left( \tilde{\cW}^N \hat{*} \mathcal{A}	\right)^j\| & \leq Y^{j,c}_0	
	\end{align*}
The first bound may be computed rather explicitly:
	\begin{align*}
			\| \Pi_\infty \cK^j   \left(
		\cW_0^{-1}  \hat G^N * (\bar{\cW}^{N})^{j}		\right) \|  & \leq  \sum_{|\alpha| \geq N+1} \Big|
		\cK^j_{\alpha}\cW_0^{-1}    \left(\hat {G}^N * (\bar{\cW}^{N})^{j}	 \right)_\alpha  
		\Big| \delta^{|\alpha|}  
		\\
		 & = \sum_{N+1 \leq |\alpha| \leq 3N} \Big|
		\cK^j_{\alpha} \cW_0^{-1}    \left(\hat {G}^N *(\bar{\cW}^{N})^{j}	 \right)_\alpha  
		\Big| \delta^{|\alpha|} =: Y^{j,a}_0
	\end{align*}
Note that the sum terminates at order $3N$ because  $ \Pi_{2N}  \hat {G}^N = \hat {G}^N$. For next bound, using Lemmas \ref{lem:ghatinfty} and \ref{prop:K_N_bundle} we find
\begin{align*}
		\| \Pi_\infty \cK^j   \left(
	\cW_0^{-1}  \hat G^\infty * (\bar{\cW}^{N})^{j}		\right) \|  
	&\leq 
	\| \Pi_\infty \cK^j  \| 
	\| 
	\cW_0^{-1}  \| \| \hat G^\infty * (\bar{\cW}^{N})^{j}		 \| 
	\\
	 &\leq 
 K^j_N
	\| 
	\cW_0^{-1}  \| \epsilon_\infty \|  (\bar{\cW}^{N})^{j}		 \| 	=: Y^{j,b}_0.
\end{align*}
For the final bound, which involves the resonant terms, first note that by  \eqref{eq:ResonantProductExpansion} it follows that 
\begin{align*}
		\| \Pi_\infty \cK^j   \left( \tilde{\cW}^N \hat{*} A	\right)^j\| 
		&\leq 
		\sum_{|\alpha| \geq N+1} \Big|
		\cK^j_\alpha   \left( \tilde{\cW}^N \hat{*} A	\right)^j_\alpha
		\Big| \delta^{|\alpha|} 
		\\
				&\leq 
		\sum_{|\alpha| \geq N+1} \Big|
		\cK^j_\alpha   \sum_{i<j} \sum_{\beta \in \mathrm{Res}_{i , j}}  	({\tilde \cW}^{N})^{i}_{\alpha-\beta}   a^{i,j}_\beta
		\Big| \delta^{|\alpha|} 
\end{align*} 
Also note that 
\[
\{ | \beta |  : \beta \in 	\mathrm{Res}_{i , j}  \} = 
\begin{dcases}
	\{0\} & \mbox{if } j =1,2 \\
	\{2\} & \mbox{if } j=3,4
\end{dcases}
\]
Hence
\begin{align*}
	\| \Pi_\infty \cK^j   \left( \tilde{\cW}^N \hat{*} A	\right)^j\| 
	&\leq 
	 \sum_{i<j}
	  \sum_{\beta \in \mathrm{Res}_{i , j}} 
	\sum_{|\alpha| = N+1}^{N+2} \Big|
	\cK^j_\alpha   	({\tilde \cW}^{N})^{i}_{\alpha-\beta}   a^{i,j}_\beta
	\Big| \delta^{|\alpha|} 	=: Y^{j,c}_0 
\end{align*} 

\end{proof}

\begin{lemma} \label{prop:Bundle_ZBound}
		For $j \in \{1,\dots,4\}$, define 
	\begin{align*}	
		Z^{j,a} &:=	
		K_N^j 
		\sum_{1 \leq  |\alpha| \leq 2N }
		\left\|
		\cW_0^{-1}  \hat{
			G}^N_{\alpha} \cW_0 		 \right\| 	\\
		Z^{j,b} &:=	K_N^j  \epsilon_\infty  
		\|
		\cW_0^{-1} \|   \|  \cW_0 		\| \\
		Z^{j,c} &:=	
		\begin{dcases}
				0 & \mbox{ if } j=1,2 
				\\
				\left(
			\sum_{i=1,2} \sum_{\beta \in \mathrm{Res}_{i , j}}  |a^{i,j}_\beta |	 \right)
			K_{N+2}^j 
			& \mbox{ if } j=3,4 
		\end{dcases}
	\end{align*}
	If we define  $ Z^{j} = Z^{j,a} + Z^{j,b} + Z^{j,c}$,  
then  
\[
\sup_{\cB^\infty \in \overline{B_r(0)}}\| DT^j(\cB^\infty)\|   \leq  Z^j.
\]
\end{lemma}

\begin{proof}From \eqref{eq:T_expand} we compute the Frech\'et derivative of $T$ acting on a vector $ \mathcal{H}^{\infty} \in (X^\infty_{\delta,2})^{4\times4}$. 
	\begin{align*}
		DT( \cW^{\infty}) \mathcal{H}^\infty &= 
		\Pi_\infty \cK\odot \left(
(		\cW_0^{-1}  \hat{
G} \cW_0 )		\hat{*} \mathcal{H}^{\infty}
- \mathcal{H}^\infty \hat{*} A
		\right)  
	\end{align*}
Note that as $T$ is affine linear, then $	DT( \cW^{\infty})$ is constant in $ \cW^{\infty}$, and $DT( \cW^{\infty})=DT(0)$. Moreover  
\begin{align*}
	\|	DT( \cW^{\infty}) \| = \sup_{\| \mathcal{H}^\infty \| =1}  \|	DT( \cW^{\infty}) \mathcal{H}^\infty \| 
	 &= \sup_{\| \mathcal{H}^\infty \| =1}  
	 \max_{1 \leq j \leq 4}
	  \|	DT^j( 0 ) \mathcal{H}^\infty \|. 
\end{align*}
Fix a vector $ \mathcal{H}^{\infty} \in (X^\infty_{\delta,2})^{4\times4}$ and assume $\| \mathcal{H}^{\infty} \| =1$. 
Putting into components we obtain 
\begin{align*}
	DT^j( 0 ) \mathcal{H}^\infty &= 
	\Pi_\infty \cK^j   \left(
	(		\cW_0^{-1}  \hat{
		G} \cW_0 )		\hat{*} \mathcal{H}^{\infty,j}
	- (\mathcal{H}^\infty \hat{*} A)^j
	\right)  
\end{align*}
Expanding out $ \hat{G}= \hat{G}^N + \hat{G}^\infty$, we will obtain bounds $Z^{j,a} ,Z^{j,b} , Z^{j,c}$ such that 
\begin{align*}
	\| \Pi_\infty \cK^j   \left(
	\cW_0^{-1}  \hat G^N \cW_0\hat{*} \mathcal{H}^{\infty,j}	\right) \| & \leq Z^{j,a}	
	\\
	\| \Pi_\infty \cK^j   \left(
	\cW_0^{-1}  \hat G^\infty\cW_0 \hat{*} \mathcal{H}^{\infty,j}	\right) \| & \leq Z^{j,b} 
	\\
	\| \Pi_\infty \cK^j   \left( \mathcal{H}^{\infty} \hat{*} A	\right)^j\| & \leq Z^{j,c}	
\end{align*}
For the first part,   
\begin{align*}
\|	\Pi_\infty \cK^j   \left(
	(		\cW_0^{-1}  \hat{
		G}^N \cW_0 )		\hat{*} \mathcal{H}^{\infty,j}  \right) \|
	& \leq 
	 \|	\Pi_\infty \cK^j   \| \left\|
	(		\cW_0^{-1}  \hat{
	G}^N \cW_0 )		\hat{*} \mathcal{H}^{\infty,j}  \right\|
\\
	& \leq 
K_N^j 
\left( \sum_{1 \leq  |\alpha| \leq 2N }
\left\|
		\cW_0^{-1}  \hat{
	G}^N_{\alpha} \cW_0 		 \right\| 
\right) 
\, \|  \mathcal{H}^{\infty,j}   \|	\\
& \leq K_N^j 
\sum_{1 \leq  |\alpha| \leq 2N }
\left\|
\cW_0^{-1}  \hat{
	G}^N_{\alpha} \cW_0 		 \right\| 
=: Z^{j,a}
\end{align*}
For the second part, 
\begin{align*}
	\|	\Pi_\infty \cK^j   \left(
	(		\cW_0^{-1}  \hat{
		G}^\infty \cW_0 )		\hat{*} \mathcal{H}^{\infty,j}  \right) \|
	& \leq 
	\|	\Pi_\infty \cK^j   \| \left\|
	(		\cW_0^{-1}  \hat{
		G}^\infty \cW_0 )		\hat{*} \mathcal{H}^{\infty,j}  \right\|
	\\
	& \leq 
	K_N^j 
	\|
	\cW_0^{-1} \|  \| \hat{
		G}^{\infty} \| \|  \cW_0 		\|  \, \|  \mathcal{H}^{\infty,j}   \|	\\
	& \leq K_N^j  \epsilon_\infty  
	\|
	\cW_0^{-1} \|   \|  \cW_0 		\| 
	=: Z^{j,b} \end{align*}
For the third part, we expand out the resonant terms as per  \eqref{eq:ResonantProductExpansion} 
\begin{align*}
	\Pi_\infty \cK^j   \left(
	   \mathcal{H}^\infty \hat{*} A   \right)^j 
	&  = \sum_{| \alpha| \geq N+1 }\cK^j_\alpha    \left(
	\mathcal{H}^\infty \hat{*} A   \right)^j_\alpha \\
	& =   \sum_{| \alpha| \geq N+1 } \cK^j_\alpha  	\sum_{i<j} \sum_{\beta \in \mathrm{Res}_{i , j}}  	\mathcal{H}^{\infty,i}_{\alpha-\beta}  a^{i,j}_\beta \\
	&=  \sum_{i<j} \sum_{\beta \in \mathrm{Res}_{i , j}}  a^{i,j}_\beta 	  \sum_{| \alpha| \geq N+3 }  \cK^j_\alpha  	\mathcal{H}^{\infty,i}_{\alpha-\beta}   
\end{align*} 
We note that the sum may be taken over $ |\alpha| \geq N+3$, because the nonzero terms of $\mathcal{H}^{\infty}_{\alpha'}$ are all of order $ |\alpha' | \geq N+1$, and the resonant terms $\mathrm{Res}_{i , j}$, if any, are all of order $2$. If $j=1,2$ there are no resonant terms, so this sum is just zero. Otherwise  
\begin{align*}
	\| 	\Pi_\infty \cK^j   \left(
	\mathcal{H}^\infty \hat{*} A   \right)^j 
  \| &\leq \left \| 
		 \sum_{i<j} \sum_{\beta \in \mathrm{Res}_{i , j}}  a^{i,j}_\beta 	  \sum_{| \alpha| \geq N+3 }  \cK^j_\alpha  	\mathcal{H}^{\infty,i}_{\alpha-\beta}   \right\| \\
		 &\leq 
		 \left(
		 \sum_{i<j} \sum_{\beta \in \mathrm{Res}_{i , j}}  |a^{i,j}_\beta |	 \right)
		 K_{N+2}^j
		 \| 	\mathcal{H}^{\infty,i}   \|  
		 \\
		  &\leq 
		 \left(
		 \sum_{i=1,2} \sum_{\beta \in \mathrm{Res}_{i , j}}  |a^{i,j}_\beta |	 \right)
				 K_{N+2}^j  =: Z^{j,c},
\end{align*}
thus obtaining the desired bounds.   
\end{proof}

We have the following corollary that inserts the bounds into Theorem \ref{thm: Radii Polynomial approach for Contraction Mapping Theorem}. 
 
\begin{corollary}\label{prop: SH radii poly for bundle}
	
	Suppose $\bar \cB^N  \in (X^N_{\delta,2})^{4\times 4}$ is a fixed point of $ \Psi^N$. 
	Suppose that $P= P^N + P^\infty$ is a parameterization of the stable manifold and $ \| P^\infty \| < \epsilon$.  
	Let $T$ be defined as in \eqref{eq: tail operator T_bundle}. Suppose $ Y^j_0$ and a $Z^j$ are given as per Lemmas \ref{prop:Bundle_YBound} and \ref{prop:Bundle_ZBound} and define 
	\begin{align}
		Y_0 &= \max_{1 \leq j \leq 4} Y_0^j &
		Z &= \max_{1 \leq j \leq 4} Z^j 
	\end{align}
	If there exists an $r_0 > 0$ such that
	$\mathfrak p(r_0) < 0$, where
	\begin{equation}\label{def: radii poly Bundle}
		\mathfrak p(r) = Z r - r + Y_0,
	\end{equation}
	then there exists a unique $\cB^\infty_* \in \overline{B_{r_0}(0 )}$ so that $T(\cB^\infty_*) = \cB^\infty_*$. 
	Furthermore $ \bar \cB^N +\cB^\infty_*= \Psi( \bar \cB^N +\cB^\infty_*)$. 
\end{corollary} 

\begin{proof} This follows directly from Theorem \ref{thm: Radii Polynomial approach for Contraction Mapping Theorem}, the definition of $T$ in \eqref{eq: tail operator T_bundle}, and Lemmas \ref{prop:Bundle_YBound} and \ref{prop:Bundle_ZBound}. 	
\end{proof}

\begin{remark}
Note that the polynomial $\mathfrak p$ in \eqref{def: radii poly Bundle} is linear: $\mathfrak p(r) = -(1-Z)r + Y_0$. Thus, if $|Z| < 1$, then any choice of $r_0 > \frac{Y_0}{1-Z}$ will work. Furthermore, since $Z$ is proportional to $K_N^j$, which can be made as small as we like by choosing $N$ large, we can ensure a fixed point exists.  
\end{remark}


\section{Computer-assisted proofs and computational results} 
\label{S:CAP}

To summarize, in Section \ref{S:mflds} we described how the parameterization method may be used to rigorously compute local parameterizations of the (un)stable manifolds of the equilibrium at the origin in \eqref{E:exist}. In Section \ref{S:nonres methods} we described how such a parameterization may be used to compute $-L_{\mathrm{conj}}^-$, a bound on the smallest possible conjugate point. 
Finally in Section \ref{Ch:res bundles} we detailed a new methodology, based on the parameterization method, to rigorously compute resonant and non-resonant vector bundles attached to the stable manifold. 

Many of these theorems, such as Corollary \ref{prop: SH radii poly for bundle}, are formulated in an \emph{a posteriori} format; they state that, if certain explicitly computable bounds are satisfied (e.g. there exists some $r_0$ for which $\mathfrak p(r_0) < 0$), then the desired result holds. When combined with rigorous numerics, this yields computer-assisted proofs. While we can guarantee that there are sufficient conditions for which the theorem will hold (e.g. by computing sufficiently many terms in the Taylor expansion, or restricting the manifold to a sufficiently small neighborhood), we can do much better. We are able to show that these methods can produce highly accurate results for a large local chart of the manifolds, thus being able to be realistically applied in practice. 


\subsection{Example computation of  the (un)stable manifolds and  $L^-_{\mathrm{conj}}$}

In the following we describe the results of our implementation of computer-assisted proofs of the above described results. The code associated with this paper is available as open source at \cite{codesResonantBundles}, and it is flexible to work with various system parameters $\nmu , \nnu$ from \eqref{E:exist}. Computations are performed using MATLAB's 64-bit interval arithmetic via Intlab \cite{rump1999intlab}. 

In the following we describe an example computation where we took $\nmu = 0.2$ and $\nnu = 1.6$. The first step is to compute the parameterization of the (un)stable manifolds. The two important computational parameters at hand are $N$, the number of terms in the Taylor expansion (we fix $N=35$), and the scaling of the eigenvectors used as the first order terms in the parameterization method. The scaling of the first order terms are intimately tied to size of the domain on which the parameterization method is valid, see Remark \ref{rmk: scaling the coefficients}. For $DP(0) = [P_{10}, P_{01}]$, we chose the scaling of the eigenvectors as $\|P_{10}\|=\|P_{01}\| = 1/2$.  Also we fix $\delta = 1$, see Remark \ref{rmk: scaling the coefficients}. This scaling choice results in a large parameterization of the manifolds; see Figure \ref{fig:Manifold_Bundle}.   

Described in Table \ref{table:ManifoldBounds} are the bounds relavent to the hypothesis of Lemma \ref{lemma: SH radii poly for mflds}.  We also note that throughout we have rounded all bounds upward to three significant digits. In sum, we are able to achieve a computer-assisted proof which validates the hypothesis of Lemma \ref{lemma: SH radii poly for mflds}  using a radius of $ r_0 = 2.95 \cdot 10^{-16}$. The computation and results for the unstable manifold are virtually identical, which is not surprising since Hamiltonian systems are reversible. 

\begin{table}[H]
	\centering
	\begin{tabular}{|r|l|l|l|} 
		\hline
		&
		$Y_0$  &  $Z_1$  &  $Z_2(r) $  \\		
		\hline 
		stable manifold &
		$4.90 \cdot 10^{-18} $& 
		$		0.983  $& 
		$		2.16 + 0.899 r $ \\
		\hline 
		unstable manifold &
		$4.90 \cdot 10^{-18} $& 
		$		0.983  $& 
		$		2.16 + 0.899 r $ \\
		\hline
	\end{tabular}
	\caption{Bounds for Manifolds.}
	\label{table:ManifoldBounds}
\end{table}

We note that the primary difficulty in these bounds is computation of $Z_1$, which at a minimum is required to be less than one in order for us to have a contraction. The element of our proof that allows us to control $Z_1$ is $K_N$, which for this computer-assisted proof turns out to be $ K_N = 0.300$. While $ K_N = \mathcal{O}(1/N)$ as $N \to \infty$, this requires computing $ \mathcal{O}(N^2)$ coefficients in the parameterization, to speak nothing of the scaling of the computation time. A contributing factor to why $K_N$ is relatively large is due to its dependence on the real part of $ \mu^\rs$. In our estimate \eqref{eq:K_N_bound}, we had to account for cancelations in the imaginary part of the two complex eigenvalues, so our bound only depended on the real part. This is somewhat reminiscent of a small divisors problem. Furthermore, we not that as $\hat \mu \to 0 $, $\mathrm{Re}(\mu^\rs_1)  \to 0$, whereby $K_N \to \infty$, so we imagine the proof would be difficult in this scenario. This would also be very much expected in a case where the spectral gap is going to zero. 

After computing the unstable manifold we are able to calculate  $L_{\mathrm{conj}}^- $. Using  Theorem \ref{thm: compute L-}, and Proposition \ref{prop: L-_alt}, we are able to compute that a bound of $L_{\mathrm{conj}}^- = 47.1$ is sufficient to guarantee the nonexistence of conjugate points on the domain $(-\infty, -L^-_\mathrm{conj}]$. Admittedly, as detailed in Remark \ref{rem:levers}, this particular value is lacking context due to a freedom we have in scaling. To set $L_{\mathrm{conj}}^- $ in perspective consider Figure \ref{fig:Manifold_Bundle}, where the stable manifold $P(\sigma)$ is plotted with $ |\sigma | \leq 1$ (the plot of the unstable manifold is largely analogous).  If there is a pulse with $ \bm{ \varphi } (x)  =  P(\sigma)$, then there cannot be any conjugate points when $| \sigma|  \leq 3.43 \cdot 10^{-5}$. This domain is several orders of magnitude smaller than the region on which we can parameterize the (un)stable manifold, demonstrating the power of the parameterization method. Nevertheless, this is not an especially small bound for $\sigma$ (equivalently, not an especially large bound for $L_{\mathrm{conj}}^- $), which would cause us to search for conjugate points over a large region. 
 
It bears remarking upon the manner in which we have plotted  Figure \ref{fig:Manifold_Bundle}. While the origin has complex eigenvalues and the parameterization of the stable  manifold has complex coefficients, we are just interested in the real manifold; see Lemma \ref{lem:real-valued}. To that end, set $\sigma_1(s, t) = s + it$ and $\sigma_2(s,t) = s - it$. The real-valued stable manifold parameterization of order $N = 35$ is then given by $P^{(N)}: \C^2  \to \R^4 $ with
\begin{align*}
	P^{(N)}(\sigma_1, \sigma_2)  &= \sum_{|\alpha| \leq N} P_{\alpha}(s + it)^{\alpha_1}(s - it)^{\alpha_2}.
\end{align*}
We then plot the stable manifold of \eqref{E:exist} on the ball of radius one in parameter space, $B_1(0) \subset \R^2$, yielding a cloud of points $(x_1, \dots, x_4) \in \R^4$. 
However for plotting purposes  we need to restrict to a 3 dimensional projection. Furthermore, we make a change of variables into eigen-coordinates to better demonstrate the curvature of the manifold. That is, we transform into the basis  
\begin{align*}
e^\rs_{\mathrm{Re}} &= \mathrm{Re}( \hat V^\rs_1),& e^\rs_{\mathrm{Im}} &= \mathrm{Im}( \hat V^\rs_1),& e^\ru_{\mathrm{Re}}& = \mathrm{Re}( \hat V^\ru_1), &e^\ru_{\mathrm{Im}}& = \mathrm{Im}( \hat V^\ru_1)
\end{align*}
and then we have plotted just the first three components.  

\begin{figure}[H]
\centering \includegraphics[width=1\textwidth]{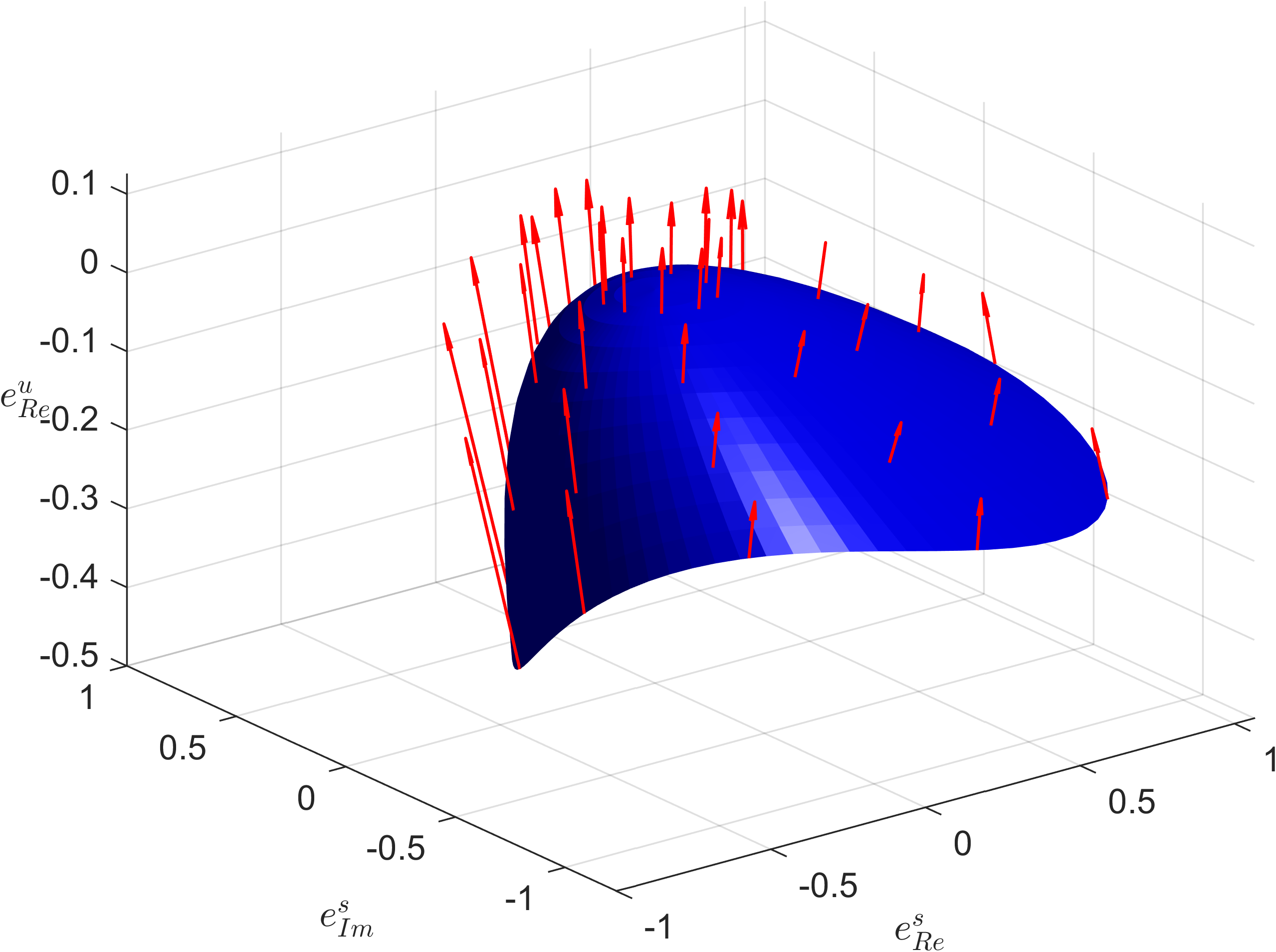}
\caption{Stable manifold with unstable bundle plotted in eigencoordinates.} \label{fig:Manifold_Bundle}
\end{figure}


\subsection{Example computation of the (un)stable bundles over the stable manifold}

After computing the stable manifold, we compute the attached stable and unstable bundles using the methodology described in Section \ref{Ch:res bundles}, again using Taylor series of order $N = 35$. Described in Table \ref{table:BundleBounds} are the bounds \changes{relevant} to the hypothesis of Corollary \ref{prop: SH radii poly for bundle}. Overall we obtain bounds of $Y_0 =  1.35 \cdot 10^{-9}$ and $ Z=0.752$, and thus we are able to achieve a computer-assisted proof that validates the hypothesis of Corollary \ref{prop: SH radii poly for bundle} using a radius of $r_0 =5.45 \cdot 10^{-9}$. 

\begin{table}[H]
	\centering
	\begin{tabular}{|r|l|l|l|l|l|l|} 
		\hline
		&
		$Y_0^{j,a}$  & $Y_0^{j,b}$  & $Y_0^{j,c}$ & $Z^{j,a}$  & $Z^{j,b}$  & $Z^{j,c}$  \\
		\hline 
		stable bundle
		&
	$	1.93 \cdot 10^{-11} $& 
$		4.14 \cdot 10^{-14} $& 
$		0 $&
$		0.677 $&
$		1.92 \cdot 10^{-15}$&
$		0$ \\
			\hline 
		unstable bundle
		&
		$	1.35 \cdot 10^{-9} $& 
		$	4.25 \cdot 10^{-14} $& 
		$	7.08 \cdot 10^{-12} $&
		$	0.697 $&
		$	2.03 \cdot 10^{-15}$&
		$	5.59 \cdot 10^{-2}$ \\
		\hline
	\end{tabular}
	\caption{Bounds for bundles.}
	\label{table:BundleBounds}
\end{table}

As perhaps expected, our estimates on the bundles are less accurate than those on the manifold, and the bounds on the unstable bundle are less accurate than those on the stable bundle. 
There are two bounds that prove especially difficult.  First, as with the parameterization of the manifolds, is the $Z$ bound, and in particular $ Z^a$. Like before, the reason for this is that the constants $K_N^j$  which are  supposed to control $Z$ are in practice only moderately small. For the parameters used in our computer-assisted proof, this turned out to be  $\max_{j} K_N^j = 0.131$.

The second obstacle is a somewhat disappointing $Y_0$ bound, which results from using interval arithmetic. Indeed, interval arithmetic was used to enclose the true coefficients at each finite order $ |\alpha| \leq N$. This was not a problem for computing the manifolds. But when computing the bundles, as the order of the bundle coefficients grew, so too did the size of these error bounds. For the coefficients of order $n \geq 23$, the error in the coefficients is comparable to their norm. This accumulation of rounding error is a well known phenomenon when using interval arithmetic. It is possible that an alternative validation approach which computes the approximate finite coefficients of the manifold and bundles, and then performs a contraction map in the entire space of coefficients, may prove even more efficacious. Nevertheless, even with our current implementation, we have developed a novel method for computing high order parameterizations of resonant vector bundles which are highly accurate over a large region of phase space.


\bibliographystyle{alpha}
\bibliography{paper2}

\end{document}